\documentclass[10pt]{article} 
\usepackage{geometry} 
\usepackage[utf8]{inputenc} 
\usepackage[T1]{fontenc} 
\usepackage[english]{babel} 
\usepackage{amsmath,amsthm} 
\usepackage{amssymb,yfonts}
\usepackage{nicefrac} 
\usepackage{hyperref}
\usepackage{graphicx}
\usepackage{braket}
\usepackage{subfig}
\usepackage{tikz}
\usepackage{color}
\usepackage[usenames,dvipsnames]{xcolor}
\usepackage{csquotes}
\usepackage{dsfont}
\usepackage[sc]{mathpazo}
\usepackage{comment}

\numberwithin{equation}{section}
\newtheorem{theorem}{Theorem}[section]
\newtheorem{proposition}[theorem]{Proposition}

\newtheorem{lemma}[theorem]{Lemma}
\newtheorem{definition}{Definition}

\newtheorem{remark}[theorem]{Remark}
\numberwithin{equation}{section}

\title{Asymptotic Analysis of Laplacian  Operator in Thin Domains on  the Sphere with Highly Oscillatory Boundary}
\date{\today}
\author{ Naísa C. Garcia \footnote{N. C. Garcia was supported by Centro de Matemática, Computação e Cognição - Universidade Federal do ABC - Santo André/SP, email: naisa.camila@ufabc.edu.br. }, \quad Raquel Lehrer\footnote{R. Lehrer was supported by Centro de Ciências Exatas e Tecnológicas - Universidade Estadual do Oeste do Paraná - UNIOESTE, Cascavel/Brazil, email: raquel.lehrer@unioeste.br.}, \quad 
	Marcus A. M. Marrocos \footnote{M. A. M. Marrocos  was supported by Departamento de Matemática - Universidade Federal do Amazonas, Manaus/Brazil, email: marcusmarrocos@ufam.edu.br.}}

\begin{document}
	
	\maketitle
	
	\begin{abstract}
	In this work we analyse the convergence of solutions of the Poisson equation with Neumann boundary conditions in a thin domain with highly oscillatory behavior $\mathcal{U}^\varepsilon$ contained  in the sphere $\mathbb{S}^2$. Using the Multiple Scales method, we obtain the homogenized limit problem and analyse the convergence of solutions, as $\varepsilon$ tends to $0$. Introducing appropriate correctors, we show strong convergence and give error estimates.
		\smallskip
		
		\smallskip
		
	\end{abstract} 
	
		\section{Introduction}

In the study of nanoscale mechanics, various physical phenomena are modeled by partial differential equation within thin domains. These include the mechanics of slender structure as thin rods, plates and shells; fluid dynamics in thin channels, as in lubrication models and blood flow; and chemical diffusion processes on membranes or narrow strips, such as catalytic reactions. In this context, the study of phenomena modeled in  thin domains with oscillating boundaries has drawn considerable interest  from the  research community. In these areas, real life problems often involve highly irregular boundaries, which significantly influence the behavior of the underlying  models. As  well known,  the mathematical tools used to analyze these phenomena involve, in general, PartialDifferential Equations and Differential Geometry.  For instance,  we emphasize the importance of the works  \cite{ACPS11}, \cite{ANVP25},  \cite{PS15},  and the references therein, which investigated elliptic and parabolic problems under homogeneous Neumann boundary conditions, exploring limit problem and convergence properties, focusing on thin domains in $\mathbb{R}^2$. 

Recently, Arrieta, Nakasato, and Villanueva-Pesqueira~\cite{ANVP25}, studied homogenization in 3D thin domains with oscillating boundaries but using a different technique from the one we will present here. In \cite{BP16,Gri08,HR92, LR10,PS15}, the authors considered only purely thin domains in different contexts. Grisier  provides a very nice survey in \cite{Gri08} that discusses various  fields where  analysing the behavior of solutions to  PDE in thin domains is relevant, such as Mathematical Physics, Spectral Geometry and Global Analysis.  In \cite{ACPS11}, Arrieta et al. analyzed the asymptotic behavior of PDE solutions in a thin domain with highly oscillatory behavior at its boundary. We study the same elliptical problem presented in \cite{ACPS11}, however, considering oscillating domains in the sphere. What distinguishes our work from theirs is that, after a change os variables, we deal with an elliptic problem with nonconstant coefficients  that is symmetric with respect to a weighted measure.

We  examine the behavior of solutions to the problem given by the Laplace-Beltrami with Neumann boundary condition defined on  domains in the  sphere that degenerate towards the equator. More specifically, let $\mathbb{S}^2$ denote the sphere equipped with the canonical metric $G$ induced by  the inner product in $\mathbb{R}^3$ and let $\mathcal{U}^\varepsilon\subset \mathbb{S}^2$ be a strip containing the equator with
a highly oscillatory behavior in its boundary that degenerates towards  the equator as $\varepsilon$ tend to  zero. We will address the following problem: consider for a suitable $f^\varepsilon$ 

	\begin{equation}\label{probu}
	\left\{\begin{array}{l}
		- \Delta v +  v= f^\varepsilon \ \ \textit{in }\ \mathcal{U}^{\varepsilon}; \\
		\frac{\partial v}{\partial \eta}  = 0\ \ \textit{on}\ \ \partial \mathcal{U}^{\varepsilon}
	\end{array}\right.
\end{equation}
where 
$$\Delta v=\frac{1}{\sqrt{|G|}}\partial_i(\sqrt{|G|}G^{ij}\partial_j v),$$  
$|G|$ is the  determinant of the round metric $G$ on $\mathbb{S}^2$  induced by the canonical metric on $\mathbb{R}^3$ and   $\eta$ is the unit outward normal vector field  to $\partial\mathcal{U}^\varepsilon$. 
Specifically, we define $\mathcal{U}^\varepsilon$  in the following way: Let $g:[0,2\pi]\rightarrow\mathbb{R}$ be a  $L-$periodic $C^1(0,2\pi)$ function with $L=2\pi/a$, $a\in\mathbb{N}^*$ and consider

\begin{equation}\label{R}
	R^{\varepsilon} =\left\{ (\varphi,{\theta}) \in \mathbb{R}^2; 0 < \varphi < 2 \pi \ \  \textrm{and} \ \  0 < {\theta}< \varepsilon  {g}  \left(\frac{ \varphi}{\varepsilon}\right)       \right\},
\end{equation}
 the spherical strip $\mathcal{U}^\varepsilon$ is defined as $\mathcal{U}^\varepsilon=\left\lbrace \chi(\varphi,\theta); (\varphi, \theta) \in R^\varepsilon \right\rbrace \subset\mathbb{S}^2$, where $\chi$ is the parametrization 	
\begin{eqnarray*}
	\chi: R^\varepsilon \subset \mathbb{R}^2 &\rightarrow &  \mathbb{R}^3 \\
	( \varphi,\theta) &\rightarrow &  \chi(\varphi, \theta)=( \cos \theta \cos \varphi ,  \cos \theta  \sin \varphi , \sin \theta ).
\end{eqnarray*} 

	We observe  that $\displaystyle\theta \leq g_1$, where
\begin{equation}\label{g1}
	g_1 := \displaystyle\max_{\varphi \in [0, 2\pi)} \lbrace g(\varphi) \rbrace< \frac{\pi}{2}.
\end{equation}

	Applying the change of variables $\chi$ on the problem \eqref{probu} we obtain
\begin{equation}\label{eqR}
	\left\{\begin{array}{l}
		- \left( \frac{1}{\cos {\theta}} \frac{\partial}{\partial {\theta}} \left(\cos {\theta} \frac{\partial}{\partial {\theta}} \right)v + \frac{1}{\cos^2 {\theta}} \frac{\partial^2}{\partial \varphi^2} v\right) + v  = f^\varepsilon \ \ \textit{in}\ R^{\varepsilon}; \\
		\frac{1}{\cos^2 {\theta} } \frac{\partial v}{\partial \varphi} N_1+\frac{\partial v}{\partial {\theta}} N_2= 0 \ \ \textit{on} \ \ \partial R^{\varepsilon}; \\
		v(\cdot, \theta)  \   \ \ 2\pi-\textrm{periodic} ,\ \ 
		
	\end{array}\right.
\end{equation}
where  $N = (N _1, N _2)$ is the normal vector field on the upper and lower boundary given by   $N =\left( - \frac{g'\left(\frac{ \varphi}{\varepsilon}\right)}{\sqrt{1+g'\left(\frac{ \varphi}{\varepsilon}\right)^2}}, \frac{1}{\sqrt{1+g'\left(\frac{ \varphi}{\varepsilon}\right)^2}}\right) $  and  $N = \left( 0 , -1 \right)$, respectively. To complete the boundary conditions, we will consider  $2\pi-$ periodic functions in the variable $\varphi$ in $R^{\varepsilon}$.

Let $C^{\infty}_{per}({R}^{\varepsilon})$ denote the  subset of $C^{\infty} ({R}^{\varepsilon})$ consisting of functions that are $2 \pi$-periodic  in the first variable. We then  consider the wheigted Sobolev space  $W^{1,2}_{per}(R^{\varepsilon}, cos\theta) = H^{1}_{per}(R^{\varepsilon},\cos \theta)$,  the closure of $C^{\infty}_{per}({R}^{\varepsilon})$ in the $H^{1}{(R^\varepsilon,\cos\theta)}$ norm
$$
\Vert v \Vert_{H^{1}_{per} (R^{\varepsilon}, \cos  \theta )} = \left(  \int_{R^{\varepsilon}} \vert  v  \vert^2 \cos   \theta  \  d \varphi  d \theta \right) ^{\frac{1}{2}} + \left(\int_{R^{\varepsilon}} \vert \nabla v  \vert^2 \cos  \theta  \ d \varphi  d \theta \right) ^{\frac{1}{2}}< \infty. 
$$

By the Lax-Milgram Theorem, problem \eqref{eqR} has  a unique solution, which we shall denote by $w^\varepsilon$, whenever $f^\varepsilon$ is  $2\pi-$periodic in the variable    $\varphi$ and $f^\varepsilon \in L^2(R^{\varepsilon}, cos \theta)$.

It is worth noting that the family of oscillating domains $\mathcal{U}^\varepsilon$ can be seen as converging to the equator line of the sphere and, in the limit, the equation reduces to an elliptic problem on $\mathbb{S}^1$.

Next, we perform a simple change of variables that involves stretching  the domain $R^\varepsilon$ in the $\theta$ direction by a factor of $\frac{1}{\varepsilon}$ (that is, $\varphi_1 = \varphi$, $\theta_1 = \frac{{\theta}}{\varepsilon}$), transforming the domain $R^{\varepsilon}$ into the domain
$$
\Omega^{\varepsilon} = \left\lbrace (\varphi_1,\theta_1) \in \mathbb{R}^2; \varphi_1 \in (0, 2\pi) \ \ \textrm{and} \ \ 0 < \theta_1 < {g} \left(\frac{ \varphi_1}{\varepsilon}\right) \right\rbrace.
$$
Recalling $(\ref{g1})$, we also define the domain $\Omega$ as
$$\Omega = \left\lbrace (\varphi_1, \theta_1) \in \mathbb{R}^2;  \varphi_1 \in (0, 2\pi) \ \ \textrm{and} \ \  0 < \theta_1 <  g_1, \right\rbrace
$$
and clearly $\Omega^\varepsilon \subset \Omega$.  By doing this, we obtain a domain  $\Omega^\varepsilon$ that is no longer thin, although it exhibits highly oscillatory behavior. Under this change of variables, the solutions $w^\varepsilon(\varphi,\theta)$ will be now denoted by $u^\varepsilon = u^\varepsilon(\varphi_1,\theta_1)$ and problem \eqref{eqR} is transformed into

\begin{equation}\label{eqO}
	\left\{\begin{array}{l}
		- \left( \frac{1}{ \varepsilon^2 \cos (\varepsilon \theta_1 )} \frac{\partial}{\partial \theta_1} \left(\cos (\varepsilon \theta_1 ) \frac{\partial}{\partial \theta_1} \right)u^{\varepsilon} + \frac{1}{ \cos^2 (\varepsilon \theta_1 )} \frac{\partial^2}{\partial \varphi_1^2} u^{\varepsilon}\right) + u^{\varepsilon}  = f^\varepsilon \ \ \textit{in}\ \Omega^{\varepsilon}; \\
		\frac{1}{ \cos^2 (\varepsilon \theta_1 )} \frac{\partial}{\partial \varphi_1} u^{\varepsilon}   N_1^\varepsilon + \frac{1}{ \varepsilon^2} \frac{\partial }{\partial \theta_1} u^{\varepsilon}N_2^\varepsilon  =0 \ \ \textit{on} \ \ \partial \Omega^{\varepsilon};\\
		u^{\varepsilon}(\cdot, \theta_{1}) \ \ \ 2\pi-\textrm{periodic},
	\end{array}\right.
\end{equation}
where $N^\varepsilon= (N^\varepsilon_1, N^\varepsilon_2)$ is the normal vector fields on the upper and lower boundary given by $N^\varepsilon = \left( - \frac{g'\left(\frac{ \varphi_1}{\varepsilon}\right)}{\sqrt{1+g'\left(\frac{ \varphi_1}{\varepsilon}\right)^2}}, \frac{1}{\sqrt{1+g'\left(\frac{ \varphi_1}{\varepsilon}\right)^2}}\right) $  and  $N^\varepsilon = \left( 0 , -1 \right)$, respectively. Here, the appropriate Sobolev space is $H^1_{per}(\Omega^{\varepsilon}, \cos(\varepsilon \theta_1))$.  The factor $\frac{1}{\varepsilon^2}$ in from of the derivative in the $\theta$-direction means a very fast diffusion in the meridian geodesic direction on the sphere $\mathbb{S}^2$.

We observe that the domain $\Omega^\varepsilon$ " converges"  to the domain $\Omega$. Therefore, we expect that the solutions $u^\varepsilon$ of  problem $(\ref{eqO})$ will converge, in a certain sense, to a  solution of a problem  defined in $\Omega$.

Not only the solutions $u^\varepsilon$ depend on $\varepsilon$, but the domain $\Omega^\varepsilon$ as well. Therefore,  the first step to tackle the convergence problem, we need to work on a fixed domain $\Omega$. To this end,  we will need some extension operator $P_\varepsilon$ that will transforme a function defined in $\Omega^\varepsilon$ into a function defined in $\Omega$ (see Lemma \ref{lema}). With this extension operator, we show that the solutions for problem $(\ref{eqO})$  satisfy 
 $$P_\varepsilon u_\varepsilon \rightarrow  w_0 \ \text{weakly in}  \ H^1(\Omega),$$
 where   $w_0$ is the unique solution of the \textit{homogenized limit problem}, which was obtained  by the use of   the Multiple Scale Method:

\begin{equation*}
	\left\{\begin{array}{ll}
		-q_0  \frac{d^2}{d \varphi^2} w_0(\varphi)   + w_0(\varphi)= f(\varphi) \  \ \varphi \in (0, 2 \pi)  ; \\
		w_0 \ \ \ \textrm{periodic},
	\end{array}\right.
\end{equation*} 
where $q_0$ is the \textit{homogenization coefficient} and $f$ is an appropriate limit of the sequence $f^\varepsilon$. For more details, see Section 2.

We will use Tartar's approach, adapted to thin domains in the sphere, to prove the convergence results. This result is presented in Theorem \ref{main}. 
However, since this result typically yields only weak convergence, in a particular sense, it does not capture finer details of the solution, especially near the singular structures or boundaries of the domain. To enhance our analysis and obtain strong convergence, we introduce correctors based on formal asymptotic expansions derived from  the multiple scale method, based on \cite{BLP78}. Specifically, we introduce first and second order correctors to obtain error estimates for the approximation of the exact solution $w^{\varepsilon}$ to derive strong convergence and error estimates in $H^1-$norm.  We point out that once we add the correctors, we can deal with the convergence problem directly on  problem $(\ref{eqR})$. These correctors allow us to rigorously quantify the approximation error and demonstrate that the homogenized model, supplemented with suitable corrector terms, provides a precise description of the behavior of $w^{\varepsilon}$ as $\varepsilon \to 0$. See Theorems \ref{teo1corretor} and \ref{teo erro} .

Moreover, intrinsically on the sphere $\mathbb{S}^2$, our limiting differential operator is essentially  the Laplace-Beltrami operator on the equator $\mathbb{S}^1$. We note  that the oscillation of the boundary of $\mathcal{U}^\varepsilon$ implies a very rapid diffusion along the  direction of the geodesic orthogonal to the equator. This effect becomes clearer in equation \eqref{eqO}. This feature is consistent with the results obtained for highly oscillating domains in $\mathbb{R}^n$.

	\section{The multiple scales method and the homogenized equation}

	In this section,  we apply the method of multiple scales to obtain formally the limit homogenized problem of \eqref{eqR},  see Section 2.2, \cite{BLP78}  for more details.
	
	Regarding the right side of the equation \eqref{eqR}, consider
$f^\varepsilon(\cdot, \theta) \in L^2(R^{\varepsilon}, cos \theta)$,   $2\pi-$ periodic and  satisfying 
		$$
		\int_{R^{\varepsilon}}  (f^\varepsilon)^2   \cos\theta \ d \varphi \ d \theta  \leq C,
		$$
		for some constant $C>0$ independent of $\varepsilon$. 	Moreover, for simplicity, we assume that the nonhomogeneous term $f^\varepsilon = f$ satisfies $f(\varphi, {\theta})=f(\varphi)$.

	In order to introduce the multiple scale method, we consider the domain 
	\begin{equation}\label{Y*}
		Y^*=\{(y,z)\in\mathbb{R}^2:\ 0<y <L, \ 0<z<g(y)\}
		\end{equation}
called the \textit{basic cell}, and functions $w_i\left( x, r, y, z \right)$, $i\in\mathbb{N} $, defined on $\mathbb{R}^2\times Y^*$,  which are $2\pi$ -periodic in the variable $x$ and 
	$L$-periodic in the variable $y$. We look for a formal asymptotic expansion of the solution to problem  \eqref{eqR} of the form
	$$
	w^{\varepsilon}( \varphi,\theta) = w_0 \left( \varphi, \theta, \frac{\varphi}{\varepsilon}, \frac{\theta}{\varepsilon} \right) + \varepsilon w_1 \left( \varphi, \theta, \frac{\varphi}{\varepsilon}, \frac{\theta}{\varepsilon} \right) + \varepsilon^2 w_2 \left( \varphi, \theta, \frac{\varphi}{\varepsilon}, \frac{\theta}{\varepsilon} \right) + \cdots  
	$$
	
	Since $R^{\varepsilon}$ degenerates along the $\varphi-$axis  as  $\varepsilon$ tends to  $0$, this suggests that $w^{\varepsilon}$ tends to become independent of the variable $ \theta$. Consequently, we consider  $w_i's$ as  functions independent of   $ \theta$, thus assuming that
	\begin{equation}\label{exp}
		w^{\varepsilon}( \varphi, \theta) = w_0 \left(\varphi,   \frac{\varphi}{\varepsilon}, \frac{\theta}{\varepsilon}\right) + \varepsilon w_1 \left(\varphi,   \frac{\varphi}{\varepsilon}, \frac{\theta}{\varepsilon}\right) + \varepsilon^2 w_2 \left(\varphi,   \frac{\varphi}{\varepsilon}, \frac{\theta}{\varepsilon}\right) + \cdots  
	\end{equation}
	where for each $i \in \mathbb{N} $,  $w_i : (0, 2\pi )\times Y^* \longrightarrow \mathbb{R}.$
	
	The variables $\varphi$ and $\theta$  refer  to the ``macroscopic" position, while the variables $\varphi/\varepsilon$ and $\theta/\varepsilon$ represent the ``microscopic" geometry of the domain. This ``microscopic structure" is represent by the \textit{basic cell} $Y^*$.
	
	The ideia of the method is to substitute  expansion \eqref{exp} into problem \eqref{eqR}. We begin by analyzing  the problem in the interior, and subsequently  on the boundary. 
	
	Let $\mathcal{L}$ denote the interior operator for  problem \eqref{eqR}, that is 
	\begin{equation}\label{operador}
	\mathcal{L} w  :=  \frac{1}{\cos \theta} \frac{\partial }{\partial \theta} \left( \cos \theta \frac{\partial }{\partial \theta} \right) w  + \frac{1}{\cos^2 \theta} \frac{\partial^2 }{\partial \varphi^2}w.  
	\end{equation}
	
	Thus, 
	$$
	\mathcal{L} (w)  := \sum_{j=1}^{+ \infty} \varepsilon^j   \mathcal{L} \left( w_j    \right).
	$$
	Consider $x= \varphi$, $y=\frac{\varphi}{\varepsilon}$ and $z=\frac{\theta}{\varepsilon}$. Under these changes of variable, we obtain:
	\begin{equation}\label{10}
		\frac{\partial}{\partial \varphi} = \frac{\partial}{\partial x} + \frac{1}{\varepsilon}\frac{\partial}{\partial y},  \ \ \ \ \  \ \ \ \   \frac{\partial}{\partial \theta} =  \frac{1}{\varepsilon}\frac{\partial}{\partial z},
	\end{equation}
	$$
	\frac{\partial^2}{\partial \varphi^2} = \frac{\partial^2}{\partial x^2} + \frac{2}{\varepsilon}\frac{\partial}{\partial x}\frac{\partial}{\partial y}+\frac{1}{\varepsilon^2}\frac{\partial^2}{\partial y^2},  \ \ \ \ \  \ \ \ \   \frac{\partial^2}{\partial \theta^2} =  \frac{1}{\varepsilon^2}\frac{\partial^2}{\partial z^2},
	$$
	and the  operator  $\mathcal{L}$ becomes:
	\begin{equation*}
		\begin{aligned}
			\mathcal{L} w_j :=&  \frac{1}{\cos \theta} \frac{\partial }{\partial \theta} \left( \cos \theta \frac{\partial }{\partial \theta} \right) w_j + \frac{1}{\cos^2 \theta} \frac{\partial^2 }{\partial \varphi^2}w_j   \\
			=&\frac{1}{\varepsilon^2}\left( \frac{\partial^2}{\partial y^2} + \frac{\partial^2}{\partial z^2}  \right)  w_j  + \frac{1}{\varepsilon}\left( 
			- \tan (\varepsilon z) \frac{\partial}{\partial z}+ 2 \frac{\partial}{\partial x}\frac{\partial}{\partial y} \right) w_j   
			+ \frac{1}{\cos^2  (\varepsilon z)} \left( \frac{\partial^2}{\partial x^2}  \right)  w_j.
		\end{aligned}
	\end{equation*}
	
	Replacing  $w^\varepsilon$ with its formal expansion given by \eqref{exp} into  equation \eqref{eqR}, we obtain the following expansion on the variables $x,y,z$: 
	
	\begin{eqnarray*}
		f&=&      \sum_{j=0}^{\infty}  \varepsilon^j \left(- \mathcal{L} w_j  +  w_j \right)
		=   -\frac{1}{\varepsilon^2} \left(  \frac{\partial^2}{\partial y^2} w_0 + \frac{\partial^2}{\partial z^2} w_0  \right) -\frac{1}{\varepsilon}\left(   \frac{\partial^2}{\partial y^2} w_1+ \frac{\partial^2}{\partial z^2}w_1+   2 \frac{\partial}{\partial x}\frac{\partial}{\partial y} w_0  \right) \\
		 &  -& \frac{\partial^2}{\partial y^2}w_2-\frac{\partial^2}{\partial z^2} w_2  - \frac{\partial^2}{\partial x^2}w_0 + z \frac{\partial w_0}{\partial z}- z^2  \frac{\partial^2}{\partial y^2} w_0  - 2  \frac{\partial}{\partial x}\frac{\partial}{\partial y} w_1 +w_0 +\cdots  
\end{eqnarray*}

Note that in the expansion above we can discard the terms with positive powers of $\varepsilon$ since we are going to take the limit when $\varepsilon \to 0$. Therefore, to obtain a necessary condition for ensuring convergence as $\varepsilon \to 0$, while temporarily disregarding the boundary conditions, we need to  solve the problems below 
	\begin{equation}\label{probsemfronteira}
		\left\{\begin{array}{ll}
			\frac{\partial^2}{\partial y^2} w_0 + \frac{\partial^2}{\partial z^2} w_0  = 0   \  \  \textrm{in} \  \ Y^{*};  \\
			\frac{\partial^2}{\partial y^2}w_1+\frac{\partial^2}{\partial z^2} w_1 + 2 \frac{\partial}{\partial x}\frac{\partial}{\partial y} w_0 = 0 \  \  \textrm{in} \  \ Y^{*}; \\ 
			- \frac{\partial^2}{\partial y^2}w_2-\frac{\partial^2}{\partial z^2} w_2 -\frac{\partial^2}{\partial x^2}w_0 - 2 \frac{\partial}{\partial x}\frac{\partial}{\partial y} w_1 +w_0 + z \frac{\partial}{\partial z}w_0  - z^2\frac{\partial^2}{\partial y^2} w_0 = f \  \  \textrm{in} \  \ Y^{*}. 
		\end{array}\right.
	\end{equation}
	
	Now  we analyze the boundary conditions. If $N(\varphi,\theta)$ and $\eta(\frac{\varphi}{\varepsilon}, \frac{\theta}{\varepsilon}) = \eta(y,z)$ denote the unit outward normal vectors to  $\partial R^{\varepsilon}$ and $\partial Y^*$, respectively, then 
	$
	N(\varphi,\theta)= \eta\left(\frac{\varphi}{\varepsilon}, \frac{\theta}{\varepsilon} \right) = \eta(y,z),
	$
	provided that $\theta = 0$ or $\theta = \varepsilon g(\frac{\varphi}{\varepsilon})$. If $\theta = 0$, then $N(\varphi,\theta)= \eta\left(\frac{\varphi}{\varepsilon}, \frac{\theta}{\varepsilon} \right) = \eta(y,z) = (0,-1).$ If  $\theta = \varepsilon g(\frac{\varphi}{\varepsilon})$, then 
	$$N(\varphi,\theta)=  \left( - \frac{g'\left(\frac{ \varphi}{\varepsilon}\right)}{\sqrt{1+g'\left(\frac{ \varphi}{\varepsilon}\right)^2}}, \frac{1}{\sqrt{1+g'\left(\frac{ \varphi}{\varepsilon}\right)^2}}\right) = \eta\left(\frac{\varphi}{\varepsilon}, \frac{\theta}{\varepsilon} \right) = \eta(y,z),
	$$
	since the upper part of $\partial R^{\varepsilon}$ is given by $(\varphi, \varepsilon g (\frac{\varphi}{\varepsilon}) )$, with tangent vector $(1, g' (\frac{\varphi}{\varepsilon}))$, and the upper boundary of  $\partial Y^*$ is $(\varphi, \varepsilon g ( \varphi ) )$ with tangent vector $(1, g' ( \varphi ))$. Hence, we shall denote both normal vectors by $N$ for simplicity.

	We have the following condition on $\partial R^{\varepsilon}$:
	$$
	0= \left( \frac{1}{\cos^2 \theta}\frac{\partial w^\varepsilon}{\partial\varphi} , \frac{\partial w^\varepsilon}{\partial\theta}  \right)\cdot ( N_1,  N_2).
	$$
	Using the formal expansion  (\ref{exp}) and performing the change of variables, we obtain:
	\begin{align*}
		0=& \left( \frac{1}{\cos^2 \theta}\frac{\partial w^\varepsilon}{\partial\varphi} , \frac{\partial w^\varepsilon}{\partial\theta}  \right)\cdot ( N_1,  N_2) 
		=   \frac{1}{\cos^2 \theta}\frac{\partial w^\varepsilon}{\partial\varphi} \cdot   N_1 + \frac{\partial w^\varepsilon}{\partial\theta}   \cdot    N_2  \\
		=& \frac{1}{\varepsilon} \left[  \frac{\partial w_0}{\partial y} N_1 + \frac{\partial w_0}{\partial z} N_2\right] + \left[   \left(  \frac{\partial w_0}{\partial x} +  \frac{\partial w_1}{\partial y}
		\right)N_1 + \frac{\partial w_1}{\partial z} N_2 \right]  \\
		 + &
		\varepsilon \left[ \left( \frac{1}{\cos^2 (\varepsilon z)}\frac{\partial w_1}{\partial x} + \frac{1}{\cos^2 (\varepsilon z)}\frac{\partial w_2}{\partial y}
		\right)N_1 + \frac{\partial w_2}{\partial z} N_2 \right] + \cdots 
	\end{align*}
	
%

 Now, we set  $\partial Y^* =  B_1\cup B_2 \cup B_3 \cup B_4$, where    $B_1$ is the upper boundary,  $B_2$  is the lower boundary, $B_3$  is the left lateral boundary and $B_4$ is  the right lateral boundary of  $Y^*$. 
Thus,  problem (\ref{probsemfronteira}), together with the corresponding boundary  conditions derived above,  becomes:

	\begin{equation}\label{1}
		\left\{\begin{array}{l}
			\frac{\partial^2}{\partial y^2} w_0 + \frac{\partial^2}{\partial z^2} w_0  = 0 \  \ \textrm{in} \ \ Y^*; \\
			\frac{\partial w_0}{\partial y} N_1 + \frac{\partial w_0}{\partial z} N_2  =0  \  \ \textrm{on} \ \ B_1   \cup B_2; \\
			w_0(x, \cdot, z) \  \ L-\textrm{periodic}.
		\end{array}\right.
	\end{equation}
	
	\begin{equation}\label{2}
		\left\{\begin{array}{l}
			\frac{\partial^2}{\partial y^2}w_1+\frac{\partial^2}{\partial z^2} w_1   = -2 \frac{\partial}{\partial x}\frac{\partial}{\partial y} w_0 \  \ \textrm{in} \ \ Y^*; \\
			\frac{\partial w_1}{\partial y}
			N_1 + \frac{\partial w_1}{\partial z} N_2   =- \frac{\partial w_0}{\partial x}  N_1 \  \ \textrm{on} \ \  B_1   \cup B_2;  \\
			w_1(x, \cdot, z) \  \ L-\textrm{periodic}.
		\end{array}\right.
	\end{equation}
	
	\begin{equation}\label{3}
		\left\{\begin{array}{l}
			-  \frac{\partial^2}{\partial y^2}w_2-\frac{\partial^2}{\partial z^2} w_2   = f-w_0 +2 \frac{\partial}{\partial x}\frac{\partial}{\partial y} w_1- \frac{\partial^2}{\partial x^2}w_0- z\frac{\partial}{\partial z}w_0  +
			z^2\frac{\partial^2}{\partial y^2} w_0 \  \ \textrm{in} \ \ Y^*;  \\
			\frac{\partial w_2}{\partial y}
			N_1 + \frac{\partial w_2}{\partial z} N_2   = -\frac{\partial w_1}{\partial x} N_1 \  \ \textrm{on} \ \ B_1   \cup B_2; \\
			w_2(x, \cdot, z) \  \ L-\textrm{periodic}.
		\end{array}\right.
	\end{equation}
	
	From \eqref{1}, it follows that $w_0(x,y,z)=w_0(x)$. Substituting this into \eqref{2}, we obtain:
	\begin{equation*}
		\left\{\begin{array}{l}
			
			\frac{\partial^2}{\partial y^2}w_1+\frac{\partial^2}{\partial z^2} w_1   = 0 \  \ \textrm{in} \ \ Y^*; \\
			\frac{\partial w_1}{\partial y}
			N_1 + \frac{\partial w_1}{\partial z} N_2   = \frac{g'(y)}{\sqrt{1+g'(y)^2}} \frac{d w_0}{d x}  \  \ \textrm{on} \ \  B_1;   \\
			\frac{\partial w_1}{\partial y}
			N_1 + \frac{\partial w_1}{\partial z} N_2  = 0 \  \ \textrm{on} \ \  B_2;  \\
			w_1(x, \cdot, z) \  \ L-\textrm{periodic}.
		\end{array}\right.
	\end{equation*}
	
	Let $X(y,z)$ be the solution of
	\begin{equation}\label{14}
		\left\{\begin{array}{l}
			\Delta_{y,z}X(y,z)  = 0 \  \ \textrm{in} \ \ Y^*; \\
			\frac{\partial X}{\partial N} (y, g(y))   = -\frac{g'(y)}{\sqrt{1+g'(y)^2}}  \  \ \textrm{on} \ \  B_1;   \\
			\frac{\partial X}{\partial N} (y, 0)   = 0 \  \ \textrm{on} \ \  B_2;  \\
			X( \cdot, z) \  \ L-\textrm{periodic}.
		\end{array}\right.
	\end{equation}
	
	We have that 
	
	\begin{equation}\label{15}
		w_1(x,y,z)= -X(y,z)\frac{d w_0}{d x}(x).
	\end{equation} Therefore, (\ref{3}) becomes: 
	
	\begin{equation}\label{4}
		\left\{\begin{array}{ l}
			-  \frac{\partial^2}{\partial y^2}w_2-\frac{\partial^2}{\partial z^2} w_2   = f-w_0  +\left(1- 2 \frac{\partial}{\partial y} X \right)  \frac{d^2 w_0}{d x^2} \  \ \textrm{in} \ \ Y^*;  \\
			\frac{\partial w_2}{\partial N}
			=  - \frac{g'(y)}{\sqrt{1+g'(y)^2}} X \frac{d^2 w_0}{d x^2} \  \ \textrm{on} \ \ B_1;  \\
			\frac{\partial w_2}{\partial N}
			=  0 \  \ \textrm{on} \ \ B_2; \\
			w_2(x, \cdot, z) \  \ L-\textrm{periodic}.
		\end{array}\right.
	\end{equation}
	
	Since $w_2$ is solution of (\ref{4}), it follows from  the Divergence Theorem that
	\begin{equation}\label{2.5}
		0 = \int_{Y^*} \left\lbrace f(x) -w_0(x)  +\left(1- 2 \frac{\partial}{\partial y} X(y,z) \right)  \frac{d^2 w_0}{d x^2}(x)  \right\rbrace \phi(x) \ d y dz + \int_{\partial Y^*} \frac{\partial w_2}{\partial N}  \phi(x) \ dS,
	\end{equation}
for all  $2\pi$-periodic  test function $\phi \in C^{\infty}_{per}(\mathbb{R})$,  $\phi(\cdot,y,z) = \phi(x)$. On the other hand, 
	\begin{eqnarray*}
		\int_{Y^*} \phi(x) \frac{d^2 w_0}{d x^2}(x) \frac{\partial}{\partial y} X(y,z) \ d y dz &= &
		\int_{Y^*} \phi(x) \frac{d^2 w_0}{d x^2}(x) \nabla_{y,z} X(y,z) \cdot (1,0) dydz\\
		&= & \int_{\partial Y^*}\phi(x) \frac{d^2 w_0}{d x^2}(x) X N_1 \ dS =  \int_{\partial Y^*}\phi \frac{\partial w_2}{\partial N} \ dS.
	\end{eqnarray*}
	
	Replacing this into (\ref{2.5}), we obtain: 
	$$
\int_{Y^*} \left\lbrace w_0(x)  - \left(1-  \frac{\partial}{\partial y} X(y,z) \right)  \frac{d^2 w_0}{d x^2}(x)  \right\rbrace \phi(x)  \ d y dz = \int_{Y^*}  f(x) \phi(x)   \ d y dz
	$$
	for all $\phi \in C^{\infty}_{per}(\mathbb{R})$.
	Consequently, we have:
	\begin{equation*}
		w_0(x)  \vert Y^* \vert -  \frac{d^2 w_0}{d x^2}(x) \int_{Y^*}  \left(1 -  \frac{\partial}{\partial y} X(y,z) \right)     \ d y dz     =  \vert Y^* \vert  f(x).  
	\end{equation*}
	Let
	\begin{equation}\label{q_0}
		q_0= \frac{1}{\vert Y^* \vert}\int_{Y^*}  \left(1 -  \frac{\partial}{\partial y} X(y,z) \right)     \ d y dz, 
	\end{equation}
	then, we have
	$$
	-q_0  \frac{d^2 w_0}{d x^2}(x)   + w_0(x)= f(x),
	$$
	where $w_0$ must satisfy:
	\begin{equation}\label{eqlimitemultiplasescalas}
		\left\{\begin{array}{ll}
			-q_0  \frac{d^2 w_0}{d x^2}(x)   + w_0(x)= f(x) \  \  x \in (0, 2 \pi);  \\
			w_0 (\cdot ) \ \ \ 2\pi-\textrm{periodic}.
		\end{array}\right.
	\end{equation}

	This second-order differential equation in $w_0$ defined on the interval $(0, 2 \pi)$, is called the \textit{homogenized equation} of  problem (\ref{eqR}).

	\section{Convergence to the homogenized limit}

	As we observed  in the Introduction, for each $\varepsilon$ fixed, problem \eqref{eqR} admits a unique solution $w^\varepsilon$. Now, we aim to investigate the behavior of these solutions as $\varepsilon$ tends to $0$.  More precisely, we will prove that these solutions  converge, in an appropriate sense,  to the solution of problem \eqref{eqlimitemultiplasescalas}.  	
	In order to simplify notation we will indicate \textit{strong convergence} by $\rightarrow$, \textit{weak convergence} by $ \rightharpoonup$ and  
	$weak^*  \ convergence$ by $\stackrel{*}{\rightharpoonup} $, always indicating the space were the convergence occours.
	
	Now, we will work again with  the domains $\Omega^\varepsilon, \Omega$ defined before and  with problem $(\ref{eqO})$, whose solutions will be called $u^\varepsilon$.

	We would like to work on a space of functions defined on a fixed domain as $\varepsilon$ goes to $0$, and  there are several possible approaches to achieve this.  
	Following the ideas of \cite{ACPS11}, we construct an extension operator $P_{\varepsilon}$, that transforms the variational formulation of problem \eqref{eqO} into an integral equation defined on a fixed domain. This step is   crucial for establishing the general results presented in the subsequent sections. The construction relies on a  reflection argument in the  $\theta_1$ direction, adapted from  \cite{ACPS11} to align with the oscillating boundary. 

		For the next lemma, we consider the following open sets:
	\begin{align*}
		&\mathcal{O} = \{ (x_1,x_2) \in \mathbb{R}^2; x_1 \in I \ \ \textrm{and} \ \ 0<x_2< G_1 \}, \\ 
		&\mathcal{O}^{\varepsilon} = \{ (x_1,x_2) \in \mathbb{R}^2; x_1 \in I \ \ \textrm{and} \ \ 0<x_2< G_{\varepsilon}(x_1) \},
	\end{align*}
	where $I \subset \mathbb{R} $ is an open interval, $G_0, G_1$ are positive constants, $G_{\varepsilon}: I \mapsto \mathbb{R}$ is a family of $C^1(I)$ $I$-periodic functions that satisfy $0 < G_0 \leq G_{\varepsilon}(x_1) \leq G_1  $ for all $x \in I$ and $\varepsilon > 0$. Notice that  $\mathcal{O}^{\varepsilon} \subset \mathcal{O}$.
	
	\begin{lemma}\label{lema} Let $\sigma_{\varepsilon}=\cos(\varepsilon x_2)$ be the weight defined in $\mathcal{O}^{\varepsilon}$, and $W^{1,p}_{per}(\mathcal{O}^{\varepsilon})$ the set of functions in $W^{1,p}(\mathcal{O}^{\varepsilon})$ which are $I$-periodic in the variable $x_1$, analogously $L_{per}^2(\mathcal{O}^{\varepsilon},cos(\varepsilon x_2))$. Then there exist a continuous  extension operator
		\begin{eqnarray*}
			P_{\varepsilon} \in  \mathbb{L}(L^p(\mathcal{O}^{\varepsilon}, \sigma_\varepsilon),L^p(\mathcal{O}) ) \cap \mathbb{L}(W^{1,p}(\mathcal{O}^{\varepsilon}, \sigma_\varepsilon),W^{1,p}(\mathcal{O} ) )   \cap \mathbb{L}(W^{1,p}_{per}(\mathcal{O}^{\varepsilon}, \sigma_\varepsilon),W^{1,p}_{per}(\mathcal{O}) ), 
		\end{eqnarray*}
	 a  constant  $K $ independent of $\varepsilon$ and $p$ such that
		\begin{align*} 
			&\Vert P_{\varepsilon} \varphi \Vert_{L^p (\mathcal{O})}  \leq K \Vert \varphi\Vert_{L^p (\mathcal{O}^{\varepsilon}, \cos(\varepsilon x_2))}  \\
			& \left\Vert \frac{\partial P_{\varepsilon} \varphi }{\partial x_1} \right\Vert_{L^p (\mathcal{O})}  \leq K  \left\lbrace  \left\Vert \frac{\partial  \varphi }{\partial x_1} \right\Vert_{L^p (\mathcal{O}^{\varepsilon}, \cos(\varepsilon x_2))} + \eta(\varepsilon) \left\Vert \frac{\partial  \varphi }{\partial x_2} \right\Vert_{L^p (\mathcal{O}^{\varepsilon}, \cos(\varepsilon x_2))} \right\rbrace \\
			&\left\Vert \frac{\partial P_{\varepsilon} \varphi }{\partial x_2} \right\Vert_{L^p (\mathcal{O})}  \leq K    \left\Vert \frac{\partial  \varphi }{\partial x_2} \right\Vert_{L^p (\mathcal{O}^{\varepsilon}, \cos(\varepsilon x_2))} 
		\end{align*} 
		
		for all $\varphi \in W^{1,p}(\mathcal{O}^{\varepsilon}, \cos(\varepsilon x_2))$ where $1 \leq p \leq \infty $ and
		$$
		\eta(\varepsilon) = sup_{x \in I} \{ \vert G'_{\varepsilon}(x) \vert \}
		$$
		
	\end{lemma}
	
	\begin{proof} The proof follows from Lemma 3.1 in~\cite{ACPS11}.
	\end{proof}
	
	
		Lemma \ref{lema} shows that the extension operator $P_{\varepsilon}: L_{per}^2(\Omega^{\varepsilon}, cos(\varepsilon\theta_1)) \to L_{per}^2(\Omega ) $ is continuous. Moreover, when restricted  to Sobolev  space $H^1_{per}(\Omega^{\varepsilon},\cos(\varepsilon \theta_1) )$, the extension operator remains continuous, with continuity measured with respect to its corresponding norm. This was the last necessary tools to state our first convergence result:
	
	\begin{theorem}\label{main}
		Let $u^{\varepsilon}\in H^1_{per}(\Omega^{\varepsilon} , \cos (\varepsilon \theta_1))$ be the solution of problem (\ref{eqO}), with $f^\varepsilon \in  L^{2}_{per}(\Omega^{\varepsilon}, \cos(\varepsilon \theta_1))$, and
		$\Vert  f^\varepsilon\Vert_{L^{2}(\Omega^{\varepsilon}, \cos(\varepsilon \theta_1))} \leq C$ with $C$ independent of the parameter $\varepsilon$. 
		Then, if we have a sequence $\varepsilon \rightarrow 0$ such that
		$$
		\hat{f}^{\varepsilon}(\cdot) = \int_0^{g(\varphi_1/\varepsilon)} f^\varepsilon(\cdot, \theta_1) d \theta_1 \rightharpoonup  \hat{f}(\cdot) \ \ \text{in} \ L^2_{per}(0,2 \pi).
		$$
		Then, 
		$$
		P_{\varepsilon}u^{\varepsilon} \rightharpoonup u_0 \ \ \ \text{ in } \ H^1_{per}(\Omega),
		$$
		where $P_{\varepsilon}$ is the extension operator constructed in Lemma~\ref{lema} (With $\mathcal{O}^{\varepsilon} = \Omega^{\varepsilon}$ and $\mathcal{O}= \Omega$ ) and $u_0(\varphi_1, \theta_1)=u_0(\varphi_1)$ for all $(\varphi_1, \theta_1) \in \Omega$ is the unique solution of 
		
		\begin{equation*}
			\left\{\begin{array}{l}
				- q_0 \ \frac{\partial^2  u_0 }{\partial \varphi_1^2} +  u_0  = f_0 \ \ \textit{in}\ (0,2 \pi); 	  \\
				u_{0} \in \  H_{per}^1(0,2 \pi);
			\end{array}\right.
		\end{equation*}
		where $q_0>0$ is given  by  $ \frac{1}{\vert Y^* \vert} \int_{Y^*} \left( 1- \frac{\partial X^0}{\partial y}\right) dy dz$  and $f_0(\cdot) =\displaystyle\frac{L}{\vert Y^* \vert} \hat{f}(\cdot)$.
		
	\end{theorem}

	  We will use Tartar's oscillating test function technique to prove the convergence of  the solutions in the thin domain to the homogenized limit solution, for more details see \cite{CP80}.
	The proof of Theorem \ref{main} will be divided in several steps. We outline them here:
	\begin{enumerate}
		\item Construction of Tartar's test functions $\tau^\varepsilon$
		\item Limit of $\chi^{\varepsilon}$ (it will be defined later).
		\item Limit of functions extended by zero $\frac{\widetilde {\partial u^{\varepsilon}}}{\partial \varphi_1}$, $\frac{\widetilde {\partial  u^{\varepsilon}}}{\partial  \theta_1}$ and $\widetilde{f^\varepsilon}$
		\item Limit of the extended functions $P_{\varepsilon}u^{\varepsilon}$  and $\frac{\partial P_{\varepsilon}u^{\varepsilon} }{\partial \theta_1}$
		\item Limit of $\tau^\varepsilon$
		\item Limit of $\widetilde{\frac{\partial \tau^\varepsilon}{\partial \varphi_1}} $.
		\item Proof of Theorem \ref{main}.
	\end{enumerate}
	
	\subsection{Construction of Tartar's test functions $\tau^\varepsilon$}
	A function  $u^{\varepsilon} \in H^1_{per}(\Omega^{\varepsilon}, \cos(\varepsilon \theta_1))$ is a weak solution for problem \eqref{eqO} if it satisfies

	\begin{equation}\label{4.14}
		\int_{\Omega^{\varepsilon}} \left( \frac{1}{\varepsilon^2} \frac{\partial u^{\varepsilon}}{\partial  \theta_1} \frac{\partial \psi}{\partial  \theta_1}    + \frac{1}{ \cos^2 ( \varepsilon \theta_1) }\frac{\partial u^{\varepsilon}}{\partial \varphi_1} \frac{\partial \psi }{\partial \varphi_1}      +  u^{\varepsilon}  \psi \right) \cos( \varepsilon \theta_1) \ d\varphi_1 d\theta_1  =  \int_{\Omega^{\varepsilon}} f^\varepsilon \psi \cos( \varepsilon \theta_1) \  d\varphi_1 d\theta_1,
	\end{equation}
	for all $ \psi \in H^1_{per}(\Omega^{\varepsilon}, \cos(\varepsilon \theta_1))$.

	By choosing $ \psi = u^{\varepsilon}$ in (\ref{4.14}) we obtain the following a priori estimate for $u^{\varepsilon}$
	
	\begin{equation}\label{4.5}
		\frac{1}{\varepsilon^2  } \left\Vert  \frac{\partial u^{\varepsilon}}{\partial  \theta_1} \right\Vert^2   + \left\Vert \frac{1}{ \cos (\varepsilon \theta_1) } \frac{\partial u^{\varepsilon}}{\partial \varphi_1} \right\Vert^2  + \left\Vert u^{\varepsilon} \right\Vert^2 
		\leq \left\Vert f^\varepsilon \right\Vert  \left\Vert u^{\varepsilon} \right\Vert,
	\end{equation}
	where $\left\Vert  \cdot \right\Vert  = \left\Vert \cdot \right\Vert_{L^2(\Omega^{\varepsilon},  \cos ( \varepsilon  \theta_1))} $.
	
	Tartar’s method consists in constructing a sequence of suitable test functions that enables the passage to the limit in the variational formulation of the problem, effectively removing the dissipation term $\frac{1}{\varepsilon}$. The first step in this approach is the formulation of an auxiliary problem posed on the \textit{basic cell} $Y^*$ (see (\ref{Y*})). 

	For each $\varepsilon>0$, we consider the  following auxiliary problem  given by
	\begin{equation}\label{probaux}
		\left\{\begin{array}{l}
			
			-\frac{1}{ \cos^2 ( \varepsilon z )} \frac{\partial^2}{\partial y^2} X^{\varepsilon}  -\frac{1}{ \cos (  \varepsilon  z)} \frac{\partial}{\partial z}
			\left( \cos (  \varepsilon z) \frac{\partial}{\partial z} \right) X^{\varepsilon}  = 0 \  \ \textrm{in} \  \ Y^{*}; \  \ \\
			\widetilde{\nabla}X^{\varepsilon} \cdot N = \xi^{\varepsilon} \  \ \textrm{on} \ \partial Y^{*}; \\
			X ^{\varepsilon}(\cdot, z)  \ \ L-\textrm{periodic};  \\
			\int_{Y^*} X^{\varepsilon} dy dz =0 ,\ \  \ \  
		\end{array}\right.
	\end{equation}
	where $\widetilde{\nabla}X^{\varepsilon} = \left(\frac{1}{\cos^2(\varepsilon z)}\frac{\partial X^{\varepsilon}}{\partial y},\frac{\partial X^{\varepsilon}}{\partial z} \right)$, $N = \left( - \frac{g'(y)}{\sqrt{1+g'(y)^2}} , \frac{1}{\sqrt{1+g'(y)^2}}\right) $ on the upper boundary $B_1$ of $Y^*$, and 
	$N = \left( 0 , -1 \right)$ on the lower boundary $B_2$ of $Y^*$. Moreover, $ \xi^\varepsilon$  is defined by parts:

	\begin{equation*}
		\xi^{\varepsilon} = \left\{\begin{array}{l}
			-\frac{g'\left( y \right)}{ \cos(  \varepsilon g(y)  ) \sqrt{1+g'\left(y \right)^2} }   \  \ \textrm{on} \ B_1 ,\ \  \\
			0 \  \ \textrm{on} \ B_2 .\ \  \\
		\end{array}\right.
	\end{equation*}

	Note that the weak formulation of  problem (\ref{probaux}) is:

	\begin{equation*}
		\int_{Y^*}     \left( \frac{1}{ \cos ( \varepsilon z) }\frac{\partial X^{\varepsilon}}{\partial y} \frac{\partial \psi }{\partial y}  + \cos ( \varepsilon z)  \frac{\partial X^{\varepsilon}}{\partial  z} \frac{\partial \psi}{\partial  z}     \right) \ dy dz =  0,
	\end{equation*}
	where $X^{\varepsilon}$, $\psi \in H^1_{per}(Y^*)$.
	
	\begin{lemma}\label{auxiliar}
	 Problem (\ref{probaux}) has a unique solution $X^{\varepsilon}$ in $H^1_{per}(Y^*)$ and it depends continuously on  $\varepsilon$.
	 \end{lemma}
	 \begin{proof}
		To simplify the notation, we will denote  $$L_{\varepsilon} X^{\varepsilon} =  -\frac{1}{ \cos^2 ( \varepsilon z )} \frac{\partial^2}{\partial y^2} X^{\varepsilon}  -\frac{1}{ \cos (  \varepsilon  z)} \frac{\partial}{\partial z}
		\left( \cos (  \varepsilon z) \frac{\partial}{\partial z} \right) X^{\varepsilon}.$$
		
		By  Lax-Milgram Theorem, problem (\ref{probaux}) has a unique solution $X^{\varepsilon}$ in $H^1_{per}(Y^*)$. Moreover, by classical regularity theory, we have $ X^{\varepsilon} \in H^2(Y^*)$.
		Let us define the map $F$  as follows: 
		\begin{eqnarray*}
			F: \mathbb{R} \times H^2_{per}(Y^{*})\setminus \{0\}&\rightarrow & L^2(Y^{*}) \times H^{ \frac{1}{2}}_{per}(\partial Y^{*}) \\
			(\varepsilon, X) & \mapsto & ( L_{\varepsilon} X, \widetilde{\nabla}X \cdot N - \xi^{\varepsilon})
		\end{eqnarray*}

		Note that if $F(0,X) = 0$, then $X=0$.We have
		$
		\frac{\partial F}{\partial X}(0, X) [v] = \left(L_0 v,\frac{\partial v}{\partial N} \right).$  To apply the Implicit Function Theorem, we need to show that $\frac{\partial F}{\partial X}(0, X)$ is an isomorphism. By Theorem 4.21 from \cite{CD99},  for every pair $(f,g) \in L^2(Y^{*}) \times H^{\frac{1}{2}}_{per}(\partial Y^{*})$,  problem (\ref{L0}) has a unique solution.
		
		\begin{equation}\label{L0}
			\left\{\begin{array}{l}
				L_{0} v  = f^\varepsilon \  \ \textrm{in} \  \ Y^{*}; \  \ \\
				\frac{\partial v}{\partial N} = g \  \ \textrm{on} \ \partial Y^{*};\\
				v(\cdot, z) \ \ \ \ \textrm{L-periodic}.
			\end{array}\right.
		\end{equation}
		
		We denote by $v^0$ the unique solution of problem (\ref{L0}). Moreover, by the aforementioned theorem, we have that 
		\begin{equation*}
			\Vert v^0 \Vert_{H^1_{per}(Y^*)} \leq C  \Vert g \Vert_{ H^{ \frac{1}{2}}_{per}(\partial Y^{*})}
		\end{equation*}
		thus, the operator $ L_{0} $  is  injective and continuous in the spaces under consideration. It follows from the Implicit Function Theorem that there exist neighborhoods of $0 $ and  the  null function in $H^1_{per}(Y^*)$, where the problem has an unique solution parametrized by $\varepsilon$, that is, there exists a pair $(\varepsilon, X^{\varepsilon})$ that solves  problem \eqref{probaux}, where  $X^{\varepsilon}$ depends differentiably on  $\varepsilon$ and therefore, continuously on $\varepsilon$.
		
	\end{proof}
	
	Note that  solution $X^\varepsilon$ captures the microscopic structure of the oscillation during the  convergence process. Since $g$ is $L$-periodic and $L<2\pi$,  the domain
	$\Omega^{\varepsilon}$ exhibits a periodic microstructure along the $\varphi_1$-direction, with period
	$\varepsilon L$ in $\varphi_1$. In particular, it can be described as the approximate
	union of cells over the intervals $	I_k := (k\varepsilon L,(k+1)\varepsilon L), $
	and, in each cell, the upper boundary is given by the graph of
	$g(\varphi_1/\varepsilon)$, which oscillates rapidly. More precisely, on the cell $I_k$, if one introduces the change of variables
	\[
	y = \frac{\varphi_1 - k\varepsilon L}{\varepsilon},
	\qquad
	z = \theta_1,
	\]
	then the portion of $\Omega_\varepsilon$ over $I_k$ is mapped onto  the \textit{basic cell} $Y^*$.

		Let us consider $Y = (0,L) \times (0, g_1)$ and the following families of diffeomorphisms $T^{\varepsilon}_k: A^{\varepsilon}_k \mapsto Y$ 
	
	\begin{equation}\label{isoT}
		T^{\varepsilon}_k(\varphi_1, \theta_1) = \left( \frac{  \varphi_1 - \varepsilon kL}{\varepsilon}  , \theta_1   \right), 
	\end{equation} 
	where
	$A^{\varepsilon}_k  = \lbrace (\varphi_1, \theta_1) \in \mathbb{R}^2; \varepsilon kL < \varphi_1 < \varepsilon L (k+1 ) \ \ and \ \ 0 < \theta_1 < g_1 \rbrace $
	with $k \in \mathbb{N} $ satisfying $0 \leq k < \frac{2 \pi}{\varepsilon L}$. Taking the diffeomorphism given in  (\ref{isoT}) and the extension operator $P_{\varepsilon}$ from Lemma \ref{lema} with $G_{\varepsilon}= g(\varphi_1)$,  independent of $ \varepsilon$, we define $\tau^{\varepsilon}_k: A^{\varepsilon}_k\rightarrow\mathbb{R}$ by
	$$
	\tau^{\varepsilon}_k(\varphi_1, \theta_1):=\varphi_1 - \varepsilon \left( P_{\varepsilon}X^{\varepsilon} \circ T^{\varepsilon}_k (\varphi_1, \theta_1) \right)= \varphi_1 - \varepsilon \left( P_{\varepsilon}X^{\varepsilon}  \left( \frac{  \varphi_1 - \varepsilon kL}{\varepsilon}  , \theta_1   \right)  \right).
	$$

	Now, we define $\tau^{\varepsilon}: Y \rightarrow\mathbb{R}$ by $\tau^{\varepsilon} (\varphi_1,\theta_1) = \tau_k^{\varepsilon} (\varphi_1,\theta_1) $ with $(\varphi_1,\theta_1) \in A^{\varepsilon}_k $, for some $k$. These are the so called \textit{Tartar's test functions}.

	\begin{proposition} The function   $\tau^{\varepsilon}$ is well defined, belongs to $H^1(\Omega, cos(\varepsilon\theta_1)) $ and depends continuously on $\varepsilon$.
	\end{proposition}
\begin{proof}
	Since $P_{\varepsilon}X^{\varepsilon}(\cdot, z)$ is $L-$periodic, it follows that  $P_{\varepsilon}X^{\varepsilon}(L, z)=P_{\varepsilon}X^{\varepsilon}(0,z),$ so  $\tau_k^{\varepsilon}(\varepsilon L (k+1),\theta_1)=\tau_{k+1}^{\varepsilon}(\varepsilon L (k+1),\theta_1)$ . For the continuous dependence of $\tau^{\varepsilon}$ on $\varepsilon$, it is sufficient to observe that $X^{\varepsilon}$ depends continuously on $\varepsilon$, which was proved in Lemma \ref{auxiliar}.
	\end{proof}

	The following result will be crucial for formulating a variational problem that excludes the rapid diffusion term $\frac{1}{\varepsilon^2}$. This will allow us  to  take the limit of the resulting problem as $\varepsilon\rightarrow 0$, see equation~\eqref{4.13}.

	\begin{lemma}
		Given a function  $\Psi \in H^1_{per}(\Omega^{\varepsilon}, cos(\varepsilon\theta_1))$ with $\Psi=0$ in a neighborhood of the lateral boundaries of $\Omega^{\varepsilon}$, there exists $\varepsilon_0>0$  such that for all $0< \varepsilon < \varepsilon_0$ we have: 
		\begin{equation}\label{4.11}
			\int_{\Omega^{\varepsilon}} \left(\frac{\partial \tau^{\varepsilon} }{\partial \varphi_1 } \frac{1}{ \cos^2 (  \varepsilon \theta_1 )}   \frac{\partial \Psi }{\partial \varphi_1}+ \frac{\partial \tau^{\varepsilon} }{\partial \theta_1 } \frac{1}{ \varepsilon^2 }   \frac{\partial \Psi}{\partial \theta_1}   \right) \cos(\varepsilon \theta_1) d\varphi_1 d\theta_1 =0.
		\end{equation}
		
	\end{lemma}
	
	\begin{proof}

		Let $\varepsilon_0 > 0$  be such that for the largest integer $k_0$ in  $[0 , \frac{2 \pi}{\varepsilon L}]$ we have $\Psi \Big\vert_{A_{k_0}^{\varepsilon} \cap \Omega^{\varepsilon}} \equiv 0$. We will analyze the integral above in each subset  $A_{k}^{\varepsilon} \cap \Omega^{\varepsilon}$, $0 \leq k < k_0 $. Taking  the diffeomorphism $T_k^{\varepsilon}: A_{k}^{\varepsilon} \cap \Omega^{\varepsilon} \rightarrow Y^*$, with $0\leq k\leq k_0$, consider $S_k^{\varepsilon}:=(T_k^{\varepsilon})^{-1}: Y^* \rightarrow  A_{k}^{\varepsilon} \cap \Omega^{\varepsilon}.$

		Since $\Omega^{\varepsilon}\ni (\varphi_1, \theta_1)\rightarrow  T^\varepsilon_k(\varphi_1,\theta_1)=\left(\frac{\varphi_1 - \varepsilon kL}{\varepsilon}, \theta_1 \right) \in Y^* $, then $\frac{\partial  }{\partial \varphi_1 } = \frac{1}{\varepsilon }\frac{\partial  }{\partial y } $ and $\frac{\partial  }{\partial \theta_1 } = \frac{\partial  }{\partial z } $. It follows that
		
		$$
		\frac{\partial \tau^{\varepsilon} }{\partial \varphi_1 }  (\varphi_1, \theta_1) = 1- \frac{\partial X^{\varepsilon} }{\partial y } \left(\frac{\varphi_1- \varepsilon kL}{\varepsilon}, \theta_1 \right)
		= 1-  \frac{\partial X^{\varepsilon} }{\partial y } \left(\frac{\varphi_1}{\varepsilon}, \theta_1 \right):= \frac{\partial \tau^{\varepsilon} }{\partial \varphi_1 } (y,z)
		$$
		
		$$
		\frac{\partial \tau^{\varepsilon} }{\partial \theta_1 } (\varphi_1, \theta_1) =   - \varepsilon \frac{\partial X^{\varepsilon} }{\partial z } \left(\frac{\varphi_1- \varepsilon kL}{\varepsilon}, \theta_1 \right) = - \varepsilon  \frac{\partial X^{\varepsilon} }{\partial z } \left(\frac{\varphi_1}{\varepsilon}, \theta_1 \right):= \frac{\partial \tau^{\varepsilon} }{\partial \theta_1 }(y,z)
		$$
		
		Replacing these equations into (\ref{4.11}), we obtain the equivalent expression:
		\begin{align*} 
			\int_{\Omega^{\varepsilon}} \left(    \frac{\partial X^{\varepsilon} }{\partial y }    \frac{1}{ \cos^2 (  \varepsilon \theta_1 )}   \frac{\partial \Psi }{\partial \varphi_1}+   \frac{1}{ \varepsilon } \frac{\partial X }{\partial z }      \frac{\partial \Psi}{\partial \theta_1}   \right) \cos(\varepsilon \theta_1) d\theta_1 \ d \varphi_1= \int_{\Omega^{\varepsilon}} \frac{1}{ \cos^2 (  \varepsilon \theta_1 )}   \frac{\partial \Psi }{\partial \varphi_1} \cos(\varepsilon \theta_1) \ d\varphi_1 d\theta_1.
		\end{align*}
		
		After straightforward computation,  the right-hand side becomes: 
		\begin{align*}
			\int_{\Omega^{\varepsilon}} \frac{1}{ \cos^2 (  \varepsilon \theta_1 )}   \frac{\partial \Psi }{\partial \varphi_1} \cos(\varepsilon \theta_1) \ d\varphi_1 d\theta_1 &= \int_{0}^{\varepsilon k_0 L}   \frac{1}{ \varepsilon \cos \left(  \varepsilon g\left(\frac{\varphi_1}{\varepsilon} \right) \right)} \Psi \left(\varphi_1, g\left(\frac{\varphi_1}{\varepsilon} \right) \right) g'\left(\frac{\varphi_1}{\varepsilon} \right)      \ d \varphi_1. 
		\end{align*}
		
		For the left side, we begin observing that 
		$$
		\nabla X^{\varepsilon} \circ T_k^{\varepsilon} \cdot\left( \frac{1}{ \cos^2 (  \varepsilon \theta_1 )}  \frac{\partial \Psi }{\partial \varphi_1},  \frac{1}{ \varepsilon }  \frac{\partial \Psi}{\partial \theta_1} \right)=  \left(    \frac{\partial X^{\varepsilon} }{\partial y }    \frac{1}{ \cos^2 (  \varepsilon \theta_1 )}   \frac{\partial \Psi }{\partial \varphi_1}+   \frac{1}{ \varepsilon } \frac{\partial X }{\partial z }      \frac{\partial \Psi}{\partial \theta_1}   \right). 
		$$
		
		Since $\Omega^{\varepsilon}=\cup_{k=1}^{k_0} (A_k^{\varepsilon}\cap\Omega^{\varepsilon})$, it suffices to show that
	$$
			\int_{\Omega^{\varepsilon} \cap A_k^{\varepsilon}} \nabla X^{\varepsilon} \circ T_k^{\varepsilon} \cdot
			\left( \frac{1}{ \cos^2 (  \varepsilon \theta_1 )}   \frac{\partial \Psi }{\partial \varphi_1} ,  \frac{1}{ \varepsilon }  \frac{\partial \Psi}{\partial \theta_1} \right) \cos(\varepsilon \theta_1) \ d\varphi_1 d\theta_1 = \int_{\varepsilon k L}^{\varepsilon (k+1) L}   \frac{\Psi \left(\varphi_1, g\left(\frac{\varphi_1}{\varepsilon} \right) \right)}{ \varepsilon \cos \left(  \varepsilon g\left(\frac{\varphi_1}{\varepsilon} \right) \right)}  g'\left(\frac{\varphi_1}{\varepsilon} \right)      \ d \varphi_1.
	$$
		In fact, applying the change of variables $S_k^{\varepsilon}: Y^*\rightarrow \Omega\cap A_k^{\varepsilon}$, $\varphi_1= \varepsilon(y+kL), \theta_1=z$; $\frac{\partial}{\partial \theta_1}=\frac{\partial}{\partial z}$,  $\varepsilon \frac{\partial }{\partial \varphi_1} = \frac{\partial }{\partial y} $ and $ \psi = \Psi \circ  S_k^{\varepsilon} $ , we have
		
		\begin{align*}
			&\int_{\Omega^{\varepsilon} \cap A_k^{\varepsilon}} \nabla X^{\varepsilon} \circ  T_k^{\varepsilon}\cdot 
			\left( \frac{1}{ \cos^2 (  \varepsilon \theta_1 )}   \frac{\partial \Psi }{\partial \varphi_1} ,  \frac{1}{ \varepsilon }  \frac{\partial \Psi}{\partial \theta_1} \right)  \cos(\varepsilon \theta_1)   \ d\varphi_1 d\theta_1 =\\
			& \int_{Y^*}     \left(\frac{\partial X^{\varepsilon}}{\partial y}, \frac{\partial X^{\varepsilon}}{\partial  z}   \right) \left( \frac{1}{ \cos ( \varepsilon z) } \frac{\partial \psi }{\partial y} , \cos ( \varepsilon z)  \frac{\partial \psi}{\partial  z}    \right)  \ dy dz    =  \int_{\partial Y^*}    \left( \psi  \frac{1}{ \cos ( \varepsilon z) }  \frac{\partial X^{\varepsilon}}{\partial y} ,  \psi \cos ( \varepsilon z)     \frac{\partial X^{\varepsilon}}{\partial  z} \right)\cdot  ( N_1, N_2)  \ dS  \\
			&= \int_{ B_1}    \left( \psi  \frac{1}{ \cos ( \varepsilon z) }  \frac{\partial X^{\varepsilon}}{\partial y} ,  \psi \cos ( \varepsilon z)     \frac{\partial X^{\varepsilon}}{\partial  z} \right)  (  N_1, N_2)  \ dS   
			= \int_{0}^{L}     \left( -\psi  \frac{g'(y)}{ \cos ( \varepsilon z) }  \frac{\partial X^{\varepsilon}}{\partial y} + \cos ( \varepsilon z)  \psi     \frac{\partial X^{\varepsilon}}{\partial  z} \right)    \ dy. 
		\end{align*}  
		
		Applying the change of variable $\varphi_1= \varepsilon(y+kL)$; $\varepsilon \frac{\partial }{\partial \varphi_1} = \frac{\partial }{\partial y} $ once more, we obtain
		\begin{align*}
			\int_{0}^{L}& \left(  -\psi  \frac{g'(y)}{ \cos ( \varepsilon z) }  \frac{\partial X^{\varepsilon}}{\partial y} \right.  \left. + \cos( \varepsilon z)  \psi \frac{\partial X^{\varepsilon}}{\partial  z} \right) \ dy \\
			&= \int_{\varepsilon k L}^{\varepsilon (k+1) L} \left(- \Psi \left(\varphi_1, g\left(\frac{\varphi_1}{\varepsilon} \right) \right)  \frac{g'\left(\frac{\varphi_1}{\varepsilon} \right)}{ \cos \left(  \varepsilon g\left(\frac{\varphi_1}{\varepsilon} \right) \right) }  \frac{ \partial X^{\varepsilon}}{\partial \varphi_1} + \cos ( \varepsilon  g\left(\frac{\varphi_1}{\varepsilon} \right))  \Psi     \frac{\partial X^{\varepsilon}}{\partial  \theta_1} \right) \frac{1}{\varepsilon}    \ d \varphi_1 \\
			&= \int_{\varepsilon k L}^{\varepsilon (k+1) L}  \frac{1}{ \varepsilon \cos \left(  \varepsilon g\left(\frac{\varphi_1}{\varepsilon} \right) \right)} \Psi \left(\varphi_1, g\left(\frac{\varphi_1}{\varepsilon} \right) \right) g'\left(\frac{\varphi_1}{\varepsilon} \right)      \ d \varphi_1. 
		\end{align*}
	\end{proof}
	
	Subsequently, Tartar's test function constructed from the weak solution of the auxiliary problem ~\eqref{probaux} will be applied to problem~\eqref{eqO}. By combining the weak solution of problem~\eqref{eqO} and equation \eqref{4.11} from the lemma above, we derive the variational form of the problem that is independent of $1/\varepsilon^2$. 
	
	Let $u^{\varepsilon}$ be the weak solution of  problem~\eqref{eqO}. Considering $\Psi= \phi u^{\varepsilon}$ as a test function in \eqref{4.11}, with a function $ \phi \in \mathcal{C}^{\infty}(0, 2 \pi)$,  such that  $ \phi   $ vanishes in a neighborhood of $0$ and $2 \pi$ and extended periodically, we obtain:

	\begin{equation}\label{4.112}
		\int_{\Omega^{\varepsilon}} \left(\frac{1}{ \cos^2 (  \varepsilon \theta_1 )}  \frac{\partial u^{\varepsilon}  }{\partial \varphi_1}\frac{\partial \tau^{\varepsilon}  }{\partial \varphi_1} \phi+\frac{1}{ \cos^2 (  \varepsilon \theta_1 )}  \frac{\partial \phi }{\partial \varphi_1}\frac{\partial \tau^{\varepsilon}  }{\partial \varphi_1}u^{\varepsilon} + \frac{1}{ \varepsilon^2 }   \frac{\partial  u^{\varepsilon}}{\partial \theta_1} \frac{\partial \tau^{\varepsilon}  }{\partial \theta_1} \phi  \right) \cos(\varepsilon \theta_1) \ d\varphi_1 d\theta_1 =0.
	\end{equation}

	On the other hand, taking as a test function $\psi= \phi \tau^{\varepsilon}$, with  $ \phi \in \mathcal{C}^{\infty}(0, 2 \pi)$,  such that  $ \phi   $ vanishes in a neighborhood of $0$ and $2 \pi$ and extended periodically, on the weak formulation of problem \eqref{eqO},  we have the following: 
	\begin{equation}
		\begin{aligned}\label{fr}
			\int_{\Omega^{\varepsilon}} \left( \frac{1}{ \varepsilon^2 }   \frac{\partial u^{\varepsilon}}{\partial \theta_1}  \frac{\partial  \tau^{\varepsilon}}{\partial \theta_1} \phi  + \frac{1}{ \cos^2 (\varepsilon \theta_1 )} \frac{\partial u^{\varepsilon}}{\partial \varphi_1}\frac{\partial \phi  }{\partial \varphi_1} \tau^{\varepsilon} + \frac{1}{ \cos^2 (\varepsilon \theta_1 )} \frac{\partial u^{\varepsilon}}{\partial \varphi_1}\frac{\partial  \tau^{\varepsilon} }{\partial \varphi_1}  \phi    +  u^{\varepsilon}\phi \tau^{\varepsilon} \right) \cos(\varepsilon \theta_1)\ d\varphi_1 d\theta_1 \\ =\int_{\Omega^{\varepsilon}}  f^\varepsilon \phi \tau^{\varepsilon} \cos(\varepsilon \theta_1) \ d\varphi_1 d\theta_1.
		\end{aligned}
	\end{equation}
	Taking the difference between  \eqref{fr} and  \eqref{4.112}, we obtain:
	
	\begin{eqnarray}\label{4.13}
		\int_{\Omega^{\varepsilon}} \left( \frac{1}{ \cos^2 ( \varepsilon \theta_1) }\frac{\partial u^{\varepsilon}}{\partial \varphi_1} \frac{\partial \phi }{\partial \varphi_1} \tau^{\varepsilon}  -    \frac{1}{ \cos^2 ( \varepsilon \theta_1) }\frac{\partial \tau^{\varepsilon}}{\partial \varphi_1} \frac{\partial \phi }{\partial \varphi_1}u^{\varepsilon} +    u^{\varepsilon} \tau^{\varepsilon}  \phi \right) \cos( \varepsilon \theta_1) \ \ d\varphi_1 d\theta_1    \\ =\int_{\Omega^{\varepsilon}} f^\varepsilon \tau^{\varepsilon}  \phi \cos( \varepsilon \theta_1) \ \ d\varphi_1 d\theta_1. \nonumber
	\end{eqnarray}

	Now we need to take the limit in \eqref{4.13} and in the weak formulation of problem \eqref{eqO}, presented in $(\ref{4.14}).$
%

	To work on a fixed domain, we extend the functions in order to study the convergence of the problem as $\varepsilon$ approaches zero. Both expressions need to be written as integrals over the same fixed domain. To do this, we will use the extension $P_{\varepsilon}$ constructed in Lemma \ref{lema}, the standard extension by zero, which we will denote by $  \widetilde{} $, and the characteristic function   $\chi^{\varepsilon}$ of $\Omega^{\varepsilon}$ as follows:

	\begin{align}\label{4.15}
		\int_{\Omega} \left( \frac{1}{ \cos^2 ( \varepsilon \theta_1) }\frac{\widetilde{\partial u^{\varepsilon}}}{\partial \varphi_1} \frac{\partial \phi }{\partial \varphi_1} \tau^{\varepsilon}  -    \frac{1}{ \cos^2 ( \varepsilon \theta_1) } \widetilde{ \frac{\partial \tau^{\varepsilon} }{\partial \varphi_1} }  \frac{\partial \phi }{\partial \varphi_1} P_{\varepsilon}(u^{\varepsilon} ) +\chi^{\varepsilon}   P_{\varepsilon} (u^{\varepsilon} ) \tau^{\varepsilon}  \phi \right) \cos( \varepsilon \theta_1) \  d\varphi_1 d\theta_1 \nonumber    \\ =\int_{\Omega} \chi^{\varepsilon}  \widetilde{f^\varepsilon} \tau^{\varepsilon}  \phi \cos( \varepsilon \theta_1) \ d\varphi_1 d\theta_1,
	\end{align}
	for all periodic function $ \phi \in \mathcal{C}^{\infty}(0, 2 \pi)$,  such that  $ \phi   $ vanishes in a neighborhood of $0$ and $2 \pi$. Considering  $\psi \in H^1_{per}(0,2 \pi)$ on \eqref{4.14},  the terms including partial derivatives with respect to 
	$\theta_1$ do not appear and we obtain: 

	\begin{equation}\label{4.16}
		\int_{\Omega} \left(   \frac{1}{ \cos^2 ( \varepsilon \theta_1) }\frac{\widetilde{\partial u^{\varepsilon}}}{\partial \varphi_1} \frac{\partial \psi }{\partial \varphi_1}   + \chi^{\varepsilon}  P_{\varepsilon}(u^{\varepsilon} )  \psi \right) \cos( \varepsilon \theta_1) \ d\varphi_1 d\theta_1  = \int_{\Omega}\chi^{\varepsilon}  \widetilde{f^\varepsilon} \psi \cos( \varepsilon \theta_1) \ d\varphi_1 d\theta_1,
	\end{equation}
	for all $\psi \in H^1_{per}(0,2 \pi)$.

	We aim to take the limit in the expressions above, (\ref{4.15}) and (\ref{4.16}). To this end, we analyze the limits of the individual functions that make up the integrands.
	The next proposition is an important tool for analyzing the convergence of periodic functions.

	\begin{proposition}\label{teochi} Let $1 \leq p \leq \infty$, $Y$ a n-cube in $\mathbb{R}^n$  and $f$ a function $Y$-periodic in $L^p(Y)$. Let 
		$
		f_{\varepsilon}(x) = f\left(\frac{x}{\varepsilon}\right) \ \ \textrm{a. e. in } \mathbb{R}^n.
		$
		Under these conditions, if $1 \leq p < \infty$ when $\varepsilon \rightarrow 0 $, then
		$$
		f_{\varepsilon} \rightharpoonup M_Y(f)=\frac{1}{\vert Y \vert} \int_Y f(y) dy  \ \ \textrm{in } L^p(\Gamma),
		$$
		for any bounded open set $\Gamma$ in $\mathbb{R}^n$. If $p=\infty$, then
		
		$$
		f_{\varepsilon} \stackrel{*}{\rightharpoonup} M_Y(f)=\frac{1}{\vert Y \vert} \int_Y f(y) dy  \ \ \textrm{in } L^{\infty}(\mathbb{R}^n).
		$$
		
	\end{proposition}
	\begin{proof} See \cite{CD99}, Theorem 2.6.
		
	\end{proof}

	\subsection{Limit of $\chi^{\varepsilon}$. }

	Let $\chi$ be  the characteristic function of $Y^*$ and we extend  $\chi$ periodically on the variable $y \in \mathbb{R} $ and denote this extension again by $\chi$.  If we denote by  $\chi^{\varepsilon}$ the characteristic function of $\Omega^{\varepsilon}$ we have:
	\begin{equation}\label{chi}
		\chi^{\varepsilon}(\varphi_1, \theta_1) = \chi \left(  \frac{\varphi_1}{\varepsilon} ,\theta_1   \right), \ \ \ \textrm{for} \ \ (\varphi_1, \theta_1) \in \Omega^{\varepsilon}.
	\end{equation}
	Since $\chi^{\varepsilon}$ is a periodic function, we use  Proposition \ref{teochi} to guarantee the convergence of periodic functions.
	
	\begin{lemma} Let $\chi^{\varepsilon}$ the characteristic function of $\Omega^{\varepsilon}$. Then, as $\varepsilon \rightarrow 0$
		\begin{equation*}
			\chi^{\varepsilon}(\varphi_1,\theta_1)\cos(\varepsilon\theta_1)
			\stackrel{*}{\rightharpoonup}
			\gamma(\theta_1)
			\quad   \text{ on } \  L^{\infty}(\Omega),
		\end{equation*}
		where
		\begin{equation*}
			\gamma(\theta_1)
			=
			\frac{1}{L}\int_0^L \chi(s,\theta_1)\,ds,
			\qquad \text{for a.e. } \theta_1\in(0,g_1).
		\end{equation*}
		Moreover, $\gamma$ does not depend on $\varphi_1$ and satisfies
		\begin{equation*}
			L\int_0^{g_1}\gamma(\theta_1)\,d\theta_1
			=
			|Y^*|.
		\end{equation*}
	\end{lemma}
	
	\begin{proof}
	By construction,  $\chi^{\varepsilon}$  is periodic in the variable $\varphi_1$, so for each fixed $\theta_1$ the function $ \chi^{\varepsilon}(\cdot, \theta_1) \cos(\varepsilon \theta_1 )$ is $I$-periodic, and we can apply Proposition  \ref{teochi}, where $I= (0, 2 \pi)$. From~(\ref{chi}) and Proposition~\ref{teochi}, we obtain, as $\varepsilon \rightarrow 0$
	\begin{equation}\label{4.17}
		\chi^{\varepsilon}(\cdot, \theta_1)  \cos(\varepsilon \theta_1 ) \stackrel{*}{\rightharpoonup} \gamma(\theta_1) := \frac{1}{ L } \int_0^L \chi(s,\theta_1 ) ds  \ \ \text{on} \ L^{\infty}(I), \forall \  \theta_1 \in (0, g_1).
	\end{equation}

	Note that the limit $\gamma$ does not depend on the variable $\varphi_1$ and we can obtain the area of the open set $Y^*$ by the formula
	\begin{equation}\label{4.18}
		L  \int_0^{g_1} \gamma(\theta_1) \ d\theta_1 = \vert Y^*\vert.
	\end{equation}
	If $\phi \in L^1(\Omega)$,  then by  Fubini-Tonelli's Theorem, $\phi(\cdot, \theta_1) \in L^1(0, 2\pi)$ for almost every $\theta_1 \in(0,g_1)$. Thus, from (\ref{4.17}) we have
	
	$$
	H^{\varepsilon}(\theta_1) := \int_I \phi(\varphi_1, \theta_1)\left\lbrace  \chi^{\varepsilon}(\varphi_1, \theta_1) \cos(\varepsilon \theta_1) - \gamma(\theta_1)  \right\rbrace \ d \varphi_1  \rightarrow 0 \ \ \textrm{as} \ \varepsilon \rightarrow 0,
	$$
	for almost every $\theta_1 \in(0,g_1)$ and for all $\phi \in L^1(\Omega)$. Since $\vert \chi^{\varepsilon} \cos(\varepsilon \theta_1) - \gamma \vert \leq 1$, it follows that 
	$$
	\int_{\Omega} \phi(\varphi_1, \theta_1) \left\lbrace  \chi^{\varepsilon}(\varphi_1, \theta_1) \cos(\varepsilon \theta_1) - \gamma(\theta_1)  \right\rbrace  \ d \varphi_1  \ d \theta_1 = \int_0^{g_1} H^{\varepsilon}(\theta_1) \ d \theta_1
	$$
	and
	$$
	\vert H^{\varepsilon}(\theta_1) \vert \leq \int_{I} \phi(\varphi_1, \theta_1)\ d \varphi_1  \ \ \textrm{a.e.  in} \ (0,g_1),
	$$
	therefore, by the Dominated Convergence Theorem of Lebesgue, we have, as $\varepsilon \rightarrow 0$
	\begin{equation}\label{4.19}
		\chi^{\varepsilon}\cos(\varepsilon \theta_1)  \stackrel{*}{\rightharpoonup} \gamma, \ \  \text{on} \ L^{\infty}(\Omega).
	\end{equation}
	\end{proof}
	
	\subsection{ Limit of functions extended by zero $\frac{\widetilde {\partial u^{\varepsilon}}}{\partial \varphi_1}$, $\frac{\widetilde {\partial  u^{\varepsilon}}}{\partial  \theta_1}$ and $\widetilde{f^\varepsilon}$}
	
	Now we will compute the limit of the functions $\frac{\widetilde {\partial u^{\varepsilon}}}{\partial \varphi_1}$, $\frac{\widetilde {\partial  u^{\varepsilon}}}{\partial  \theta_1}$ and $\widetilde{f^\varepsilon}$, as $\varepsilon \rightarrow 0$.
	
	\begin{lemma} Let $ \Vert f^\varepsilon \Vert_{L^2(\Omega,\cos (\varepsilon\theta_1))}$ be uniformly bounded, then
		$$ \frac{\widetilde {\partial u^{\varepsilon}}}{\partial \varphi_1} \rightharpoonup \xi^* \ \  \text{in}  \ \ L^2(\Omega, \cos (\varepsilon\theta_1))   \  \text{and} \ \ \frac{\widetilde {\partial  u^{\varepsilon}}}{\partial  \theta_1} \rightarrow 0 \ \  \text{in} \ L^2(\Omega,\cos (\varepsilon\theta_1)) \ \ \text{as} \ \ \varepsilon \rightarrow 0. $$
	\end{lemma}
	
	\begin{proof}
	
	Since $ \Vert f^\varepsilon \Vert_{L^2(\Omega,\cos (\varepsilon\theta_1))}$ is uniformly bounded, we obtain from (\ref{4.5}) that there exists $M$ independent of $\varepsilon$ such that
	
	\begin{equation}\label{4.20}
		\left\Vert \widetilde{u^{\varepsilon}} \right\Vert_{L^2(\Omega,\cos (\varepsilon\theta_1))}, 
		\frac{1}{\varepsilon  }\left\Vert  \frac{\widetilde {\partial  u^{\varepsilon}}}{\partial  \theta_1} \right\Vert_{L^2(\Omega,\cos (\varepsilon\theta_1))} \ \ \textrm{and} \ \  \left\Vert \frac{1}{ \cos (\varepsilon \theta_1) } \frac{\widetilde {\partial u^{\varepsilon}}}{\partial \varphi_1} \right\Vert_{L^2(\Omega,\cos (\varepsilon\theta_1))} \leq M ,
	\end{equation}
	for all  $\varepsilon > 0$.
	Consequently, $\left\Vert \widetilde{u^{\varepsilon}} \right\Vert_{L^2(\Omega), \cos (\varepsilon\theta_1)} \leq M$. Then, we can extract a subsequence, still denoted by $\widetilde{u^{\varepsilon}}$, such that $\widetilde{u^{\varepsilon}} \rightharpoonup u^* $  on $L^2(\Omega, \cos (\varepsilon\theta_1))$. Similarly, we can extract a subsequence, still denoted by 
	$\frac{\widetilde {\partial  u^{\varepsilon}}}{\partial  \theta_1}$ and $\frac{\widetilde {\partial u^{\varepsilon}}}{\partial \varphi_1}$, such that
	$ \frac{\widetilde {\partial u^{\varepsilon}}}{\partial \varphi_1} \rightharpoonup \xi^*$  in $ L^2(\Omega, \cos (\varepsilon\theta_1))   $ and
	$\frac{\widetilde {\partial  u^{\varepsilon}}}{\partial  \theta_1} \rightarrow 0$ in  $L^2(\Omega, \cos (\varepsilon\theta_1))$ as $ \varepsilon \rightarrow 0 $ .
	\end{proof}
	
		For the function $\widetilde{f^\varepsilon}$, we defined  
	\begin{equation*}
		\hat{f}^{\varepsilon}(\varphi_1) : =  \int_0^{g_1} \widetilde{f^\varepsilon}(\varphi_1, \theta_1)  d\theta_1 = \int_0^{g(\varphi_1/\varepsilon)}  f^\varepsilon( \varphi_1, \theta_1)  \ d\theta_1.
	\end{equation*}
	
	Note that  $\hat{f}^{\varepsilon} \in L^2(0, 2 \pi )$ and  $||\hat{f}^{\varepsilon}||_{L^2(0, 2 \pi )}\leq ||f^\varepsilon||_{L^2(\Omega^{\varepsilon}, \cos(\varepsilon \theta_1))}\leq C$, by Jensen's  inequality. Hence, by subsequences, we have the existence of a function $\hat{f} \in L^2(0,2 \pi)$ such that  $	\hat{f}^{\varepsilon} \rightharpoonup \hat{f} \ \ \ \text{in } \ \ L^2(0,2 \pi).$

	\begin{remark}\label{obs4.1}
		Observe that in the case where $f^\varepsilon(\varphi_1, \theta_1 ) = f(\varphi_1)$ then   
		\begin{equation*}
			\hat{f}(\varphi_1) = \left( \int_0^{g_1} \gamma(\theta_1) \ d\theta_1 \right) f(\varphi_1) = \frac{\vert Y^*\vert}{L} f^\varepsilon(\varphi_1),
		\end{equation*} 
		where we have used (\ref{4.18}).
	\end{remark}
	
	\subsection{ Limit of the extended functions $P_{\varepsilon}u^{\varepsilon}$  and $\frac{\partial P_{\varepsilon}u^{\varepsilon} }{\partial \theta_1}$}
	
	\begin{lemma} Let  $f^\varepsilon \in L^2(\Omega,\cos (\varepsilon\theta_1)$ be  uniformly bounded  in $\varepsilon$ and  let $u^{\varepsilon} \in H^1_{per}(\Omega^{\varepsilon}, \cos(\varepsilon\theta_1))$.  Then, 
		\begin{align}\label{4.25}
			&P_{\varepsilon}u^{\varepsilon} \rightharpoonup u_0 \,\,\,\,\ \text{in } \  H^1_{per}(\Omega, \cos (\varepsilon\theta_1)) \ \textrm{and strongly} \  \text{in } \ L_{per}^p(\Omega, \cos (\varepsilon\theta_1)) \  \forall p \geq 1,  \nonumber \\
			&\frac{\partial P_{\varepsilon}u^{\varepsilon} }{\partial \theta_1} \rightarrow 0 \ \text{in } \ \ L_{per}^2(\Omega, \cos (\varepsilon\theta_1)) .
		\end{align}
	\end{lemma}
	
	\begin{proof}
	Using the  priori estimate (\ref{4.5}), the fact that $u^{\varepsilon} \in H^1_{per}(\Omega^{\varepsilon}, \cos(\varepsilon\theta_1))$,  the results from Lemma (\ref{lema}) and that  $ \Vert  \frac{\partial u^{\varepsilon} }{\partial \varphi_1}\Vert_{L^2(\Omega, \cos (\varepsilon\theta_1))} \leq  \Vert  \frac{1}{\cos(\varepsilon \theta)}\frac{\partial u^{\varepsilon} }{\partial \varphi_1}\Vert_{L^2(\Omega, \cos (\varepsilon\theta_1))}$, we obtain that
	
	$$
	\Vert P_{\varepsilon} u^{\varepsilon} \Vert_{L^2(\Omega, \cos (\varepsilon\theta_1))}, \left\Vert \frac{\partial P_{\varepsilon} u^{\varepsilon}}{\partial \varphi_1} \right\Vert_{L^2(\Omega, \cos (\varepsilon\theta_1))} \ \textrm{and} \ \frac{1}{\varepsilon}\left\Vert \frac{\partial P_{\varepsilon} u^{\varepsilon}}{\partial \theta_1} \right\Vert_{L^2(\Omega, \cos (\varepsilon\theta_1))} \leq \widetilde{M} \ \textrm{for all} \ \varepsilon>0,
	$$
	where  $\widetilde{M} >0$ is a positive constant independent of $\varepsilon$ given by  (\ref{4.20}) and  Lemma \ref{lema}. Then, we can extract a subsequence, which still will be denoted by $P_{\varepsilon}u^{\varepsilon}$ and a  function $u_0 \in H^1(\Omega, \cos (\varepsilon\theta_1))$, such that the result follows. 

	\end{proof}

	A consequence of the limits  in (\ref{4.25}) is that $u_0(\varphi_1, \theta_1)=u_0(\varphi_1)$ on $\Omega$. More precisely, 
	$$
	\frac{\partial u_0 }{\partial \theta_1}(\varphi_1, \theta_1) = 0 \ \ \textrm{a. e. } \Omega.
	$$
	In fact, for all $ \varphi_1-$periodic function $ \phi \in \mathcal{C}^{\infty}(0, 2 \pi)$  such that  $ \phi   $ vanishes in a neighborhood of $0$ and $2 \pi$, we have by (\ref{4.25}) that 
	\begin{eqnarray*}
		\int_{\Omega} \frac{\partial u_0}{\partial \theta_1} \phi  \ d \varphi_1 \ d  \theta_1 & =& - \int_{\Omega} u_0 \frac{\partial \phi }{\partial \theta_1} \ d \varphi_1 \ d  \theta_1  = - \int_{\Omega}  \lim_{\varepsilon \rightarrow 0} P_{\varepsilon} u^{\varepsilon} \frac{\partial \phi}{\partial \theta_1} \ d \varphi_1 \ d  \theta_1 \\
		& = &- \lim_{\varepsilon \rightarrow 0} \int_{\Omega} P_{\varepsilon} u^{\varepsilon} \frac{\partial\phi }{\partial \theta_1} \ d \varphi_1 \ d  \theta_1 =  \lim_{\varepsilon \rightarrow 0} \int_{\Omega}\frac{\partial P_{\varepsilon} u^{\varepsilon}}{\partial \theta_1}  \phi \ d \varphi_1 \ d  \theta_1 = 0. \\
	\end{eqnarray*}
	In particular, the function $u_0$ belongs to $ H^1_{per}(0,2\pi)$, since $ P_{\varepsilon}u^{\varepsilon} \in H^1_{per}(\Omega) $ and $\frac{\partial u_0 }{\partial \theta_1}(\varphi_1, \theta_1) = 0 \ \ \textrm{a. e. }$.

	\begin{remark} We saw in (\ref{4.19}) that $\chi^{\varepsilon} \cos(\varepsilon \theta_1)  \stackrel{*}{\rightharpoonup} \gamma $ in $L^{\infty}(\Omega)$. In particular, we have that $\chi^{\varepsilon} \rightharpoonup \gamma $ in  $L^2(\Omega )$. Furthermore, it is worth noting that $\chi^{\varepsilon},  \gamma \in  L^2(\Omega )$ (since $\Omega$ is bounded). 
		Since $\widetilde{u}^{\varepsilon} = \chi^{\varepsilon} P_{\varepsilon}u^{\varepsilon}$ almost everywhere in $\Omega$ and we have that  $\widetilde{u}^{\varepsilon} \rightharpoonup u^*$ in $L^2(\Omega )$ and $P_{\varepsilon}u^{\varepsilon} \rightarrow u_0 $ in $L^2_{per}(\Omega )$, it follows that     $\chi^{\varepsilon} P_{\varepsilon}u^{\varepsilon} \rightharpoonup \gamma u_0 $ in $L^2(\Omega )$ and therefore, 
		$$
		u^*(\varphi_1, \theta_1) = \gamma( \theta_1) u_0 (\varphi_1) \ \ \textrm{a. e. } \ \Omega.
		$$
	\end{remark}

	\subsection{ Limit in  $\tau^{\varepsilon}$.}
	
		\begin{lemma} When $\varepsilon \rightarrow 0$ we obtain
		\begin{equation}\label{c}
			\tau^{\varepsilon} \rightharpoonup \varphi_1 \ \text{in} \ H^1(\Omega)  \ \text{and strongly  in } \ L^p(\Omega)  \ \text{for all} \  p \geq 1 \  \ \text{and} \  \ 
		\frac{\partial \tau^{\varepsilon} }{\partial \theta_1} \rightarrow 0 \ \text{in}\  \ L^2(\Omega).
	\end{equation}

	\end{lemma}
	
	\begin{proof}
	From the definition of   $\tau^{\varepsilon}$, we have:
		\begin{align*}
		&\int_{A_k^{\varepsilon}} \vert \tau^{\varepsilon} - \varphi_1 \vert^2 \cos(\varepsilon \theta_1)   \  d\varphi_1 d\theta_1 =\int_{A_k^{\varepsilon}} \left\vert \varepsilon \left( P_{\varepsilon}X^{\varepsilon}  \left( \frac{  \varphi_1 - \varepsilon kL}{\varepsilon}  , \theta_1   \right)   \right) \right\vert^2 \cos(\varepsilon \theta_1)   \  d\varphi_1 d\theta_1   \\
		&= \int_Y \left\vert \varepsilon  P_{\varepsilon}X^{\varepsilon}(y,z)  \right\vert^2 \varepsilon \cos(\varepsilon z)   \  dy dz 
		\leq  \int_{Y^*} C \varepsilon^3 \left\vert  X^{\varepsilon}(y,z)  \right\vert^2  \cos(\varepsilon z)   \  dy dz, 
	\end{align*}
	where the last inequality comes from Lemma \ref{lema}. Then, 
	\begin{align*}
		\int_{\Omega} \vert \tau^{\varepsilon} - \varphi_1 \vert^2 \cos(\varepsilon \theta_1)   \  d\varphi_1 d\theta_1  =  & \displaystyle\sum_{k=1}^{\frac{1}{\varepsilon L}} \int_{Y^*}  C \varepsilon^3 \left\vert  X^{\varepsilon}(y,z)  \right\vert^2  \cos(\varepsilon z)   \  dy dz + o(\varepsilon) \\ 
		\\ = & \frac{\varepsilon^2}{ L}  \int_{Y^*}  C  \left\vert  X^{\varepsilon}(y,z)  \right\vert^2  \cos(\varepsilon z)   \  dy dz \rightarrow 0,
	\end{align*}
	as $\varepsilon \rightarrow 0$, since, we know from Lemma \ref{auxiliar}  that $\left\Vert  X^{\varepsilon}(y,z)  \right\Vert_{H^1(Y^*)} < K$ $\forall\varepsilon\in [0,\varepsilon]$.  Similarly, 
	\begin{align*}
		&\int_{A_k^{\varepsilon}} \left\vert \frac{\partial}{\partial \varphi_1} (\tau^{\varepsilon} - \varphi_1) \right\vert^2 \cos(\varepsilon \theta_1)   \  d\varphi_1 d\theta_1  
		=\int_{A_k^{\varepsilon}} \left\vert  \varepsilon \frac{\partial P_{\varepsilon}X^{\varepsilon}}{\partial \varphi_1} \left( \frac{  \varphi_1 - \varepsilon kL}{\varepsilon}  , \theta_1   \right) \right\vert^2 \cos(\varepsilon \theta_1)   \  d\varphi_1 d\theta_1 \\  
		&= \int_Y \varepsilon \left\vert  \frac{\partial P_{\varepsilon}X^{\varepsilon}}{\partial y}(y,z)  \right\vert^2  \cos(\varepsilon z)   \  dy dz 
		\leq  \varepsilon \int_{Y^*} C  \left\vert  \frac{\partial X^{\varepsilon}}{\partial y}(y,z)  \right\vert^2  \cos(\varepsilon z)   \  dy dz,
	\end{align*}
	
	and 
	\begin{align*}
		\int_{A_k^{\varepsilon}} \left\vert \frac{\partial}{\partial \theta_1} (\tau^{\varepsilon} - \varphi_1) \right\vert^2 \cos(\varepsilon \theta_1)   \  d\varphi_1 d\theta_1  
		=& \int_{A_k^{\varepsilon}} \left\vert  \varepsilon \frac{\partial P_{\varepsilon}X^{\varepsilon}}{\partial \theta_1} \left( \frac{  \varphi_1 - \varepsilon kL}{\varepsilon}  , \theta_1   \right)  \right\vert^2 \cos(\varepsilon \theta_1)   \  d\varphi_1 d\theta_1 
		\\ \leq & \varepsilon^3 \int_{Y^*} C  \left\vert  \frac{\partial X^{\varepsilon}}{\partial z}(y,z)  \right\vert^2  \cos(\varepsilon z)   \  dy dz. 
	\end{align*}

	We also have, for all $\varepsilon>0$, with $\widetilde{C} = \displaystyle\frac{1}{\varepsilon L}\varepsilon C =\displaystyle\frac{C}{ L}$ that
	
	\begin{align*}
		\int_{\Omega} \left\vert \frac{\partial}{\partial \varphi_1} (\tau^{\varepsilon} - \varphi_1) \right\vert^2 \cos(\varepsilon \theta_1)   \  d\varphi_1 d\theta_1 = & \displaystyle\sum_{k=1}^{\frac{1}{\varepsilon L}} \int_{Y^*}  C \varepsilon \left\vert \frac{\partial X^{\varepsilon}}{\partial y} (y,z)  \right\vert^2  \cos(\varepsilon z)   \  dy dz  + o(\varepsilon)\\ 
		 =  & \widetilde{C}  \int_{Y^*}    \left\vert  \frac{\partial X^{\varepsilon}}{\partial y} (y,z)  \right\vert^2  \cos(\varepsilon z)   \  dy dz + o(\varepsilon), 
	\end{align*}
	
	and 
	\begin{align*}
		\int_{\Omega} \left\vert \frac{\partial}{\partial \theta_1} (\tau^{\varepsilon} - \varphi_1) \right\vert^2 \cos(\varepsilon \theta_1)   \  d\varphi_1 d\theta_1 \leq  & \displaystyle\sum_{k=1}^{\frac{1}{\varepsilon L}} \int_{Y^*}  C \varepsilon^3 \left\vert \frac{\partial X^{\varepsilon}}{\partial z} (y,z)  \right\vert^2  \cos(\varepsilon z)   \  dy dz 
		\\ = & \varepsilon^2   \int_{Y^*} \widetilde{C}   \left\vert  \frac{\partial X^{\varepsilon}}{\partial z} (y,z)  \right\vert^2  \cos(\varepsilon z)   \  dy dz \rightarrow 0,
	\end{align*}     
	as $\varepsilon \rightarrow 0$. Furthermore, as $A \leq \cos(\varepsilon \theta_1)$, we have:
	$$
	\int_{\Omega} \left\vert \frac{\partial}{\partial \theta_1} (\tau^{\varepsilon} - \varphi_1) \right\vert^2  A  \  d\varphi_1 d\theta_1 \leq \int_{\Omega} \left\vert \frac{\partial}{\partial \theta_1} (\tau^{\varepsilon} - \varphi_1) \right\vert^2 \cos(\varepsilon \theta_1)   \  d\varphi_1 d\theta_1 \rightarrow 0,
	$$
	as $ \varepsilon \rightarrow 0$.   
	That is, $ \frac{\partial \tau^{\varepsilon} }{\partial \theta_1} \rightarrow 0 $ in  $L^2(\Omega)$. Arguing in a similar way, we obtain the others convergences.

\end{proof}
	
	\subsection{ Limit of $\widetilde{\frac{\partial \tau^\varepsilon}{\partial \varphi_1}} $.}
	
	\begin{lemma}
		Let $\chi$ be the characteristic function of the set $Y^*$ and
		let $\widetilde{\frac{\partial X^0}{\partial y}}$ denote the zero extension
		of $\frac{\partial X^0}{\partial y}$.
		Then, for almost every $\theta_1\in(0,g_1)$ when $\varepsilon\rightarrow 0$ we obtain
		\[
		\widetilde{\frac{\partial \tau^\varepsilon}{\partial \varphi_1}} (\cdot,\theta_1)
		\stackrel{*}{\rightharpoonup}
		q(\theta_1)
		\quad \text{ in } L^{\infty}(0,2\pi),
		\]
		where
		$
		q(\theta_1)
		=
		\frac{1}{L}\int_0^L
		\left(
		1-\frac{\widetilde{\partial X^0}}{\partial y}(s,\theta_1)
		\right)
		\chi(s,\theta_1)\,ds
		$. Moreover,
		\[
		\widetilde{\frac{\partial \tau^\varepsilon}{\partial \varphi_1}} 
		\cos(\varepsilon\theta_1)
		\stackrel{*}{\rightharpoonup}
		q
		\quad \text{in } L^{\infty}(\Omega), \text{when} \  \varepsilon\rightarrow 0.
		\]
		
	\end{lemma}
	
	\begin{proof}
	From the definition of $\frac{\partial \tau^\varepsilon}{\partial \varphi_1}$, it follows that 
	$$\widetilde{\frac{\partial \tau^\varepsilon}{\partial \varphi_1}} (\varphi_1, \theta_1) =  \chi\left( \frac{\varphi_1}{\varepsilon}, \theta_1 \right) -  \frac{\widetilde{ \partial X^{0}}}{\partial y} \left( \frac{\varphi_1}{\varepsilon}, \theta_1 \right).$$
	Therefore, since $\frac{  \partial X^{0}}{\partial y}\in L^{\infty}(Y^*),$ it follows from Proposition~\ref{teochi} that

	$$
	\widetilde{\frac{\partial \tau^\varepsilon}{\partial \varphi_1} } (\cdot, \theta_1) 	\stackrel{*}{\rightharpoonup} \frac{1}{L} \int_0^L \left( 1- \frac{\widetilde{ \partial X^{0}}}{\partial y} (s, \theta_1)\right) \chi (s, \theta_1) ds:= q(\theta_1) \ \ \text{in} \  L^{\infty}(0,2\pi).
	$$
	Hence, arguing as (\ref{4.19}) we prove that
	$
	\widetilde{\frac{\partial \tau^\varepsilon}{\partial \varphi_1} }  \cos(\varepsilon \theta_1) 	\stackrel{*}{\rightharpoonup} q \ \ \text{in}  \  L^{\infty}(\Omega ).
	$
	\end{proof}

%
	
	\subsection{Proof of the main theorem}
	\begin{proof}
First  using \eqref{4.19}, \eqref{4.25}, \eqref{c} we have, when $\varepsilon \rightarrow 0$
	$$
	\int_{\Omega^{\varepsilon}}   u^\varepsilon \phi \tau^{\varepsilon} \cos(\varepsilon \theta_1) d \varphi_1 d \theta_1 = \int_{\Omega} \chi^{\varepsilon}   P_{\varepsilon}u^\varepsilon \phi \tau^{\varepsilon} \cos(\varepsilon \theta_1)  d \varphi_1 d \theta_1 \rightarrow  \int_{\Omega} \gamma   u_0 \phi \varphi_1   d \varphi_1 d \theta_1,
	$$
	for all   periodic function $ \phi \in \mathcal{C}^{\infty}(0, 2 \pi)$,  such that  $ \phi   $ vanishes in a neighborhood of $0$ and $2 \pi$.

	Hence, passing to the limit in \eqref{4.15} and  \eqref{4.16} we obtain respectively:  
	\begin{equation}\label{4.31}
		\int_{\Omega} \left(   \xi^* \frac{\partial  }{\partial \varphi_1} (\phi \varphi_1) - \xi^*  \phi     -     q  \frac{\partial \phi }{\partial \varphi_1} u_0 + \gamma    u_0 \varphi_1  \phi \right)  \ d\theta_1 d\varphi_1    \\ =\int_{0}^{2 \pi}  \hat{f}(\varphi_1)   \phi \varphi_1 \ d\varphi_1, 
	\end{equation}
	and
	\begin{equation}\label{eq limite 2}
		\int_{\Omega} \left(    \xi^* \frac{\partial \psi }{\partial \varphi_1}   + \gamma  u_0  \psi \right)  \ d\theta_1 d\varphi_1   =\int_{0}^{2 \pi}  \hat{f}\psi \ d\varphi_1, 
	\end{equation} 
	for all periodic  $ \phi \in \mathcal{C}^{\infty}(0, 2 \pi)$,  such that  $ \phi   $ vanishes in a neighborhood of $0$ and $2 \pi$ and  $\psi \in H_{per}^1(0,2 \pi)$.
	

	Taking $\psi(\varphi_1) = \phi(\varphi_1) \varphi_1$ as  a test function we obtain 
	
	\begin{equation}\label{4.34}
		\int_{\Omega} \left(   \xi^* \frac{\partial  }{\partial \varphi_1} (\phi \varphi_1)   + \gamma   u_0  \phi \varphi_1 \right) \ d\theta_1 d\varphi_1   =\int_{0}^{2 \pi}  \hat{f} \phi \varphi_1 \ d\varphi_1. 
	\end{equation}
	
	Applying integration by parts to the difference between  \eqref{4.34} and \eqref{4.31}, we obtain 
	
	\begin{equation}\label{4.35}
		0 = \int_{\Omega} \left(  \xi^*  \phi   +    q  \frac{\partial \phi }{\partial \varphi_1} u_0 \right)  \ d\theta_1 d\varphi_1  = \int_{\Omega} \left(  \xi^*  \phi   -     q \phi  \frac{\partial  u_0 }{\partial \varphi_1} \right)  \ d\theta_1 d\varphi_1. 
	\end{equation}
	
	Defining
	$$
	\hat{q} := \int_0^{g_1}  q(s) ds =  \frac{1}{L} \int_{Y^*} \left( 1- \frac{\partial X^{0}}{\partial y} (y, z)\right) dy dz,
	$$

	and computing the iterated integral in \eqref{4.35}, it follows that
	$$
	\int_0^{2 \pi}  \phi  \left( \int_0^{g_1}   \xi^*(\varphi_1,\theta_1)  \ d\theta_1 -    \hat{q}  \frac{\partial  u_0 }{\partial \varphi_1}  \right)   d\varphi_1 =0,
	$$
	for all periodic function $ \phi \in \mathcal{C}^{\infty}(0, 2 \pi)$,  such that  $ \phi   $ vanishes in a neighborhood of $0$ and $2 \pi$. This implies that  
	\begin{equation*}
		\int_0^{g_1}   \xi^*(\varphi_1,\theta_1)  \ d\theta_1 =    \hat{q}  \frac{\partial  u_0 }{\partial \varphi_1} \ \ \textrm{a. e.} \ \ \varphi_1\in (0, 2 \pi).
	\end{equation*}
	
	Remembering that $ L  \displaystyle\int_0^{g_1} \gamma(\theta_1) \ d\theta_1 = \vert Y^*\vert $ and applying  Fubini's Theorem in equation \eqref{eq limite 2}, we obtain, 	for all $\psi \in H^1_{per}(0, 2 \pi)$ that
	
$$
		\int_{0}^{2 \pi}  \hat{f}\psi \ d\varphi_1 = \int_0^{2 \pi} \left\lbrace  \left( \int_0^{g_1}    \xi^*(\varphi_1, \theta_1 )  \ d\theta_1  \right) \frac{\partial \psi }{\partial \varphi_1}   +  \frac{\vert Y^* \vert}{L}    u_0  \psi \right\rbrace d\varphi_1    
		= \int_0^{2 \pi}   \left(\hat{q}  \frac{\partial  u_0 }{\partial \varphi_1}  \frac{\partial \psi }{\partial \varphi_1}  + \frac{\vert Y^* \vert}{L}    u_0  \psi \right)   d\varphi_1, 
	$$

	Defining  $q_0 = \frac{L}{\vert Y^* \vert}\hat{q}$, that is, 
	\begin{equation}\label{4.39}
		q_0 =  \frac{1}{\vert Y^* \vert} \int_{Y^*} \left( 1- \frac{\partial X^{0}}{\partial y} (y, z)\right) dy dz,
	\end{equation}
	and also
	$$
	f_0(\varphi_1) = \frac{L}{\vert Y^* \vert}\hat{f}(\varphi_1),
	$$
	we have
	$$
	\int_0^{2 \pi}   \left(q_0  \frac{\partial  u_0 }{\partial \varphi_1}  \frac{\partial \psi }{\partial \varphi_1}  +   u_0  \psi \right)   d\varphi_1 =\int_{0}^{2 \pi}  f_0 \psi \ d\varphi_1, \ \ \forall \psi \in H_{per}^1(0,2 \pi),
	$$
	which is the weak formulation of 
	\begin{equation}\label{4.42}
		\left\{\begin{array}{cll}
			- q_0 \ \frac{\partial^2  u_0 }{\partial \varphi_1^2}   +  u_0  &= f_0 \ \ \textit{in}\ (0,2 \pi); 	  \\
			u_0 \in H_{per}^1(0,2 \pi).
		\end{array}\right.
	\end{equation}
	

	Now we will show that  problem (\ref{4.42}) is well-posed in the sense that the coefficient $q_0$ is greater than zero. To do this, we use the variational formulation of the auxiliary problem (\ref{probaux}), which for each $\varepsilon$ is the bilinear form

	\begin{equation*}
		b_{\varepsilon} (\phi, \psi) = \int_{Y^*} \left( \frac{1}{\cos^2(\varepsilon z)} \frac{\partial \phi }{\partial y} , \frac{\partial \phi }{\partial z} \right) \cdot \left(   \frac{\partial \psi }{\partial y} , \frac{\partial \psi }{\partial z} \right) \cos(\varepsilon z)  \ dy dz. 
	\end{equation*}

	So, for all $\phi \in H^1_{per} (Y^*)$ which is  $L$-periodic in the $ \varphi_1$ variable we obtain that $X^{\varepsilon}$ satisfies
	$$
	b_{\varepsilon}(X^{\varepsilon}, \phi)= \int_{B_1} \frac{g'(y)}{\cos(\varepsilon z) \sqrt{1 + g'(y)^2}}  \phi \ dS,
	$$
	where $B_1$ is the upper bondary of $Y^*$.  Consequently, $y-X^{\varepsilon}$ satisfies
	
	\begin{equation}\label{4.43}
		b_{\varepsilon}(y-X^{\varepsilon}, \phi)= \int_{B_1} \frac{g'(y)}{\cos(\varepsilon z) \sqrt{1 + g'(y)^2}}  \phi  \ dS  - \int_{B_1} \frac{g'(y)}{\cos(\varepsilon z) \sqrt{1 + g'(y)^2}}  \phi \ dS  = 0,
	\end{equation}
	for all $\varepsilon$ greater than $0$ and $\phi \in H^1_{per}(Y^*)$.  Also, we have by equation \eqref{4.39} that
	
	\begin{equation}\label{4.44}
		q_0 \vert Y^* \vert =  \int_{Y^*} \left( 1- \frac{\partial X^{0}}{\partial y} \right) dy dz =  \int_{Y^*} \frac{\partial}{\partial y} (y - X^{0})  \frac{\partial y}{\partial y} dy dz = b_0(y - X^{0}, y). 
	\end{equation}
	
	Thus, from equation \eqref{4.43} with $\phi = - X^{0}$,  \eqref{4.44} and the periodicity of  $X^{0}$, we obtain
	
	$$
	q_0 \vert Y^* \vert = b_0(y - X^{0}, y) + b_0(y-X^{0}, - X^{0}) = \Vert \nabla(y-X^{0}) \Vert^2_{L^2(Y^*)}>0. 
	$$
	The inequality is strict because if $ \Vert \nabla(y-X^{0}) \Vert^2_{L^2(Y^*)} = 0 $, it would imply that there exists $c \in \mathbb{R}$ such that $X^{0}+c=y$, which is a contradiction because  $X^{0}+c$ is periodic and $y$ is not. Furthermore, since $\vert Y^* \vert > 0 $ we have $q_0>0$ and the above problem is well-posed. In particular, we obtain the uniqueness of the solution to the problem  (\ref{4.42}).
	\end{proof}

	\begin{remark}
		We may consider a more general version of  problem  \eqref{eqO} with a potential term $V^{\varepsilon} u^{\varepsilon}$, which depends on the parameter $\varepsilon$,  instead of considering only $ u^{\varepsilon}$. Thus, we consider the following problem:
		
		\begin{equation*}
			\left\{\begin{array}{lll}
				-  \frac{1}{ \varepsilon^2 \cos (\varepsilon \theta_1 )} \frac{\partial}{\partial \theta_1} \left(\cos (\varepsilon \theta_1 ) \frac{\partial}{\partial \theta_1} \right) u^{\varepsilon} - \frac{1}{ \cos^2 (\varepsilon \theta_1 )} \frac{\partial^2 u^{\varepsilon}}{\partial \varphi_1^2} + V^{\varepsilon} u^{\varepsilon}  = f^\varepsilon \ \ \textit{in}\ \Omega^{\varepsilon}; 	  \\
				\frac{1}{ \cos^2 (\varepsilon \theta_1 )} \frac{\partial u^{\varepsilon}}{\partial \varphi_1} N_1^{\varepsilon} + \frac{1}{ \varepsilon^2} \frac{\partial u^{\varepsilon} }{\partial \theta_1} N_2^{\varepsilon}  =0 \ \ \textit{on} \ \ \partial \Omega^{\varepsilon}; \\ 
				u^{\varepsilon}(\cdot, \theta_1)\ \ 2\pi- \textit{periodic}, 
			\end{array}\right.
		\end{equation*}
		where $N^\varepsilon= (N^\varepsilon_1, N^\varepsilon_2)$ is the normal vector fields on the upper and lower boundary given by $N^\varepsilon = \left( - \frac{g'\left(\frac{ \varphi_1}{\varepsilon}\right)}{\sqrt{1+g'\left(\frac{ \varphi_1}{\varepsilon}\right)^2}}, \frac{1}{\sqrt{1+g'\left(\frac{ \varphi_1}{\varepsilon}\right)^2}}\right) $  and  $N^\varepsilon = \left( 0 , -1 \right)$, respectively. 
		We  assume $f^\varepsilon\in L^2(\Omega^\varepsilon, \cos (\varepsilon\theta_1))$ such that $\Vert f^\varepsilon \Vert_{L^2(\Omega^{\varepsilon},\cos (\varepsilon\theta_1))} \leq C$ independent of  $ \varepsilon$ and $V^{\varepsilon} \in L^{\infty}(\Omega^{\varepsilon})$,  $V^{\varepsilon}(\cdot,\theta_1), \ \ 2\pi\text{-periodic}, V^{\varepsilon} \geq 1$ and there exists $V_0 \in L^{\infty}(\Omega)$ independent of the  $\theta_1$  (that is, $V_0(\varphi_1, \theta_1) = V_0(\varphi_1)$)  such that
		\begin{equation*}
			\int_{\Omega^{\varepsilon}} \vert V^{\varepsilon} - V_0\vert^p \cos( \varepsilon \theta_1 ) d \varphi_1 d \theta_1 \underset{\varepsilon \to 0}{\longrightarrow} 0,
		\end{equation*} 
		for some $p>1$.	Under these hypotheses, the same conclusions as before hold. 
	\end{remark}

	\section{Higher Order Multiscale Expansion of solutions}
	In this section, we employ a suitable corrector technique developed by Bensoussan, Lions, and Papanicolaou \cite{BLP78} to establish strong convergence results and derive error estimates in the $H^1-$norm. To this end we use again the multiple scale method.  Recalling the formal expansion presented in $(\ref{exp})$ and  the homogenized  equation  $(\ref{eqlimitemultiplasescalas})$,
	we can rewrite problem  $(\ref{4})$ as follows:
	
	\begin{equation}\label{21}
		\left\{\begin{array}{l}
			- div_{y,z} \left(  \nabla_{y,z}  w_2 (x,y,z) -\frac{d^2 w_0}{d x^2} (x) \left(\begin{array}{c}
				X \\ 0 \end{array}\right)  \right)  = \left(  1- q_0 -  \frac{\partial}{\partial y} X \right)  \frac{d^2 w_0}{d x^2} \  \ \textrm{in} \ \ Y^*;  \\
			\frac{\partial w_2}{\partial N}
			=  - \frac{g'(y)}{\sqrt{1+g'(y)^2}} X \frac{d^2 w_0}{d x^2} \  \ \textrm{on} \ \ B_1;  \\
			\frac{\partial w_2}{\partial N}
			=  0 \  \ \textrm{on} \ \ B_2;  \\
			w_2(x, \cdot, z) \  \ L-\textrm{periodic}.
		\end{array}\right.
	\end{equation}
	
	The linearity of \eqref{21}  together with the  fact that $\frac{d^2 w_0}{d x^2}$ does not  depend on the  variables  $y$ and $z$ suggest that we look for $w_2(x,y,z)$ of the form
	\begin{equation}\label{22}
		w_2(x,y,z) = \Theta(y,z) \frac{d^2 w_0}{d x^2}(x) \  \ \forall x \in (0, 2 \pi) \  \ \textrm{and }  \  \ (y,z) \in Y^*,
	\end{equation}
	where $\Theta$ is the solution of the auxiliary problem 
	
	\begin{equation}\label{23}
		\left\{\begin{array}{l}
			- div_{y,z} \left(  \nabla_{y,z} \Theta (y,z) -\left(\begin{array}{c}
				X(y,z)  \\ 0 \end{array}\right)  \right)  =   1- q -  \frac{\partial}{\partial y} X (y,z)  \  \ \textrm{in} \ \ Y^*  ;\\
			\left(  \nabla_{y,z} \Theta (y,z) -\left(\begin{array}{c}
				X(y,z)  \\ 0 \end{array}\right)  \right)  \cdot N  = 0 \  \ \textrm{on} \ \ B_1  \cup B_2; \\
			\Theta ( \cdot, z) \  \ L-\textrm{periodic}.
		\end{array}\right.
	\end{equation}
	
	Then, we use  \eqref{15} and \eqref{22} to introduce the following asymptotic expansion for \eqref{eqR}:
	
	\begin{equation*}
		w^{\varepsilon}(\varphi,\theta) = w_0(\varphi) - \varepsilon X \left( \frac{\varphi}{\varepsilon},\frac{\theta}{\varepsilon}  \right) \frac{d w_0}{d \varphi} ( \varphi) + \varepsilon^2  \Theta \left( \frac{\varphi}{\varepsilon},\frac{\theta}{\varepsilon}  \right) \frac{d^2 w_0}{d \varphi^2} ( \varphi) + \cdots    
	\end{equation*}
	
	Now we are able to define the correctors:
	
	\begin{definition}\label{def corretores} According to \cite{BLP78}, the functions $X$ and $\Theta$ define the first-order correctors
		\begin{equation*}
			\kappa^{\varepsilon}(\varphi,\theta): = - \varepsilon X \left( \frac{\varphi}{\varepsilon},\frac{\theta}{\varepsilon}  \right) \frac{d w_0}{d \varphi} ( \varphi) \ \ \ (\varphi,\theta) \in R^{\varepsilon},
		\end{equation*}
		and the second-order correctors   
		\begin{equation*} 
			\mu^{\varepsilon}(\varphi,\theta): = - \varepsilon X \left( \frac{\varphi}{\varepsilon},\frac{\theta}{\varepsilon}  \right) \frac{d w_0}{d \varphi} ( \varphi) + \varepsilon^2  \Theta \left( \frac{\varphi}{\varepsilon},\frac{\theta}{\varepsilon}  \right) \frac{d^2 w_0}{d \varphi^2} ( \varphi)  \ \ \ (\varphi,\theta) \in R^{\varepsilon}.
		\end{equation*}
	\end{definition}
	
	\begin{remark}\label{5} The functions  $X$ and $\Theta$ are originally  defined in the  basic cell  $Y^*$, but to consider these  functions in the thin  domain  $R^{\varepsilon}$, we use their periodicities at $y$ to extend them to the band 
		$
		Y =\{ (y,z)\in \mathbb{R}^2; y\in \mathbb{R}, 0<z<g(y)\},
		$
		and we compose them with the diffeomorphisms  
		$$
		T^{\varepsilon}: R^{\varepsilon} \mapsto Y : (\varphi, \theta) \to \left(\frac{\varphi}{\varepsilon},\frac{\theta}{\varepsilon}\right).  
		$$
	\end{remark}
	
	In the analysis below, with some abuse of notation we will denote these compositions by  $X\left(\frac{\varphi}{\varepsilon},\frac{\theta}{\varepsilon}\right)$ and $ \Theta\left(\frac{\varphi}{\varepsilon},\frac{\theta}{\varepsilon}\right)$ everywhere for $(\varphi, \theta) \in R^{\varepsilon}$. With these considerations we can obtain some estimates on $ R^{\varepsilon}$ for $X$ and $\Theta$. We have that:  
	
	\begin{proposition}\label{prop estimativas}

		\begin{equation}\label{27}
			\left\Vert X \right\Vert_{L^2(R^{\varepsilon}, \cos \theta)}^2    \leq  \varepsilon/ L \left\Vert X \right\Vert_{L^2(Y^*)}^2, 
		\end{equation}  
		\begin{equation*} \left\Vert \Theta \right\Vert_{L^2(R^{\varepsilon}, \cos \theta)}^2  \leq \varepsilon/ L \left\Vert \Theta \right\Vert_{L^2(Y^*)}^2, \end{equation*} 
		\begin{equation*} \left\Vert \partial_y X \right\Vert_{L^2(R^{\varepsilon}, \cos \theta)}^2  \leq \varepsilon/ L \left\Vert  \partial_y X \right\Vert_{L^2(Y^*)}^2,\end{equation*} 
		\begin{equation}\label{28}\left\Vert  \partial_z X \right\Vert_{L^2(R^{\varepsilon}, \cos \theta)}^2  \leq \varepsilon/ L \left\Vert  \partial_z X \right\Vert_{L^2(Y^*)}^2, \end{equation} 
		\begin{equation*} \left\Vert  \partial_y \Theta \right\Vert_{L^2(R^{\varepsilon}, \cos \theta)}^2  \leq \varepsilon/ L \left\Vert \partial_y \Theta \right\Vert_{L^2(Y^*)}^2, \end{equation*} 
		\begin{equation*} \left\Vert \partial_z \Theta \right\Vert_{L^2(R^{\varepsilon}, \cos \theta)}^2  \leq \varepsilon/ L \left\Vert \partial_z \Theta\right\Vert_{L^2(Y^*)}^2, \end{equation*}            
		\begin{equation*} \left\Vert \partial_y^2 X \right\Vert_{L^2(R^{\varepsilon}, \cos \theta)}^2  \leq   \varepsilon/ L    \left\Vert \partial_y^2 X \right\Vert_{L^2(Y^*)}^2,\end{equation*} 
		\begin{equation*}  \left\Vert \partial_y^2 \Theta \right\Vert_{L^2(R^{\varepsilon}, \cos \theta)}^2  \leq \varepsilon/ L    \left\Vert \partial_y^2 \Theta \right\Vert_{L^2(Y^*)}^2. \end{equation*} 
		
	\end{proposition}
	
	\begin{proof}
		For the first one, we have
		\begin{eqnarray*}
		\left\Vert X \right\Vert_{L^2(R^{\varepsilon}, \cos \theta)}^2 &= &\int_{R^{\varepsilon}}  \left\vert X \left(\frac{\varphi}{\varepsilon},\frac{\theta}{\varepsilon}\right) \right\vert^2 \ \cos \theta \ d \varphi d \theta \leq \int_{R^{\varepsilon}}  \left\vert X \left(\frac{\varphi}{\varepsilon},\frac{\theta}{\varepsilon}\right) \right\vert^2 \  \ d \varphi d \theta \\
		&\leq &  \sum_{k=1}^{1/\varepsilon L}  \varepsilon^2 \int_{Y^*} \left\vert X \left(y,z\right) \right\vert^2 \  \ d y d z 
		\leq \varepsilon/ L \left\Vert X \right\Vert_{L^2(Y^*)}^2 .
		\end{eqnarray*}

		Similarly we can obtain the estimate for $\Theta$ and the first derivatives.
%
%
		Next, we shall estimate the second order derivatives of $X$.  By [Proposition 7.7,  \cite{Tay96}],
		$\frac{\partial^2 X}{\partial y^2 } \in L^2(Y^* )$. Therefore
		
		\begin{eqnarray*}
			\left\Vert  \partial_y^2 X \right\Vert_{L^2(R^{\varepsilon}, \cos \theta)}^2 &= &\int_{R^{\varepsilon}}  \left\vert \partial_y^2 X \left(\frac{\varphi}{\varepsilon},\frac{\theta}{\varepsilon}\right) \right\vert^2 \ \cos \theta \ d \varphi d \theta 
			\leq \int_{R^{\varepsilon}}  \left\vert \partial_y^2 X \left(\frac{\varphi}{\varepsilon},\frac{\theta}{\varepsilon}\right) \right\vert^2 \  \ d \varphi d \theta \\
			&\leq &\sum_{k=1}^{1/\varepsilon L}  \varepsilon^2 \int_{Y^*} \left\vert \partial_y^2 X \left(y,z\right) \right\vert^2 \  \ d y d z 
			\leq \varepsilon/ L \left\Vert \partial_y^2 X \right\Vert_{L^2(Y^*)}^2 .
		\end{eqnarray*}

		
	Using an analogous argument, we obtain the  remain estimatives.
%
		
	\end{proof}

	\subsection{Variational formulation of the problem}
	
	Now, we adapt the methods presented in \cite{PS13} to our setting. Using an appropriate corrector approach, we aim to demonstrate strong convergence.  Since $R^{\varepsilon}$ depends on a positive parameter  $\varepsilon$ and, as $\varepsilon$ tends to  $0$, the domains  $R^{\varepsilon}$ collapse onto the interval $(0, 2 \pi)$, we rescale the Lebesgue measure by $\frac{1}{\varepsilon}$ in order to preserve the relative capacity of a measurable set  $\mathcal{O}\subset R^{\varepsilon} $. This leads to the singular measure 
	$$
	\rho_{\varepsilon}(\mathcal{O}) = \varepsilon^{-1} \vert \mathcal{O}\vert.
	$$
	This measure has been considered in studies involving thin domains, e.g. \cite{HR92}  and allows us to introduce the Lebesgue space  $L^2(R^{\varepsilon}, \cos \theta; \rho_{\varepsilon} )$ and the Sobolev space $H^1_{per}(R^{\varepsilon}, \cos\theta; \rho_{\varepsilon} )$. The norms in these spaces will be  denoted by  $||| \cdot |||_{L^2(R^{\varepsilon}, \cos\theta)}$ and $||| \cdot |||_{H^1_{per}(R^{\varepsilon}, \cos\theta)}$, respectively, being introduced by the inner products:
	$$
	(u,v)_{\varepsilon}=\varepsilon^{-1} \int_{R^{\varepsilon}} uv \cos \theta d \varphi d \theta, \ \ \ \forall u, v \in L^2(R^{\varepsilon}, \cos \theta),
	$$
	and
	$$
	a_{\varepsilon}(u,v):=\varepsilon^{-1} \int_{R^{\varepsilon}} \left( \frac{\partial u^{\varepsilon}}{\partial  \theta} \frac{\partial v}{\partial  \theta}    + \frac{1}{ \cos^2 \theta }\frac{\partial u^{\varepsilon}}{\partial \varphi} \frac{\partial v }{\partial \varphi}  + uv  \right) \cos \theta  d \varphi d \theta, \ \ \ \forall u, v \in H^1_{per}(R^{\varepsilon}, \cos \theta),
	$$  
	respectively. 
	
	\begin{remark} Observe that the norms $||| \cdot  ||| $ are equivalent to the usual norms of the spaces  $L^2(R^{\varepsilon}, \cos \theta)$ and $H^1_{per}(R^{\varepsilon}, \cos \theta)$. Moreover, these norms are related by: 
		$$
		||| u |||_{L^2(R^{\varepsilon}, \cos \theta)} = \varepsilon^{-1/2} || u ||_{L^2(R^{\varepsilon}, \cos \theta)}, \ \ \ \ \forall u \in L^2(R^{\varepsilon}, \cos \theta), 
		$$
		
		$$
		||| u |||_{H^1_{per}(R^{\varepsilon}, \cos \theta)} = \varepsilon^{-1/2} || u ||_{H^1_{per}(R^{\varepsilon}, \cos \theta)}, \ \ \ \ \forall u \in H^1_{per}(R^{\varepsilon}, \cos \theta).
		$$  
	\end{remark}
	
	The variational formulation of  \eqref{eqR} is  to find $w^{\varepsilon} \in H^1_{per}(R^{\varepsilon}, \cos \theta) $ such that 
	\begin{equation*} 
		\int_{R^{\varepsilon}} \left(  \frac{\partial w^{\varepsilon}}{\partial  \theta} \frac{\partial \psi}{\partial  \theta}    + \frac{1}{ \cos^2  \theta }\frac{\partial w^{\varepsilon}}{\partial \varphi} \frac{\partial \psi }{\partial \varphi}      + w^{\varepsilon}  \psi \right) \cos   \theta \ d\varphi d\theta  =  \int_{R^{\varepsilon}} f^\varepsilon\psi \cos \theta \  d\varphi d\theta,
	\end{equation*}
	for all $\psi \in H^1_{per}(R^{\varepsilon}, \cos \theta)$, where $\psi$ and $w^{\varepsilon}$ satisfies the boundaries condition given in \eqref{eqR}. This variational formula is  equivalent to find  $w^{\varepsilon} \in H^1_{per}(R^{\varepsilon}, \cos \theta; \rho^{\varepsilon}) $ such that

	\begin{equation}\label{4,1}
		a_{\varepsilon}(\psi, w^{\varepsilon} ) = (\psi, f^\varepsilon )_{\varepsilon } \ \ \ \ \forall \psi \in H^1_{per}(R^{\varepsilon}, \cos \theta; \rho^{\varepsilon}).
	\end{equation}

	Observe that  solutions $w^{\varepsilon}$ satisfy a priori  estimates uniformly with respect to  $\varepsilon$. In fact, by choosing $ \psi = w^{\varepsilon} $ in 
	\eqref{4,1}, we obtain, denoting momentarily $|||\cdot |||_{L^2(R^{\varepsilon}, \cos \theta)}$ by $|||\cdot|||_{L^2}$:
	
	\begin{equation*}
		\begin{aligned}
				\left| \left|   \left|   \frac{\partial w^{\varepsilon}}{\partial  \theta}  \right|\right|\right|^2_{L^2} +  \left| \left|   \left| \frac{1}{ \cos  \theta }\frac{\partial w^{\varepsilon}}{\partial \varphi} \right|\right|\right|^2_{L^2} + \left| \left|   \left| w^{\varepsilon}  \right|\right|\right|^2_{L^2}\leq |||   f^\varepsilon|||_{L^2} ||| w^{\varepsilon}  || |_{L^2} 
		\end{aligned}
	\end{equation*}

	To capture the limiting behavior of $a_{\varepsilon}(w^{\varepsilon}, w^{\varepsilon})$ as $\varepsilon \to 0$, we  consider the sesquilinear form $a_0$ in $H^1_{per}(0, 2\pi)$ given by 
	\begin{equation}\label{6}
		a_0(u,v) = \hat{g} \int_0^{2\pi} \left(q_0 \frac{du}{d \varphi} \frac{dv}{d\varphi} + uv\right) dx, \ \ \ \forall u,v \in H^1_{per}(0, 2\pi),    
	\end{equation}
	where $\hat{g} = \frac{1}{L} \int_0^L g(s)   ds$ is the average of the  function  $g$ on the interval $(0,L)$, and $q_0$ is the homogenized  coefficient defined in  \eqref{q_0}. We also will consider $L^2(0,2\pi)$ endowed with the norm induced by the inner product  $(\cdot, \cdot)_0$, given by 
	\begin{equation*} 
		(u,v)_0 = \hat{g} \int_0^1  uv \ dx, \ \ \ \forall u,v \in L^2(0,2\pi).  
	\end{equation*}

	\begin{remark}   We can solve  problems \eqref{1}, \eqref{2}, \eqref{14}, \eqref{21} and \eqref{23} applying Lax-Milgram Theorem to the elliptic form
		$$
		a_{Y^*}(\psi, \phi)= \int_{Y^*} \nabla_{y,z}\psi \cdot \nabla_{y,z}\phi  \ dy dz,  \  \  \forall \psi, \phi \in H^1_{per}(Y^*)/\mathbb{R},
		$$
		where we consider the norm    
		$$
		\Vert \psi \Vert = \left(  \int_{Y^*} \vert \nabla \psi \vert^2  \ dy dz\right)^{1/2} . 
		$$
	Here $H^1_{per}(Y^*)/\mathbb{R}$ represents  $H^1_{per}(Y^*)$ without the constant functions.
	\end{remark}
	
	\subsection{First-order corrector}
	We  will improve the convergence of $w^{\varepsilon}$ to a strong convergence in $H^1_{per}(R^\varepsilon, cos\theta)$ by using the first corrector.

	\begin{theorem}\label{teo1corretor} Let $f^\varepsilon \in L^2(R^{\varepsilon}, \cos \theta) $ satisfying  
		$$
		||| f^\varepsilon|||_{L^2(R^{\varepsilon}, \cos \theta)} \leq C, 
		$$    
		for some $C>0$ independent of  $\varepsilon$ and the  family of  functions  $\hat{f}^{\varepsilon} \in L^2(0,2\pi)$ defined by 
		\begin{equation*}
			\hat{f}^{\varepsilon}(\varphi) = \varepsilon^{-1} \int_0^{\varepsilon g (\varphi/\varepsilon)} f^\varepsilon(\varphi, \theta) \cos \theta d\theta \ \,\forall \,\varepsilon\in [0,1],
		\end{equation*}  and let  $w^{\varepsilon} \in H^1_{per}(R^{\varepsilon}, cos\theta)$ be the solution of problem \eqref{eqR}.
		
		If   $\hat{f}^{\varepsilon} \rightharpoonup \hat{f} $  in $ L^2(0,2\pi)$ and $w_0 \in H^2(0,2\pi) \cap C^1(0,2\pi)$ is the unique  solution of the homogenized equation \eqref{eqlimitemultiplasescalas} with 
		$ f_0 = \frac{1}{\hat{g}}\hat{f}$, then
		\begin{equation*}
			\displaystyle\lim_{\varepsilon \to 0} ||| w^{\varepsilon} - w_0 - \kappa^{\varepsilon} |||_{H^1(R^{\varepsilon}, \cos \theta)}=0, 
		\end{equation*}
		where $\kappa^{\varepsilon}$ is the first-order  corrector of $w^{\varepsilon}$ according to Definition \ref{def corretores} . 
	\end{theorem}

	\begin{proof}
		By variational formulation  of \eqref{eqR}, we have
		$
		a_{\varepsilon}(\psi, w^{\varepsilon})= (\psi, f^\varepsilon)_{\varepsilon}, \ \ \ \forall \psi \in H^1_{per}(R^{\varepsilon}, \cos \theta).
		$ Thus, observing that  $w_0 + \kappa^{\varepsilon} \in H^1_{per}(R^{\varepsilon}, \cos \theta)$, we obtain by the  symmetry  of $a_{\varepsilon}$	
	\begin{eqnarray}\label{35}
			||| w^{\varepsilon} - w_0 - \kappa^{\varepsilon}  |||_{H^1(R^{\varepsilon}, \cos \theta)}^2 &=&a_{\varepsilon} (w^{\varepsilon}-w_0-\kappa^{\varepsilon},w^{\varepsilon}-w_0-\kappa^{\varepsilon} ) \\
				&= &(w^{\varepsilon} -   w_0, f^\varepsilon)_{\varepsilon} -2( \kappa^{\varepsilon}, f^\varepsilon)_{\varepsilon}-(   w_0, f^\varepsilon)_{\varepsilon}+ a_{\varepsilon} (w_0+\kappa^{\varepsilon} ,w_0+\kappa^{\varepsilon} ). \nonumber
	\end{eqnarray}

		We now take the limit in each of the terms above. For the first term, we begin by noticing that 
		$$
		(w^{\varepsilon}-w_0, f^\varepsilon)_{\varepsilon} \leq ||| w^{\varepsilon}-w_0 |||_{L^2(R^{\varepsilon}, \cos \theta)}||| f^\varepsilon |||_{L^2(R^{\varepsilon}, \cos \theta)}.
		$$
		Using the change of variables $(\varphi, \theta) \to (\varphi, \theta/ \varepsilon) =(\varphi_1, \theta_1) $, we have 
		$$ 
		||| w^{\varepsilon}-w_0 |||^2_{L^2(R^{\varepsilon}, \cos \theta)} =  || u^{\varepsilon} - w_0 ||^2_{L^2(\Omega^{\varepsilon}, \cos (\varepsilon \theta_1))}.
		$$
		From Theorem \ref{main},  it follows that $||| w^{\varepsilon} - w_0    |||_{L^2(R^{\varepsilon}, \cos \theta)} \stackrel{\varepsilon \to 0}{\longrightarrow} 0 $.

		For the second term, by inequality \eqref{27}, we  obtain
		\begin{equation*} 
			(\kappa^{\varepsilon}, f^\varepsilon)_{\varepsilon} \leq  \varepsilon^{-1} ||  \kappa^{\varepsilon}||_{L^2(R^{\varepsilon}, \cos \theta)} ||  f^\varepsilon||_{L^2(R^{\varepsilon}, \cos \theta)} \leq \frac{\varepsilon C}{L^{1/2}} || X ||_{L^2(Y^*)} \left|\left|  \frac{d w_0}{d \varphi} \right|\right|_{L^{\infty}(0,2\pi)} \to 0, 
		\end{equation*}
		as $\varepsilon$ tends to $0$. 
		
		Now, since $\hat{f}^{\varepsilon} \rightharpoonup \hat{f}$  in $L^2(0,2\pi)$, we have for the third term that
		\begin{equation*} 
			\begin{aligned}
				(w_0, f^\varepsilon)_{\varepsilon} = & \ \varepsilon^{-1} \int_0^{2\pi} w_0 (\varphi) \int_0^{\varepsilon g(\varphi/ \varepsilon)} f^\varepsilon(\varphi, \theta) \cos \theta   \ d \theta  \ d \varphi \\
				= & \ \int_0^{2\pi} w_0 (\varphi) \hat{f}^{\varepsilon} (\varphi) \ d \varphi 
				\to \ \hat{g} \int_0^{2\pi} w_0 (\varphi) f_0(\varphi) \ d \varphi = (w_0, f_0)_0, \ \ \ \textrm{as} \ \varepsilon \to 0.
			\end{aligned}
		\end{equation*}

		Next, we show that 
		\begin{equation}\label{produtointerno}
		a_{\varepsilon}(w_0+\kappa^{\varepsilon}, w_0+\kappa^{\varepsilon}) \to a_0(w_0,w_0) \ \ \ \textrm{as} \  \  \varepsilon  \to 0.
		\end{equation}
	 First, we compute the limit of $ a_{\varepsilon}(w_0+\kappa^{\varepsilon}, w_0) $ as $\varepsilon \to 0.$ Note that
		\begin{equation*} 
			\begin{aligned}
				a_{\varepsilon}(w_0+\kappa^{\varepsilon}, w_0) = &\varepsilon^{-1} \int_{R^{\varepsilon}}  \left( \frac{\partial (w_0+\kappa^{\varepsilon}) }{\partial  \theta} \frac{\partial w_0}{\partial  \theta}    + \frac{1}{ \cos^2 ( \theta) }\frac{\partial (w_0+\kappa^{\varepsilon})}{\partial \varphi} \frac{\partial w_0 }{\partial \varphi}   \right) \cos \theta  d \varphi d \theta \\
				+& \varepsilon^{-1} \int_{R^{\varepsilon}}   (w_0+\kappa^{\varepsilon})w_0    \cos \theta  d \varphi d \theta \\
				= &\varepsilon^{-1} \int_{R^{\varepsilon}} \frac{1}{ \cos^2 ( \theta) } \frac{d^2 w_0}{d \varphi^2}  \left(  1 - \frac{\partial X}{\partial y }\right)  \cos \theta  d \varphi d \theta  + \varepsilon^{-1} \int_{R^{\varepsilon}} |w_0|^2  \cos \theta  d \varphi d \theta    \\  
				-& \varepsilon^{-1} \int_{R^{\varepsilon}} \left( \frac{1}{ \cos^2 ( \theta) } \varepsilon X    \frac{d w_0 }{d \varphi}\frac{d^2 w_0 }{d \varphi^2}   + \varepsilon X   w_0 \frac{d w_0 }{d \varphi} \right) \cos \theta  d \varphi d \theta. 
			\end{aligned}
		\end{equation*}
		Moreover, since
		\begin{equation*}
			\begin{aligned}
				&\varepsilon^{-1} \int_{R^{\varepsilon}} \frac{1}{ \cos^2 ( \theta) }  \left(\frac{d w_0}{d \varphi} \right)^2  \left(  1 - \frac{\partial X\left(\frac{\varphi}{\varepsilon},\frac{\theta}{\varepsilon}\right)}{\partial y}\right)  \cos \theta  d \varphi d \theta \\
				& = \int_0^{2 \pi}\int_0^{g(\varphi/\varepsilon)}\frac{1}{ \cos^2 (\varepsilon z) }  \left(\frac{d w_0}{d \varphi} \right)^2  \left(  1 - \frac{\partial X\left(\frac{\varphi}{\varepsilon},z\right)}{\partial y}\right)  \cos (\varepsilon z)  d z d \varphi,
			\end{aligned}
		\end{equation*}
		and $y \mapsto \int_0^{g(y)} \frac{1}{\cos (\varepsilon z)}\left(  1 - \frac{\partial X(y,z)}{\partial y}\right)  \ d z $ is a $L$-periodic function, we obtain from Proposition \ref{teochi}  that 
		\begin{equation}\label{42}
			\varepsilon^{-1} \int_{R^{\varepsilon}} \frac{1}{ \cos^2 ( \theta) }\left(\frac{d w_0}{d \varphi} \right)^2    \left(  1 - \frac{\partial X\left(\frac{\varphi}{\varepsilon},\frac{\theta}{\varepsilon}\right)}{\partial y}\right)  \cos \theta  d \varphi d \theta \to \hat{g}\int_0^{2 \pi} q_0 \left(\frac{d w_0}{d \varphi} \right)^2   d \varphi, 
		\end{equation}    
		as $\varepsilon \to 0$.  By Proposition \ref{teochi}, we also have   
		\begin{equation}\label{43}
			\varepsilon^{-1} \int_{R^{\varepsilon}} |w_0|^2  \cos \theta  d \varphi d \theta   \to \hat{g}\int_0^{2 \pi}  |w_0|^2  d \varphi \ \ \textrm{as} \  \ \varepsilon \to 0.
		\end{equation}   
		
			Since $w_0$ does not depend on  $\theta$ and
			\begin{equation*}
				\begin{aligned}
					\varepsilon^{-1} & \int_{R^{\varepsilon}} \left( \frac{1}{ \cos^2 ( \theta) } \varepsilon X    \frac{d w_0 }{d \varphi}\frac{d^2 w_0 }{d \varphi^2}   + \varepsilon X   w_0 \frac{d w_0 }{d \varphi} \right) \cos \theta  d \varphi d \theta  \\  
					& \leq   \Vert  X \Vert_{L^2(R^{\varepsilon}, cos \theta)} \left(  \left\Vert  \frac{d w_0 }{d \varphi} \right\Vert_{L^  \infty(R^\varepsilon)} \left\Vert \frac{d^2 w_0 }{d \varphi^2} \right\Vert_{L^2(R^{\varepsilon}, cos \theta)} +  \left\Vert w_0  \right\Vert_{L^ \infty(R^\varepsilon)}  \left\Vert \frac{d w_0 }{d \varphi}  \right\Vert_{L^  \infty(R^\varepsilon)}  \Vert 1 \Vert_{L^2(R^{\varepsilon}, cos \theta)} \right),\\ 
				\end{aligned}
			\end{equation*}

			it follows from \eqref{27} that 
			\begin{equation}\label{44}
				- \varepsilon^{-1} \int_{R^{\varepsilon}} \left( \frac{1}{ \cos^2 ( \theta) } \varepsilon X    \frac{d w_0 }{d \varphi}\frac{d^2 w_0 }{d \varphi^2}   + \varepsilon X   w_0 \frac{d w_0 }{d \varphi} \right) \cos \theta  d \varphi d \theta  \to  0,
			\end{equation}   
			
			as  $\varepsilon \to 0$.  Hence, we have from \eqref{42}, \eqref{43}, \eqref{44} and \eqref{6} that \eqref{produtointerno} follows. Finally, we still need to prove that 
					$$
					a_{\varepsilon}(w_0 + \kappa^{\varepsilon}, \kappa^{\varepsilon}) \to 0, \ \ \textrm{as} \  \ \varepsilon \to 0.
					$$
					
					Since
					$$
					\kappa^{\varepsilon} = - \varepsilon X \frac{d w_0}{d \varphi },  \   \  \frac{\partial \kappa^{\varepsilon} }{\partial \varphi}  = - \frac{\partial X}{\partial y} \frac{d w_0}{d \varphi } - \varepsilon X \frac{d^2 w_0}{d \varphi^2 } \   \  \textrm{and} \   \ \frac{\partial \kappa^{\varepsilon} }{\partial \theta}  = - \frac{\partial X}{\partial z} \frac{d w_0}{d \varphi },
					$$
					from straightforward calculations and inequalities  \eqref{27} and   \eqref{28}, we obtain:
					\begin{eqnarray*}
						a_{\varepsilon}(w_0 + \kappa^{\varepsilon}, \kappa^{\varepsilon}) &=& \varepsilon^{-1} \int_{R^{\varepsilon}} \left( \frac{\partial (w_0 + \kappa^{\varepsilon})}{\partial  \theta} \frac{\partial \kappa^{\varepsilon}}{\partial  \theta}    + \frac{1}{ \cos^2 ( \theta) }\frac{\partial (w_0 + \kappa^{\varepsilon})}{\partial \varphi} \frac{\partial \kappa^{\varepsilon} }{\partial \varphi}  + (w_0 + \kappa^{\varepsilon})\kappa^{\varepsilon}  \right) \cos \theta  d \varphi d \theta \\
						&= &\varepsilon^{-1} \int_{R^{\varepsilon}} \left( \left( \frac{\partial  \kappa^{\varepsilon}}{\partial  \theta}\right)^2 +  \frac{1}{ \cos^2 ( \theta) }  \left( \frac{\partial w_0   }{\partial \varphi} \frac{\partial \kappa^{\varepsilon} }{\partial \varphi}  +  \frac{\partial  \kappa^{\varepsilon} }{\partial \varphi} \frac{\partial \kappa^{\varepsilon}}{\partial \varphi} \right)
						+ (w_0 + \kappa^{\varepsilon})\kappa^{\varepsilon}  \right) \cos \theta  d \varphi d \theta \\
						&\leq& \frac{1}{L} \left\Vert \frac{\partial X}{\partial z} \right\Vert^2_{L^2(Y^*)} \left\Vert \frac{d w_0}{d \varphi } \right\Vert^2_{L^2(R^{\varepsilon},\cos \theta )} +\frac{1}{L} \left\Vert \frac{1}{ \cos^2 ( \theta) } \right\Vert_{L^{\infty}} \left\Vert \frac{\partial X}{\partial y} \right\Vert^2_{L^2(Y^*)} \left\Vert \frac{d w_0}{d \varphi } \right\Vert^2_{L^2(R^{\varepsilon},\cos \theta )}\\ 
						&+& \frac{1}{\sqrt{L\varepsilon}}  
						\left\Vert \frac{1}{ \cos^2 ( \theta) } \right\Vert_{L^{\infty}}    \left\Vert \frac{d^2 w_0}{d \varphi^2 }  \right\Vert_{L^2(R^{\varepsilon},\cos \theta )}    \left\Vert \frac{\partial X}{\partial y} \right\Vert_{L^2(Y^*)}  +\varepsilon  C_1  + \varepsilon^{1/2} C_2.
					\end{eqnarray*}
					
					Thus, we have
					$
					a_{\varepsilon}(w_0 + \kappa^{\varepsilon}, \kappa^{\varepsilon}) \to 0, \ \ \textrm{as} \  \ \varepsilon \to 0.
					$ Therefore, in accordance with \eqref{35}, we obtain 
					$$
					||| w^{\varepsilon} - w_0 - \kappa^{\varepsilon}  |||_{H^1(R^{\varepsilon}, \cos \theta)}^2 \stackrel{\varepsilon \to 0}{\longrightarrow} a_0(w_0,w_0) - (w_0, f_0)_0 =0,
					$$
					completing the proof.
				\end{proof}

				\subsection{ Second-order corrector}

				Now, we use Definition \ref{def corretores} to present an error estimate when the solution $w^{\varepsilon}$ of \eqref{eqR} is approximated   by its first-order truncation 
				\begin{equation}\label{46}
					\mathcal{W}_1^{\varepsilon}(\varphi, \theta)  :=w_0(\varphi)  + \kappa^{\varepsilon} =  w_0(\varphi) - \varepsilon X \left( \frac{\varphi}{\varepsilon},\frac{\theta}{\varepsilon} \right) \frac{d w_0}{d \varphi}(\varphi), \ \ (\varphi, \theta) \in R^{\varepsilon},
				\end{equation}
				with respect to the norm  $ ||| \cdot |||_{H^1(R^{\varepsilon}, cos\theta)}$. To refine this approximation, we also   consider the second-order  truncation in $R^{\varepsilon}$:
				\begin{equation}\label{47}
					\mathcal{W}_2^{\varepsilon}(\varphi, \theta)  := w_0(\varphi) + \mu^{\varepsilon} =  w_0(\varphi) - \varepsilon X \left( \frac{\varphi}{\varepsilon},\frac{\theta}{\varepsilon} \right) \frac{d w_0}{d \varphi}(\varphi) + \varepsilon^2 \Theta \left( \frac{\varphi}{\varepsilon},\frac{\theta}{\varepsilon} \right) \frac{d^2 w_0}{d \varphi^2}(\varphi),
				\end{equation}
				where $w_0$ denotes the homogenized  solution of \eqref{eqlimitemultiplasescalas},  $X$ and $\Theta$ are the  auxiliary solutions defined in  \eqref{14} and \eqref{23}, respectively.  These functions are  originally defined on the  \textit{basic cell} $Y^*$ and extended to the thin domain  $R^{\varepsilon}$ as explained in Remark \ref{5}.

				\begin{theorem}\label{teo erro} Let $R^{\varepsilon}$ be the thin domain defined in \eqref{R} and let $w^{\varepsilon}$ be the solution of  problem \eqref{eqR} with $f^\varepsilon(\varphi, \theta)= f^\varepsilon(\varphi)$, $f \in W^{2,\infty}(0,2 \pi )$ $L-$periodic.
					
					Then, if $\mathcal{W}_1^{\varepsilon}$ and $\mathcal{W}_2^{\varepsilon}$ are given by \eqref{46} and  \eqref{47} respectively, we obtain that  
					\begin{equation}\label{48}
						||| w^{\varepsilon} - \mathcal{W}_2^{\varepsilon}|||_{H^1(R^{\varepsilon}, \cos \theta)} \leq K_1 \sqrt{\varepsilon}.
					\end{equation}
					
					Consequently, we obtain the following rate for the first-order approximation
					\begin{equation}\label{49}
						||| w^{\varepsilon} - \mathcal{W}_1^{\varepsilon}|||_{H^1(R^{\varepsilon},\cos \theta)} \leq K_2 \sqrt{\varepsilon},
					\end{equation}
					
					where $K_1$ and $K_2$  are positive constants  independent of $\varepsilon >0$.
				\end{theorem}
				
				\begin{proof}
					First, we note that \eqref{49} is a direct consequence from \eqref{48} and \eqref{28}, since 
					

							$$\left|\left| \varepsilon^2 \Theta \frac{d^2w_0}{d \varphi^2}  \right|\right|^2_{L^2(R^{\varepsilon},\cos \theta)}  \leq \varepsilon^{4} \left|\left|  \frac{d^2w_0}{d \varphi^2}  \right|\right|^2_{L^{\infty}(0,2\pi)} \left|\left|  \Theta   \right|\right|^2_{L^2(R^{\varepsilon},\cos \theta)}  
							\leq \varepsilon^{5} \left|\left|  \frac{d^2w_0}{d \varphi^2}  \right|\right|^2_{L^{\infty}(0,2\pi)} \left|\left|  \Theta   \right|\right|^2_{L^2(Y^{*})}  \leq \varepsilon^{5} C_1,
					$$
					and   
					\begin{equation*}
						\begin{aligned}
							& \left|\left| \nabla_{(\varphi, \theta)}  \left(\varepsilon^2 \Theta \frac{d^2w_0}{d \varphi^2} \right) \right|\right|^2_{L^2(R^{\varepsilon},\cos \theta)}          
							\leq  
							    \varepsilon^{5}\left|\left|  \frac{d^3w_0}{d \varphi^3}  \right|\right|^2_{L^{\infty}(0,2\pi)}  \left|\left|  \Theta  \right|\right|^2_{L^2(Y^*)} \\
							&+ \varepsilon^{3}\left|\left|  \frac{d^2w_0}{d \varphi^2}  \right|\right|^2_{L^{\infty}(0,2\pi)} \left( \left|\left|  \frac{\partial  \Theta}{\partial y}\right|\right|^2_{L^2(Y^*)}  + \left|\left|  \frac{\partial  \Theta}{\partial z}\right|\right|^2_{L^2(Y^*)}  \right) \\
							&+  2 \varepsilon^{4} \left|\left|  \frac{d^2w_0}{d \varphi^2}  \right|\right|^2_{L^{\infty}(0,2\pi)} \left( \left|\left|  \frac{\partial  \Theta}{\partial y}\right|\right|^2_{L^2(Y^*)}   + \left|\left|  \frac{\partial  \Theta}{\partial z}\right|\right|^2_{L^2(Y^*)}  \right) 
							\left|\left|  \frac{d^3w_0}{d \varphi^3}  \right|\right|^2_{L^{\infty}(0,2\pi)}  \left|\left|  \Theta  \right|\right|^2_{L^2(Y^*)}  \leq C_2 \varepsilon^{3}.
						\end{aligned}
					\end{equation*}

					With this, we have
					\begin{equation*}
						\varepsilon^{-1} \left|\left| \varepsilon^2 \Theta \frac{d^2w_0}{d \varphi^2}  \right|\right|^2_{H^1(R^{\varepsilon},\cos \theta)} \leq  K_2 \varepsilon^{2}, 
					\end{equation*}
					for some  constant $K_2 >0$ independent of  $\varepsilon$.  
					
					Now we estimate the norm $|||  \phi^{ \varepsilon} |||^2_{H^1(R^{\varepsilon}, \cos \theta)} = a_{\varepsilon}( \phi^{ \varepsilon},  \phi^{ \varepsilon})$ of the function  $\phi^{ \varepsilon}$ given by:
					$$
					\phi^{ \varepsilon} = w^{ \varepsilon} -  \mathcal{W}_2^{\varepsilon}.
					$$
					In order to do it, we going to show that   the function  $\phi^{\varepsilon}$ satisfies the following boundary value  problem  where $\mathcal{L}$ is given by $(\ref{operador})$ and $N_1, N_2$ are as in $(\ref{eqR})$
					\begin{equation*}
						\left\{\begin{array}{l}
							- \mathcal{L}  \left(  \phi^{\varepsilon}\right) + \phi^{\varepsilon}  =\varepsilon F^{\varepsilon} + \mathcal{J}^{\varepsilon} \ \ \textit{in}\ R^{\varepsilon}; \\
							\frac{1}{\cos^2 \theta } \frac{\partial \phi^{\varepsilon}}{\partial \varphi}   N_1+  \frac{\partial \phi^{\varepsilon}}{\partial \theta}N_2=\varepsilon^2 \mathcal{H}^{\varepsilon}N_1 +  r^{\varepsilon}N_2 \ \ \textit{on} \ \ \partial R^{\varepsilon},\\
							\phi^{\varepsilon} (\cdot, \theta) \  \  \  \ \  \  \ 2\pi- \textrm{periodic}, 
						\end{array}\right.
					\end{equation*}
					where $ F^{\varepsilon}, \mathcal{J}^{\varepsilon}, \mathcal{H}^{\varepsilon} $ and $ r^{\varepsilon}$ will be computed below. With this,  we consider the  variational formulation of  problem, that is, find $\phi^{\varepsilon} \in H_{per}^1(R^{\varepsilon}, \cos \theta)$ such that 
					\begin{equation*}
					a_{\varepsilon}(\phi^{\varepsilon}, \psi)= \int_{R^{\varepsilon}}   F^{\varepsilon}\psi \cos \theta \ d \varphi d \theta + \frac{1}{\varepsilon}\int_{R^{\varepsilon}}   \mathcal{J}^{\varepsilon}\psi \cos \theta \ d \varphi d \theta 
+ \varepsilon \int_{\partial R^{\varepsilon}}   \mathcal{H}^{\varepsilon} N_1 \psi \cos \theta \ d S +  \frac{1}{\varepsilon} \int_{\partial R^{\varepsilon}}   r^{\varepsilon} N_1 \psi \cos \theta \ d S.
					\end{equation*} 
					
					Observe that the  function $\phi^{\varepsilon}$ must satisfy a uniform  a priori estimate on $\varepsilon$. By  taking  $\psi = \phi^{\varepsilon}$ in the equation above, we obtain:
					\begin{equation*}
						\begin{aligned}
							\vert \vert \vert  \phi^{\varepsilon} \vert \vert \vert^2_{H^1(R^{\varepsilon}, cos\theta)}=&  \vert a_{\varepsilon} (\phi^{\varepsilon},\phi^{\varepsilon})\vert
							 \leq \vert \vert \phi^{\varepsilon}  \vert \vert_{L^2(R^{\varepsilon}, \cos \theta)}  \vert \vert F^{\varepsilon}  \vert \vert_{L^2(R^{\varepsilon},\cos \theta)}  + \frac{1}{\varepsilon}\vert \vert \phi^{\varepsilon}  \vert \vert_{L^2(R^{\varepsilon},\cos \theta)}  \vert \vert \mathcal{J}^{\varepsilon}  \vert \vert_{L^2(R^{\varepsilon} , \cos \theta)} \\
							 +& \varepsilon \vert \vert \phi^{\varepsilon}  \vert \vert_{L^2(\partial R^{\varepsilon}, \cos \theta)}  \vert \vert \mathcal{H}^{\varepsilon} N_1 \vert \vert_{L^2(\partial R^{\varepsilon},\cos \theta)}  + \frac{1}{\varepsilon}\vert \vert \phi^{\varepsilon}  \vert \vert_{L^2(\partial R^{\varepsilon}, \cos \theta)}  \vert \vert r^{\varepsilon} N_2 \vert \vert_{L^2(\partial R^{\varepsilon},\cos \theta)}.
						\end{aligned}
					\end{equation*} 
					
					Therefore, in order to estimate $a_{\varepsilon} (\phi^{\varepsilon},\phi^{\varepsilon})$, it is necessary to obtain a precise estimate for  $F^{\varepsilon}$, $ \mathcal{J}^{\varepsilon}$, $\mathcal{H}^{\varepsilon}$ and $r^{\varepsilon}$.
					
						To this end, we will use the following identity satisfied by the linear differential operator $ \mathcal{L}$,
						$$ \mathcal{L} (uv)= v \mathcal{L} u + u \mathcal{L} v +  2  \frac{\partial u }{\partial \theta} \frac{\partial v }{\partial \theta}+  \frac{2}{\cos^2 \theta} \frac{\partial u}{\partial \varphi} \frac{\partial v}{\partial \varphi}. $$
						In particular if $ v(\varphi,\theta) = v (\varphi) $ we have
						
						\begin{equation}\label{L do produto}
							\mathcal{L} (u(\varphi,\theta)v(\varphi))=      v \mathcal{L} u + u \mathcal{L} v +    \frac{2}{\cos^2 \theta} \frac{\partial u}{\partial \varphi} \frac{\partial v}{\partial \varphi}.  
						\end{equation}
						Using the change of variables defined in \eqref{10} and the fact  that $X$ satisfies \eqref{14}, we obtain 
						
						\begin{equation*}
						\mathcal{L} X =  \frac{1}{ \varepsilon^2 \cos \theta} \frac{\partial }{\partial z} \left( \cos (\varepsilon z) \frac{\partial }{\partial z} \right) X + \frac{1}{\varepsilon^2\cos^2 \theta} \frac{\partial^2 }{\partial y^2}  X  
							= \frac{1}{ \varepsilon^2 }     
								\left( \frac{1}{ \cos^2 \theta} -1\right) \frac{\partial^2 X }{\partial y^2}  - \frac{1}{ \varepsilon}    \tan \theta    \frac{\partial X }{\partial z}.
						\end{equation*}
						
						It follows from \eqref{L do produto} that,
							\begin{equation*}
							\begin{aligned} \mathcal{L} \left(X \frac{dw_0}{d \varphi}\right) =&   \frac{dw_0}{d \varphi}  \mathcal{L} X +  X \mathcal{L} \left(  \frac{dw_0}{d \varphi}\right) + \frac{2}{\cos^2 \theta} \frac{\partial X}{\partial \varphi} \frac{d^2 w_0}{d \varphi^2} \\
								= & \frac{1}{ \varepsilon^2 }     
								\left( \frac{1}{ \cos^2 \theta} -1\right) \frac{dw_0}{d x}  \frac{\partial^2 X }{\partial y^2}+ \frac{1}{ \varepsilon} \left( -     \tan \theta    \frac{\partial X }{\partial z}\frac{dw_0}{d x} 
								+ \frac{2}{ \cos^2 \theta} \frac{\partial X}{\partial y} \frac{d^2 w_0}{d x^2} \right) \\
								 +& X  \frac{1}{ \cos^2  \theta}\frac{d^3 w_0}{d x^3}.
							\end{aligned}
						\end{equation*}
						
						Analogously, using the fact  that $\Theta$ satisfies \eqref{23}, we obtain that 
						
						\begin{equation*}
							\begin{aligned} \mathcal{L} \Theta    & = \frac{1}{ \varepsilon^2} \left(    \left( \frac{1}{ \cos^2 \theta} -1\right) \frac{\partial^2 \Theta }{\partial y^2} +             \frac{\partial^2 \Theta }{\partial z^2}        +   \frac{\partial^2 }{\partial y^2} \Theta   \right) - \frac{1}{ \varepsilon }   \tan \theta   \frac{\partial \Theta }{\partial z} \\
								&= \frac{1}{ \varepsilon^2 } \left(    
								\left( \frac{1}{ \cos^2 \theta} -1\right) \frac{\partial^2 \Theta }{\partial y^2}   - 1 + q + 2 \frac{\partial X}{\partial y }   \right) - \frac{1}{ \varepsilon }   \tan \theta   \frac{\partial \Theta }{\partial z},
							\end{aligned}
						\end{equation*}
						and
						\begin{equation*}
							\begin{aligned} \mathcal{L} \left(\Theta \frac{d^2w_0}{d \varphi^2} \right) 
								= & \frac{d^2w_0}{d \varphi^2}  \mathcal{L} \Theta  + \Theta   \frac{1}{ \cos^2  \theta}\frac{d^4 w_0}{d \varphi^4}     
								+ \frac{2}{\cos^2 \theta} \frac{\partial \Theta }{\partial \varphi} \frac{d^3 w_0}{d \varphi^3}\\
								= & \frac{1}{ \varepsilon^2 } \left(  \left( \frac{1}{ \cos^2 \theta} -1\right) \frac{d^2w_0}{d x^2}  \frac{\partial^2 \Theta}{\partial y^2} -\left(  1 - q - 2 \frac{\partial X}{\partial y }\right)\frac{d^2w_0}{d x^2}             \right)    \\ 
								 +& \frac{1}{\varepsilon } \left( \frac{2}{ \cos^2 \theta} \frac{\partial \Theta }{\partial y} \frac{d^3 w_0}{d x^3} -     \tan \theta   \frac{d^2w_0}{d x^2}  \frac{\partial \Theta  }{\partial z}    
								\right) + \Theta  \frac{1}{ \cos^2  \theta}\frac{d^4 w_0}{d x^4}.
							\end{aligned}
						\end{equation*}
						
						To simplify subsequent computations, we define
						$ A_{\theta}:=\left( \frac{1}{ \cos^2 \theta} -1\right),
						$
						with this, we have
						\begin{equation*}
							\begin{aligned} \mathcal{L}  \left( \mathcal{W}_2^{\varepsilon}(\varphi, \theta)  \right) =&  - \frac{1}{ \varepsilon  }          
								A_{\theta} \frac{dw_0}{d x}  \frac{\partial^2 X }{\partial y^2}           
								+ \tan \theta  \frac{dw_0}{d x}  \frac{\partial X }{\partial z}  +\frac{d^2 w_0}{d x^2} \left[  A_{\theta} \left( 1 - 2  \frac{\partial X}{\partial y}  +    \frac{\partial^2 \Theta}{\partial y^2}\right)
								+ q     \right]                \\ 
								+&\varepsilon  \left[   A_{\theta}\frac{d^3 w_0}{d x^3}   \left( 2 \frac{\partial \Theta }{\partial y}    
								-  X   \right) + \frac{d^3 w_0}{d x^3}   \left( 2 \frac{\partial \Theta }{\partial y}    
								-  X   \right) -     \tan \theta   \frac{d^2w_0}{d x^2}  \frac{\partial \Theta  }{\partial z}  \right]   \\ 
								 + &\varepsilon^2 \Theta  \frac{d^4 w_0}{d x^4} \left( 1+ A_{\theta}\right).
							\end{aligned}
						\end{equation*}
						Finally, 
						\begin{eqnarray*}
								 - \mathcal{L} (\phi^{ \varepsilon}) + \phi^{ \varepsilon} &=&   -\varepsilon  \left[     \frac{d^3 w_0}{d x^3}   \left( X - 2 \frac{\partial \Theta }{\partial y}    
								\right) - X \frac{d  w_0}{d x } + \varepsilon \left(  \Theta  \frac{d^4 w_0}{d x^4}        -   \Theta  \frac{d^2 w_0}{d x^2} \right)  \right]\\   
								&-& \frac{1}{ \varepsilon  }          
								A_{\theta} \frac{dw_0}{d x}  \frac{\partial^2 X }{\partial y^2}           
								+ \tan \theta   \frac{dw_0}{d x}  \frac{\partial X }{\partial z}     +\frac{d^2 w_0}{d x^2}    A_{\theta} \left( 1 - 2  \frac{\partial X}{\partial y}  +    \frac{\partial^2 \Theta}{\partial y^2}\right)\\ 
								&   +&\varepsilon  \left[    A_{\theta}\frac{d^3 w_0}{d x^3}   \left( 2 \frac{\partial \Theta }{\partial y}    
								-  X   \right)    -     \tan \theta   \frac{d^2w_0}{d x^2}  \frac{\partial \Theta  }{\partial z}   \right]   + \varepsilon^2  A_{\theta}  \left[  \Theta  \frac{d^4 w_0}{d x^4}           \right],\\ 
						\end{eqnarray*}
						where, we define 
						\begin{equation*} 
							F^{\varepsilon} = -  \left[   \frac{d^3 w_0}{dx^3} \left(  X   -  2  \frac{\partial \Theta}{\partial y}  \right) -  X \frac{d w_0}{dx}\right]
							+  \varepsilon \left(   \Theta \left(  \frac{d^4 w_0}{dx^4}  -   \frac{d^2 w_0}{dx^2}\right) \right),
						\end{equation*} 
						and
						\begin{equation}\label{Resto F} 
							\begin{aligned}
								\mathcal{J}^{\varepsilon} =& - \frac{1}{ \varepsilon  }          
								A_{\theta} \frac{dw_0}{d x}  \frac{\partial^2 X }{\partial y^2}           
								+ \tan \theta   \frac{dw_0}{d x}  \frac{\partial X }{\partial z}     +\frac{d^2 w_0}{d x^2}    A_{\theta} \left( 1 - 2  \frac{\partial X}{\partial y}  +    \frac{\partial^2 \Theta}{\partial y^2}\right)
								\\   
								&   +\varepsilon  \left[    A_{\theta}\frac{d^3 w_0}{d x^3}   \left( 2 \frac{\partial \Theta }{\partial y}    
								-  X   \right)    -     \tan \theta   \frac{d^2w_0}{d x^2}  \frac{\partial \Theta  }{\partial z}   \right]    + \varepsilon^2  A_{\theta}  \left[  \Theta  \frac{d^4 w_0}{d x^4}           \right].\\ 
							\end{aligned}
						\end{equation}

						We have then
						\begin{equation*}
							- \mathcal{L} (\phi^{ \varepsilon}) + \phi^{ \varepsilon} = -  \frac{1}{\cos \theta} \frac{\partial }{\partial \theta} \left( \cos \theta \frac{\partial }{\partial \theta} \right)\phi^{ \varepsilon} - \frac{1}{\cos^2 \theta} \frac{\partial^2 }{\partial \varphi^2}\phi^{ \varepsilon} + \phi = \varepsilon F^{\varepsilon}+  \mathcal{J}^{\varepsilon}.
						\end{equation*}

						On the boundary $\partial R^{\varepsilon}$, we define the operator
						\begin{equation*}
							\begin{aligned}
								\mathcal{B}(u) = & \frac{1}{\cos^2 \theta }  \frac{\partial u }{\partial \varphi} N_1 +  \frac{\partial u }{\partial \theta }  N_2  = \frac{\partial u }{\partial \varphi} N_1 +  \frac{\partial u }{\partial \theta }  N_2 +A_{\theta}\frac{\partial u }{\partial \varphi} N_1 \\
								= & \left(\frac{\partial   }{\partial x} + \frac{1}{\varepsilon   } \frac{\partial   }{\partial y} \right)     u \ N_1 +  \frac{1}{\varepsilon   } \frac{\partial  }{\partial z } u  \   N_2   + A_{\theta}   \left(\frac{\partial   }{\partial x} + \frac{1}{\varepsilon   }  \frac{\partial   }{\partial y} \right)u  \  N_1
								\\
								= & \frac{\partial u  }{\partial x} N_1 +  \frac{1}{\varepsilon   } \left(    \frac{\partial u  }{\partial y}  \ N_1 +    \frac{\partial u }{\partial z } u  \   N_2   \right)  + A_{\theta}   \left(\frac{\partial   }{\partial x} + \frac{1}{\varepsilon   }  \frac{\partial   }{\partial y} \right)u  \  N_1.
							\end{aligned}
						\end{equation*}

						Then, we have
						$$
						\mathcal{B}(X)= \frac{1}{\varepsilon   }  \left(   \frac{\partial X }{\partial y} N_1 +  \frac{\partial X }{\partial z }  N_2   + A_{\theta}   \frac{\partial X }{\partial y} N_1\right) =  \frac{1}{\varepsilon   }  \left(  N_1 +   A_{\theta}   \frac{\partial X }{\partial y} N_1 \right), 
						$$
						where we use boundary condition given in \eqref{14}. Therefore
						\begin{equation*}
							\begin{aligned}
								\mathcal{B} \left(X\frac{d w_0}{dx} \right)=& \left(\frac{\partial   }{\partial x} + \frac{1}{\varepsilon   } \frac{\partial   }{\partial y} \right)     \left(X\frac{d w_0}{dx} \right)N_1 +  \frac{1}{\varepsilon   } \frac{\partial  }{\partial z }  \left(X\frac{d w_0}{dx} \right) N_2  + A_{\theta}   \left(\frac{\partial   }{\partial x} + \frac{1}{\varepsilon   }  \frac{\partial   }{\partial y} \right)\left(X\frac{d w_0}{dx} \right) N_1 \\  
								=&  \frac{1}{\varepsilon} \left( \frac{\partial X }{\partial y} N_1 +  \frac{\partial X }{\partial z}     N_2   \right) \frac{d w_0}{dx}+ 
								X\frac{d^2 w_0}{dx^2}      N_1 +    A_{\theta}  \left(X\frac{d^2 w_0 }{dx^2}  + \frac{1}{\varepsilon} \frac{\partial X }{\partial y}\frac{d w_0}{dx}    \right)N_1\\  
								=&  \frac{1}{\varepsilon}  N_1  \frac{d w_0}{dx}+ 
								X\frac{d^2 w_0}{dx^2}      N_1 +    A_{\theta}  \left(X\frac{d^2 w_0 }{dx^2}  + \frac{1}{\varepsilon} \frac{\partial X }{\partial y}\frac{d w_0}{dx}    \right)N_1\\  
								=& \left[ \frac{1}{\varepsilon}      \frac{d w_0}{dx}  \left(1+  A_{\theta}\frac{\partial X }{\partial y}\right)
								+ X\frac{d^2 w_0}{dx^2}    (1   +    A_{\theta})      \right] N_1.
							\end{aligned}
						\end{equation*}
						
						Analogously, we obtain that 
						$$
						\mathcal{B}(\Theta)= \frac{1}{\varepsilon   }  \left(   \frac{\partial \Theta }{\partial y} N_1 +  \frac{\partial \Theta }{\partial z }  N_2   + A_{\theta}   \frac{\partial \Theta }{\partial y} N_1\right) =  \frac{1}{\varepsilon   }  \left(   X N_1    + A_{\theta}   \frac{\partial \Theta }{\partial y} N_1\right), 
						$$
						where we used the  boundary condition given in \eqref{23}. Thus

						\begin{equation*}
							\begin{aligned}
								\mathcal{B} \left(\Theta\frac{d^2w_0}{dx^2} \right)=& \left(\frac{\partial   }{\partial x} + \frac{1}{\varepsilon   } \frac{\partial   }{\partial y} \right)     \left(\Theta\frac{d^2w_0}{dx^2} \right)N_1 +  \frac{1}{\varepsilon   } \frac{\partial  }{\partial z }  \left(\Theta \frac{d^2 w_0}{dx^2} \right) N_2   A_{\theta}   \left(\frac{\partial   }{\partial x} + \frac{1}{\varepsilon   }  \frac{\partial   }{\partial y} \right)\left(\Theta \frac{d^2 w_0}{dx^2} \right) N_1 \\  
								=&  \left( \Theta \frac{d^3 w_0}{dx^3} + \frac{1}{\varepsilon} \frac{\partial \Theta }{\partial y}\frac{d^2 w_0}{dx^2}     \right) N_1 +  \frac{1}{\varepsilon} \frac{\partial \Theta }{\partial z}\frac{d^2 w_0}{dx^2}    N_2 +  A_{\theta}  \left(\Theta \frac{d^3 w_0 }{dx^3}  + \frac{1}{\varepsilon} \frac{\partial \Theta }{\partial y}\frac{d^2w_0}{dx^2}    \right)N_1\\
								=&  \frac{1}{\varepsilon} \left( \frac{\partial \Theta }{\partial y}  N_1 + \frac{\partial \Theta }{\partial z}   N_2 \right) \frac{d^2 w_0}{dx^2}  + \Theta \frac{d^3 w_0}{dx^3} N_1 +  A_{\theta}  \left(\Theta \frac{d^3 w_0 }{dx^3}  + \frac{1}{\varepsilon} \frac{\partial \Theta }{\partial y}\frac{d^2w_0}{dx^2}    \right)N_1\\
								=&  \frac{1}{\varepsilon} X    \frac{d^2 w_0}{dx^2} N_1 + \Theta \frac{d^3 w_0}{dx^3} N_1 +  A_{\theta}  \left(\Theta \frac{d^3 w_0 }{dx^3}  + \frac{1}{\varepsilon} \frac{\partial \Theta }{\partial y}\frac{d^2w_0}{dx^2}    \right)N_1\\
								=&  \left[ \frac{1}{\varepsilon}  \left( X  +  A_{\theta} \frac{\partial \Theta }{\partial y} \right)  \frac{d^2 w_0}{dx^2} +  \Theta \frac{d^3 w_0}{dx^3}  \left( 1 +    A_{\theta}\right) \right] N_1.
							\end{aligned}
						\end{equation*}
						
						It  follows that, 
						\begin{equation*}
							\begin{aligned}
								\mathcal{B} \left(\mathcal{W}_2^{\varepsilon} \right)= &  \mathcal{B} \left(  w_0 \right) - \varepsilon \mathcal{B} \left( X   \frac{d w_0}{d \varphi} \right) + \varepsilon^2  \mathcal{B} \left(\Theta   \frac{d^2 w_0}{d \varphi^2} \right)  \\
								= &      - \frac{d w_0}{dx} A_{ \theta} \left(  \frac{\partial X }{\partial y} -  1\right) N_1  +\varepsilon \left[       A_{\theta}  \left(     \frac{\partial \Theta }{\partial y} -  X \right)  \frac{d^2 w_0}{dx^2} \right] N_1 + \varepsilon^2          \Theta \frac{d^3 w_0}{dx^3}  \left( 1 +    A_{\theta}\right) N_1.
							\end{aligned}
						\end{equation*}

						Finally, 
						\begin{equation*}
							\begin{aligned}
								\mathcal{B} \left(\phi^{\varepsilon} \right)= &  \mathcal{B} \left(  w^{\varepsilon} \right) -  \mathcal{B} \left(\mathcal{W}_2^{\varepsilon} \right) =   -  \mathcal{B} \left(\mathcal{W}_2^{\varepsilon} \right),\\
							\end{aligned}
						\end{equation*} 
						since now, from equation \eqref{eqR}, we have $\mathcal{B} \left(  w^{\varepsilon} \right) = \frac{1}{\cos^2 \theta }  \frac{\partial  w^{ \varepsilon} }{\partial \varphi} N_1 +\frac{  \partial w^{ \varepsilon} }{\partial \theta }  N_2  =0$.
						
						We define
						\begin{equation}\label{52}
							\mathcal{H}^{\varepsilon}  = -      \Theta \frac{d^3 w_0}{dx^3},    
						\end{equation}
						and
						\begin{equation*}
							\begin{aligned}
								r^{\varepsilon} = A_{ \theta} \left[   \left(  \frac{\partial X }{\partial y} -  1\right) \frac{d w_0}{dx}  -\varepsilon          \left(     \frac{\partial \Theta }{\partial y} -  X \right)  \frac{d^2 w_0}{dx^2}     - \varepsilon^2            \Theta \frac{d^3 w_0}{dx^3} \right].   
							\end{aligned}
						\end{equation*}    
						Thus, the function  $\phi^{\varepsilon}$ satisfies the following boundary value  problem  below
							\begin{equation}\label{50}
								\left\{\begin{array}{l}
									- \mathcal{L} \left( \phi^{\varepsilon} \right) + \phi^{\varepsilon}   =\varepsilon F^{\varepsilon} + \mathcal{J}^{\varepsilon} \ \ \textit{in}\ R^{\varepsilon}; \\
									\mathcal{B} \left( \phi^{\varepsilon} \right)  =\varepsilon^2 \mathcal{H}^{\varepsilon}N_1 +  r^{\varepsilon}N_1 \ \ \textit{on} \ \ \partial R^{\varepsilon},\\
									\phi^{\varepsilon}(\cdot, \theta)  \  \  \  \ \  \  \  2\pi-\textrm{periodic}. 
								\end{array}\right.
							\end{equation} 
							We consider now the  variational formulation of  problem \eqref{50}: Find $\phi^{\varepsilon} \in H_{per}^1(R^{\varepsilon}, \cos \theta)$ such that 
							\begin{equation}\label{53}
								\begin{aligned}
									a_{\varepsilon}(\phi^{\varepsilon}, \psi)=& \int_{R^{\varepsilon}}   F^{\varepsilon}\psi \cos \theta \ d \varphi d \theta + \frac{1}{\varepsilon}\int_{R^{\varepsilon}}   \mathcal{J}^{\varepsilon}\psi \cos \theta \ d \varphi d \theta \\
									+& \varepsilon \int_{\partial R^{\varepsilon}}   \mathcal{H}^{\varepsilon} N_1 \psi \cos \theta \ d S +  \frac{1}{\varepsilon} \int_{\partial R^{\varepsilon}}   r^{\varepsilon} N_1 \psi \cos \theta \ d S.
								\end{aligned}
							\end{equation} 
							
							Observe that the  function $\phi^{\varepsilon}$ must satisfy an uniform  a priori estimate on $\varepsilon$. In fact, if we take  $\psi = \phi^{\varepsilon}$ in \eqref{53}, we obtain:
							\begin{equation}\label{54}
								\begin{aligned}
									\vert \vert \vert  \phi^{\varepsilon} \vert \vert \vert^2_{H^1(R^{\varepsilon}, cos\theta)}=&  \vert a_{\varepsilon} (\phi^{\varepsilon},\phi^{\varepsilon})\vert \leq \vert \vert \phi^{\varepsilon}  \vert \vert_{L^2(R^{\varepsilon}, \cos \theta)}  \vert \vert F^{\varepsilon}  \vert \vert_{L^2(R^{\varepsilon},\cos \theta)}  + \frac{1}{\varepsilon}\vert \vert \phi^{\varepsilon}  \vert \vert_{L^2(R^{\varepsilon},\cos \theta)}  \vert \vert \mathcal{J}^{\varepsilon}  \vert \vert_{L^2(R^{\varepsilon} , \cos \theta)} \\
									 + &\varepsilon \vert \vert \phi^{\varepsilon}  \vert \vert_{L^2(\partial R^{\varepsilon}, \cos \theta)}  \vert \vert \mathcal{H}^{\varepsilon} N_1 \vert \vert_{L^2(\partial R^{\varepsilon},\cos \theta)}  + \frac{1}{\varepsilon}\vert \vert \phi^{\varepsilon}  \vert \vert_{L^2(\partial R^{\varepsilon}, \cos \theta)}  \vert \vert r^{\varepsilon} N_1 \vert \vert_{L^2(\partial R^{\varepsilon},\cos \theta)}.
								\end{aligned}
							\end{equation} 
							
							We need to obtain precise inequalities for  $F^{\varepsilon}$, $ \mathcal{J}^{\varepsilon}$, $\mathcal{H}^{\varepsilon}$ and $r^{\varepsilon}$ in order to estimate $a_{\varepsilon} (\phi^{\varepsilon},\phi^{\varepsilon})$. It is clear form their definitions that theses estimates will follow from those  obtained for $w_0$, $X$ and  $\Theta$. Since $f$  is a smooth  function, by the  classical  regularity results, the  solution  $w_0$ of the homogenized problem is sufficiently  regular to ensure that its  derivatives up to fourth order belong to  $L^{\infty}(0, 2 \pi)$. Note that similar statements also hold for $X$ and  $\Theta \in H^1(Y^*)$.

							\begin{equation*} 
								\begin{aligned}
									\Vert  F^{\varepsilon} \Vert_{L^2(R^{\varepsilon}, \cos \theta)} \leq &  \left\Vert      \frac{d^3 w_0}{dx^3} \left(  X   -  2  \frac{\partial \Theta}{\partial y}  \right) -  X \frac{d w_0}{dx}  \right\Vert_{L^2(R^{\varepsilon}, \cos \theta)} +  \varepsilon     \left\Vert \Theta \left(  \frac{d^4 w_0}{dx^4}  -   \frac{d^2 w_0}{dx^2}\right)  \right\Vert_{L^2(R^{\varepsilon}, \cos \theta)}  \\
									\leq &  \left\Vert      \frac{d^3 w_0}{dx^3}\right\Vert_{L^{\infty}}    \left(  \sqrt{\frac{\varepsilon}{L}} \left\Vert X \right\Vert_{L^2(Y^*)}  +  2  \sqrt{\frac{\varepsilon}{L}} \left\Vert \frac{\partial \Theta}{\partial y} \right\Vert_{L^2(Y^*)} \right)  + \sqrt{\frac{\varepsilon}{L}}  \left\Vert  X \right\Vert_{ L^2(Y^*)}\left\Vert \frac{d w_0}{dx}  \right\Vert_{L^{\infty}}\\
									+ & \varepsilon    \sqrt{\frac{\varepsilon}{L}}  \left\Vert \Theta \right\Vert_{L^2(Y^*)} \left(   \left\Vert \frac{d^4 w_0}{dx^4} \right\Vert_{L^{\infty}} +    \left\Vert \frac{d^2 w_0}{dx^2}  \right\Vert_{L^{\infty}} \right)  
									\leq   C_1 \sqrt{ \varepsilon} + C_2 \sqrt{ \varepsilon} + C_3    \varepsilon  \sqrt{ \varepsilon}\leq  K_0 \sqrt{ \varepsilon}.
								\end{aligned}
							\end{equation*}
							
							Consequently, due to the periodicity of $X$ and from Proposition \ref{prop estimativas}, we have that there exists $K_0$ independent of $\varepsilon$ such that  
							\begin{equation}\label{55}
								\Vert F^{\varepsilon}\Vert_{L^2(R^{\varepsilon}, \cos \theta)} \leq K_0 \sqrt{ \varepsilon}.
							\end{equation}
							Furthermore, from the definition of $\mathcal{J}^{\varepsilon} $   in \eqref{Resto F},  and Proposition \ref{prop estimativas}, we can obtain an estimative for $\mathcal{J}^{\varepsilon} $. To do this end, we will estimate each of the terms that compose $\mathcal{J}^{\varepsilon} $. 
%
%
%
For the first term in $(\ref{Resto F})$, we have
							\begin{equation*} 
								\begin{aligned}
									\left\Vert  A_{\theta} \frac{dw_0}{d x}  \frac{\partial^2 X }{\partial y^2} \right\Vert_{L^2(R^{\varepsilon}, \cos \theta)}  &\leq \sqrt{\frac{  \varepsilon}{L}} \left\Vert A_{\theta} \right\Vert_{L^{\infty}(R^{\varepsilon})}  \left\Vert  \frac{dw_0}{d x} \right\Vert_{L^{\infty}(R^{\varepsilon})}   \left\Vert    \frac{\partial^2 X }{\partial y^2} \right\Vert_{L^2(Y^{*})}. 
								\end{aligned}
							\end{equation*}
							We now compute 
							\begin{equation*} 
								\left\Vert A_{\theta} \right\Vert_{L^{\infty}(R^{\varepsilon})}    =\sup_{( \varphi,  \theta)  \in R^{\varepsilon} } \left(  \frac{ 1}{\cos^2 \theta}   -1 \right) =  \left(  \frac{ 1}{\cos^2 (\varepsilon g_1)}   -1 \right) =  A_{\theta}\left(\varepsilon g_1 \right),
							\end{equation*} 
							where  $ g_1$ is defined in $(\ref{g1})$. So, 
							\begin{equation*} 
								\begin{aligned}
									\left\Vert  A_{\theta} \frac{dw_0}{d x}  \frac{\partial^2 X }{\partial y^2} \right\Vert_{L^2(R^{\varepsilon}, \cos \theta)}  &\leq \sqrt{\frac{  \varepsilon}{L}} C_1 A_{\theta}\left(\varepsilon g_1 \right).
								\end{aligned}
							\end{equation*}
							
							Considering the second term, we obtain
							
							\begin{equation*} 
								\begin{aligned}
									\left\Vert   \tan \theta   \frac{dw_0}{d x}  \frac{\partial X }{\partial z}     \right\Vert_{L^2(R^{\varepsilon}, \cos \theta)}  &\leq \sqrt{\frac{  \varepsilon}{L}} \left\Vert \tan \theta   \right\Vert_{L^{\infty}(R^{\varepsilon})}  \left\Vert  \frac{dw_0}{d x} \right\Vert_{L^{\infty}(R^{\varepsilon})}   \left\Vert    \frac{\partial X }{\partial z} \right\Vert_{L^2(Y^{*})},
								\end{aligned}
							\end{equation*}
							
						and $\left\Vert \tan \theta \right\Vert_{L^{\infty}(R^{\varepsilon})}   =\sup_{( \varphi,  \theta)  \in R^{\varepsilon} }   \tan \theta  =     \tan   (\varepsilon g_1). $  Therefore,
							\begin{equation*} 
								\begin{aligned}
									\left\Vert   \tan \theta   \frac{dw_0}{d x}  \frac{\partial X }{\partial z}     \right\Vert_{L^2(R^{\varepsilon}, \cos \theta)}  &\leq \sqrt{\frac{  \varepsilon}{L}} C_2 \tan   (\varepsilon g_1).
								\end{aligned}
							\end{equation*}
							
							For the remaining terms, we have
							\begin{eqnarray*} 
									\left\Vert      \frac{d^2 w_0}{d x^2}    A_{\theta} \left( 1 - 2  \frac{\partial X}{\partial y}  +    \frac{\partial^2 \Theta}{\partial y^2}\right) \right\Vert_{L^2(R^{\varepsilon}, \cos \theta)}   &\leq & \left\Vert  \frac{d^2w_0}{d x^2} \right\Vert_{L^{\infty}(R^{\varepsilon})} \left\Vert A_{\theta} \right\Vert_{L^{\infty}(R^{\varepsilon})} \left(  1+ 2 \sqrt{\frac{  \varepsilon}{L}}   \left\Vert    \frac{\partial  X }{\partial y} \right\Vert_{L^2(Y^{*})}\right)
									\\ 
									& +& \left\Vert  \frac{d^2w_0}{d x^2} \right\Vert_{L^{\infty}(R^{\varepsilon})} \left\Vert A_{\theta} \right\Vert_{L^{\infty}(R^{\varepsilon})}  \sqrt{\frac{  \varepsilon}{L}}   \left\Vert    \frac{\partial^2 \Theta }{\partial y^2} \right\Vert_{L^2(Y^{*})}\\
									 & \leq &\sqrt{\frac{  \varepsilon}{L}} C_3 A_{\theta}(\varepsilon g_1),
							\end{eqnarray*}
							and
							\begin{equation*} 
								\begin{aligned}
									\left\Vert     A_{\theta}\frac{d^3 w_0}{d x^3}   \left( 2 \frac{\partial \Theta }{\partial y}    
									-  X   \right)  \right\Vert_{L^2(R^{\varepsilon}, \cos \theta)}  
									&\leq  \left\Vert A_{\theta}   \frac{d^3w_0}{d x^3} \right\Vert_{L^{\infty}(R^{\varepsilon})} \sqrt{\frac{  \varepsilon}{L}} \left( 2  \left\Vert    \frac{\partial  \Theta }{\partial y} \right\Vert_{L^2(Y^{*})} 
									-    \left\Vert   X \right\Vert_{L^2(Y^{*})} \right) \\
									&\leq \sqrt{\frac{  \varepsilon}{L}} C_4 A_{\theta}(\varepsilon g_1).
								\end{aligned}
							\end{equation*}
							We also have
							\begin{equation*} 
								\begin{aligned}
									\left\Vert     \tan \theta   \frac{d^2w_0}{d x^2}  \frac{\partial \Theta  }{\partial z} \right\Vert_{L^2(R^{\varepsilon}, \cos \theta)}  
									&\leq   \sqrt{\frac{  \varepsilon}{L}} 
									\left\Vert  \tan \theta \right\Vert_{L^{\infty}(R^{\varepsilon})}  \left\Vert  \frac{d^2w_0}{d x^2} \right\Vert_{L^{\infty}(R^{\varepsilon})}  \left\Vert    \frac{\partial \Theta  }{\partial z} \right\Vert_{L^2(Y^{*})}   \\&\leq \sqrt{\frac{  \varepsilon}{L}} C_5 \tan(\varepsilon g_1).
								\end{aligned}
							\end{equation*}
							and
							\begin{equation*} 
								\begin{aligned}
									\left\Vert      A_{\theta}     \Theta  \frac{d^4 w_0}{d x^4} \right\Vert_{L^2(R^{\varepsilon}, \cos \theta)}  
									&\leq   \sqrt{\frac{  \varepsilon}{L}} 
									\left\Vert  A_{\theta} \right\Vert_{L^{\infty}(R^{\varepsilon})}  \left\Vert  \frac{d^4w_0}{d x^4} \right\Vert_{L^{\infty}(R^{\varepsilon})}  \left\Vert   \Theta \right\Vert_{L^2(Y^{*})} \\&\leq \sqrt{\frac{  \varepsilon}{L}} C_6 A_{\theta}(\varepsilon g_1). 
								\end{aligned}
							\end{equation*}
							
							Thus, 
							\begin{equation} \label{primeira estimatiava R}
								\begin{aligned}
									\Vert \mathcal{J}^{\varepsilon}  \Vert_{L^2(R^{\varepsilon}, \cos \theta)}  &\leq\frac{ 1}{\varepsilon} \sqrt{\frac{  \varepsilon}{L}} C_1 A_{\theta}\left(\varepsilon g_1 \right)+ \sqrt{\frac{  \varepsilon}{L}} C_2 \tan   (\varepsilon g_1)+ \sqrt{\frac{  \varepsilon}{L}} C_3 A_{\theta}(\varepsilon g_1) \\ &+ \varepsilon  \left(  \sqrt{\frac{  \varepsilon}{L}} C_4 A_{\theta}(\varepsilon g_1) + \sqrt{\frac{  \varepsilon}{L}} C_5 \tan(\varepsilon g_1) \right) +  \varepsilon^2 \sqrt{\frac{  \varepsilon}{L}} C_6 A_{\theta}(\varepsilon g_1) \\
									& \leq c_1  A_{\theta}(\varepsilon g_1) \sqrt{\varepsilon} \left(1 + \frac{1}{ \varepsilon} +\varepsilon + \varepsilon^2 \right) +  c_2   \tan(\varepsilon g_1)\sqrt{\varepsilon} (1 +  \varepsilon ).
								\end{aligned}
							\end{equation}
						 Applying  Taylor series expansions, we obtain:
							\begin{equation*} 
								A_{\theta}\left(\varepsilon g_1 \right) =  \left(\varepsilon g_1 \right)^2 + \frac{2}{3}\left(\varepsilon g_1 \right)^4+ \cdots
							\end{equation*}
							and
							\begin{equation*} 
								\tan   (\varepsilon g_1)  =  \varepsilon g_1 + \frac{1}{3}\left(\varepsilon g_1 \right)^3+ \cdots
							\end{equation*}
							
							Replacing  these in \eqref{primeira estimatiava R},   we have:
							
							\begin{equation} \label{estimatiava R}
								\begin{aligned}
									\Vert \mathcal{J}^{\varepsilon}  \Vert_{L^2(R^{\varepsilon}, \cos \theta)}  
									& \leq c_1  \left(   \left(\varepsilon g_1 \right)^2 + \frac{2}{3}\left(\varepsilon g_1 \right)^4+ \cdots\right)  \sqrt{\varepsilon} \left(1 + \frac{1}{ \varepsilon} +\varepsilon + \varepsilon^2 \right) \\ &+   c_2 \left( \varepsilon g_1 + \frac{1}{3}\left(\varepsilon g_1 \right)^3+ \cdots\right) \sqrt{\varepsilon} (1 +  \varepsilon ) \\
									& \leq C_1'\varepsilon^2 \sqrt{\varepsilon} \left(1 + \frac{1}{ \varepsilon} +\varepsilon + \varepsilon^2 \right) + C_2'\varepsilon \sqrt{\varepsilon} (1 +  \varepsilon )  \leq K_1 \varepsilon \sqrt{\varepsilon}.
								\end{aligned}
							\end{equation}

							Let us observe that $K_0$ and $K_1$ depend on the period $L$ of the  norms of $X$, $\Theta$ and $\partial_y \Theta$ in $L^2(Y^*)$, as well of the norms of $\frac{dw_0}{dx}$, $\frac{d^2w_0}{dx^2}$, $\frac{d^3 w_0}{dx^3}$ and $\frac{d^4w_0}{dx^4}$ in $L^{\infty}(0, 2\pi)$.

							Now, let us  denote the oscillatory part of $\partial R^{\varepsilon}$ by  $\partial_0 R^{\varepsilon}= \lbrace (\varphi_1, \varepsilon g(\varphi_1/\varepsilon)), 0<\varphi_1< 2\pi \rbrace$, the fixed part by $\partial_f R^{\varepsilon}= \lbrace (\varphi_1, 0), 0<\varphi_1< 2\pi \rbrace$ and the lateral part of $\partial R^{\varepsilon}$ as  $\partial_l R^{\varepsilon}= \lbrace (0, \theta_1 ), 0<\theta_1< \varepsilon g(0) \rbrace \cup \lbrace (2 \pi, \theta_1 ), 0<\theta_1< \varepsilon g(1/\varepsilon) \rbrace $. From  \eqref{52} we have:

							\begin{equation*}
								\begin{aligned}
									\vert \vert   \mathcal{H}^{\varepsilon} N_1 \vert \vert^2_{L^2(\partial R^{\varepsilon}, \cos \theta) }=&  \int_{\partial R^{\varepsilon}}\left\vert  \Theta \left( \frac{\varphi_1}{\varepsilon}, \frac{\theta_1}{\varepsilon} \right) \frac{d^3 w_0}{dx^3}(\varphi_1) N_1(\varphi_1, \theta_1)\right\vert^2  \cos \theta dS\\
									 \leq &\left\vert \left\vert \frac{d^3 w_0}{dx^3}   \right\vert \right\vert_{L^{\infty}(0, 2\pi)}^2  \left(  \int_{\partial_0 R^{\varepsilon}}\left\vert  \Theta   \right\vert^2 dS+ \int_{\partial_f R^{\varepsilon}}\left\vert  \Theta   \right\vert^2 dS \right) \\
									\leq &K_2 \left( \sum_{k=1}^{1/\varepsilon L} \varepsilon \int_{0}^{L} \vert \Theta(y,g(y))\vert^2 dy + \sum_{k=1}^{1/\varepsilon L} \varepsilon \int_{0}^{L} \vert \Theta(y,0)\vert^2 dy \right)  \\
									 \leq& \frac{K_2}{L} \vert \vert  \Theta \vert \vert^2_{L^2(\partial Y^*)}. 
								\end{aligned}
							\end{equation*}

							Note that $K_2=  \vert \vert \frac{d^3 w_0}{dx^3} \vert \vert_{L^{\infty}(0,2\pi)}  $ is independent of $\varepsilon$, and we have used the periodicity of $\Theta$ to get $\int_{\partial l R^{\varepsilon}} \vert \Theta(\frac{\varphi}{\varepsilon}, \frac{\theta}{\varepsilon}) \vert^2 dS =0 $. Consequently there exists $\widetilde{K_2}>0$ independent of $\varepsilon$ such that 
							\begin{equation}\label{56}
								\vert\vert \mathcal{H}^{\varepsilon} N_1  \vert\vert_{L^2(\partial R^{\varepsilon}, cos\theta)}  \leq \widetilde{K_2}.    
							\end{equation}
						Furthermore 
						\begin{equation*}
							\begin{aligned}
								\left\Vert  r^{\varepsilon}  \right\Vert_{L^2(\partial R^{\varepsilon}, \cos \theta)}  \leq &  \left\Vert A_{\theta} \right\Vert_{L^{\infty}(0, 2\pi)}        \left( \left\Vert \frac{\partial X }{\partial y} \right\Vert_{L^2(\partial R^{\varepsilon},  \cos \theta))}  +  1\right)    \left\Vert \frac{d w_0}{dx} \right\Vert_{L^{\infty}(0, 2\pi)}   \\
								 +& \varepsilon    \left\Vert A_{\theta} \right\Vert_{L^{\infty}(0, 2\pi)}          \left(  \left\Vert   \frac{\partial \Theta }{\partial y} \right\Vert_{L^2(\partial R^{\varepsilon},  \cos \theta))}  +  \left\Vert X \right\Vert_{L^2(\partial R^{\varepsilon},  \cos \theta))} \right)  \left\Vert \frac{d^2 w_0}{dx^2} \right\Vert_{L^{\infty}(0, 2\pi)}     \\
							    +& \varepsilon^2  \left\Vert A_{\theta} \right\Vert_{L^{\infty}(0, 2\pi)}                 \left\Vert \Theta   \right\Vert_{L^2(\partial R^{\varepsilon},  \cos \theta))}  \left\Vert \frac{d^3 w_0}{dx^3} \right\Vert_{L^{\infty}(0, 2\pi)}         \\
								\leq &    C A_{\theta}   \left( \varepsilon g_1 \right) \sqrt{\varepsilon} \left( 1 +\varepsilon   + \varepsilon^2     \right) 
								\leq   C \left(   \left(\varepsilon g_1 \right)^2 + \frac{2}{3}\left(\varepsilon g_1 \right)^4+ \cdots\right)  \sqrt{\varepsilon} \left( 1 + \varepsilon   + \varepsilon^2     \right) \\
							  \leq &  \widetilde{K_3} \varepsilon^{5/2}. 
							\end{aligned}
						\end{equation*}
					 Consequently 
						\begin{equation}\label{est r}
							\left\Vert  r^{\varepsilon} N_1 \right\Vert_{L^2(\partial R^{\varepsilon}, \cos \theta)}  \leq   \widetilde{K_3} \varepsilon^{5/2}.
						\end{equation}
						Therefore,
						\begin{equation*} 
							\begin{aligned}
								\vert \vert \vert  \phi^{\varepsilon} \vert \vert \vert^2_{H^1(R^{\varepsilon}, cos\theta)}&=  \vert a_{\varepsilon} (\phi^{\varepsilon},\phi^{\varepsilon})\vert\\
								& \leq \vert \vert \phi^{\varepsilon}  \vert \vert_{L^2(R^{\varepsilon}, \cos \theta)}  \vert \vert F^{\varepsilon}  \vert \vert_{L^2(R^{\varepsilon},\cos \theta)}  + \frac{1}{\varepsilon}\vert \vert \phi^{\varepsilon}  \vert \vert_{L^2(R^{\varepsilon},\cos \theta)}  \vert \vert \mathcal{J}^{\varepsilon}  \vert \vert_{L^2(R^{\varepsilon} , \cos \theta)} \\
								& + \varepsilon \vert \vert \phi^{\varepsilon}  \vert \vert_{L^2(\partial R^{\varepsilon}, \cos \theta)}  \vert \vert \mathcal{H}^{\varepsilon} N_1 \vert \vert_{L^2(\partial R^{\varepsilon},\cos \theta)}  + \frac{1}{\varepsilon}\vert \vert \phi^{\varepsilon}  \vert \vert_{L^2(\partial R^{\varepsilon}, \cos \theta)}  \vert \vert r^{\varepsilon} N_1 \vert \vert_{L^2(\partial R^{\varepsilon},\cos \theta)}\\ 
								& \leq \vert \vert \phi^{\varepsilon}  \vert \vert_{L^2(R^{\varepsilon}, \cos \theta)}  K_0 \sqrt{\varepsilon}  + \frac{1}{\varepsilon}\vert \vert \phi^{\varepsilon}  \vert \vert_{L^2(R^{\varepsilon},\cos \theta)}      K_1 \varepsilon \sqrt{\varepsilon}  \\
								& + \varepsilon \vert \vert \phi^{\varepsilon}  \vert \vert_{L^2(\partial R^{\varepsilon}, \cos \theta)}  \vert \vert \mathcal{H}^{\varepsilon} N_1 \vert \vert_{L^2(\partial R^{\varepsilon},\cos \theta)}  + \frac{1}{\varepsilon}\vert \vert \phi^{\varepsilon}  \vert \vert_{L^2(\partial R^{\varepsilon}, \cos \theta)}  \vert \vert r^{\varepsilon} N_1 \vert \vert_{L^2(\partial R^{\varepsilon},\cos \theta)}.
							\end{aligned}
						\end{equation*}

						Now we have all the ingredients to estimate $a_{\varepsilon}(\phi^{\varepsilon},\phi^{\varepsilon} )$. Due to \eqref{55},  \eqref{estimatiava R}, \eqref{est r} and \eqref{56} we obtain from \eqref{54} that 
						\begin{equation*} 
							\vert\vert\vert  \phi^{\varepsilon}  \vert\vert \vert_{H^1( R^{\varepsilon}, \cos \theta)}^2 \leq  \sqrt{\varepsilon}  \left( K_0 + K_1\right) \vert\vert  \phi^{\varepsilon} \vert \vert_{L^2( R^{\varepsilon}, \cos \theta)} +\left( \varepsilon \widetilde{K_2}   + \varepsilon^{3/2} \widetilde{K_3}      \right)\vert\vert  \phi^{\varepsilon} \vert \vert_{L^2( \partial R^{\varepsilon}, \cos \theta)}.    
						\end{equation*}
						
						Hence, the desired  result follows from the following fact,   proof of which can be found in \cite{CP80}: if $\psi \in H^1(R^{\varepsilon}, cos\theta)$, then there exists a constant $C$ independent of $\varepsilon$ such that
						$$
						\vert\vert  \psi \vert \vert_{L^2( \partial R^{\varepsilon}, cos\theta)}   \leq C \varepsilon^{-1/2}\vert\vert  \psi \vert \vert_{H^1( R^{\varepsilon}, cos\theta)}.
						$$
				Thus 
						\begin{equation*}
							\begin{aligned}
								\vert\vert\vert  \phi^{\varepsilon}  \vert\vert \vert_{H^1( R^{\varepsilon}, \cos \theta)}^2 & \leq  \sqrt{\varepsilon}  \left( K_0 + K_1\right) \vert\vert  \phi^{\varepsilon} \vert \vert_{L^2( R^{\varepsilon}, \cos \theta)} +\left( \varepsilon \widetilde{K_2}   + \varepsilon^{3/2} \widetilde{K_3}      \right) C \varepsilon^{-1/2} \vert\vert  \phi^{\varepsilon} \vert \vert_{H^1(   R^{\varepsilon}, \cos \theta)} \\
								& \leq \sqrt{\varepsilon} K  \vert\vert  \phi^{\varepsilon} \vert \vert_{H^1(   R^{\varepsilon}, \cos \theta)}.
							\end{aligned}
						\end{equation*}
						
					\end{proof}
					
	\section{Concluding Remarks}

    Having established the weak convergence in Theorem~\ref{main},  the next natural step is to analyze the convergence of solutions for a semilinear problem considering a nonlinearity $f(u)$ in equation \ref{probu}. In addition, extensions to parabolic and hyperbolic problems could also be explored. Moreover, studying the  spectral properties of the Laplace-Beltrami operator in the highly oscillating domains in the Sphere could provide insight into the set of equilibrium solutions of the associated evolution equations. Finally, one could  consider the analogous oscillating problems in general manifolds.

	\textbf{Acknowledgements:} The third author was supported by Coordenação de Aperfeiçoamento de Pessoal de Nível Superior (CAPES) and partially supported by Fundação de Amparo à Pesquisa do Estado de São Paulo - FAPESP 202/14075-6.

\end{document}